\RequirePackage{rotating}

\makeatletter
\def\@cons#1#2{\begingroup\let\@elt\relax\xdef#1{\ifx#1\relax\else#1\fi\@elt #2}\endgroup}
\makeatother
\documentclass{svjour3}

\usepackage{amssymb,amsmath}
\usepackage{longtable}
\usepackage{rotating}
\usepackage{multirow}
\usepackage{algorithm}
\usepackage{rotating} 
\usepackage{algorithmicx}
\usepackage{pdflscape}
\usepackage{booktabs}
\usepackage{hyperref}
\usepackage{commath}
\usepackage[capitalise]{cleveref}
\crefname{figure}{fig.}{figs.}%
\usepackage[toc,page]{appendix}
\usepackage{dsfont}

\usepackage{algpseudocode}
\usepackage{mathtools}
\usepackage{graphicx}
\usepackage{tikz}
\usepackage{pgfplots}
\usepgfplotslibrary{groupplots}
\pgfplotsset{compat=newest}
\usetikzlibrary{arrows}
\usepgfplotslibrary{groupplots}
\usetikzlibrary{decorations}
\usetikzlibrary{decorations.markings}
\usetikzlibrary{shapes}
\usetikzlibrary{intersections}
\usetikzlibrary{patterns}

\usepackage{relsize} 

\makeatletter \let\cl@chapter\relax \makeatother

\makeatletter
\newcommand{\labelx}[1]{
	\relax
	\ifmmode
	\label{#1}
	\else
	\ifnum\pdfstrcmp{\@currenvir}{document}=0
	\label{#1}
	\else
	\label[\@currenvir]{#1}
	\fi
	\fi
}
\makeatother

\pgfplotsset{compat=1.16}

\newcommand{\nosemic}{\renewcommand{\@endalgocfline}{\relax}}
\newcommand{\dosemic}{\renewcommand{\@endalgocfline}{\algocf@endline}}
\let\oldnl\nl
\newcommand{\nonl}{\renewcommand{\nl}{\let\nl\oldnl}}

\newcommand{\direct}{\texttt{DIRECT}}
\newcommand{\directgl}{\texttt{DIRECT-GL}}
\newcommand{\halrect}{\texttt{HALRECT}}
\newcommand{\birect}{\texttt{BIRECT}}

\newcommand{\adc}{\texttt{ADC}}

\newcommand{\indexsett}{\mathbb{I}}
\newcommand{\directlib}{\texttt{DIRECTGOLib v1.1}}
\newcommand{\directgo}{\texttt{DIRECTGO v1.1.0}}

\let\oldnl\nl

\newcommand{\algrule}[1][.2pt]{\par\vskip.5\baselineskip\hrule height #1\par\vskip.5\baselineskip}

\DeclareMathOperator*{\argmax}{arg\,max}

\definecolor{onyx}{rgb}{0.06, 0.06, 0.06}
\definecolor{sandstorm}{rgb}{0.93, 0.84, 0.25}
\definecolor{princetonorange}{rgb}{1.0, 0.56, 0.0}
\definecolor{sienna}{rgb}{0.53, 0.18, 0.09}
\definecolor{psychedelicpurple}{rgb}{0.87, 0.0, 1.0}

\pgfmathdeclarefunction{fpumod}{2}{%
	\pgfmathfloatdivide{#1}{#2}%
	\pgfmathfloatint{\pgfmathresult}%
	\pgfmathfloatmultiply{\pgfmathresult}{#2}%
	\pgfmathfloatsubtract{#1}{\pgfmathresult}%
	\pgfmathfloatifapproxequalrel{\pgfmathresult}{#2}{\def\pgfmathresult{5}}{}%
}

\algrenewcommand{\algorithmiccomment}[1]{\hfill// #1}
\algrenewcommand\algorithmicrequire{\textbf{Input:}}
\algrenewcommand\algorithmicensure{\textbf{Output:}}
\algnewcommand\algorithmicforeach{\textbf{for each}}
\algdef{S}[FOR]{ForEach}[1]{\algorithmicforeach\ #1\ \algorithmicdo}

\definecolor{redwood}{rgb}{0.67, 0.31, 0.32}

\emergencystretch=\maxdimen
\usetikzlibrary{external}
\begin{document}
	\title{Lipschitz-inspired \halrect{} Algorithm for Derivative-free Global Optimization}
		
	\author{Linas Stripinis \and Remigijus Paulavi\v{c}ius}

	\institute{L. Stripinis, R. Paulavi\v{c}ius \at
			Vilnius University, Institute of Data Science and Digital Technologies, Akademijos 4, LT-08663 Vilnius, Lithuania \\
			\email{linas.stripinis@mif.vu.lt} \and R. Paulavi\v{c}ius \at \email{remigijus.paulavicius@mif.vu.lt}}
		
	\date{Received: date / Accepted: date}

	\maketitle

\begin{abstract}
This article considers a box-constrained global optimization problem for Lipschitz-continuous functions with an unknown Lipschitz constant.
Motivated by the famous \direct{} (DIviding RECTangles), a new \halrect{} (HALving RECTangles) algorithm is introduced.
A new deterministic approach combines halving (bisection) with a new multi-point sampling scheme in contrast to trisection and midpoint sampling used in most existing \direct-type algorithms.
A new partitioning and sampling scheme uses more comprehensive information on the objective function.
Four different strategies for selecting potentially optimal hyper-rectangles are introduced to exploit the objective function's information effectively.
The original algorithm \halrect{} and other introduced \halrect{} variations (twelve in total) are tested and compared with the other twelve recently introduced \direct-type algorithms on $96$ box-constrained benchmark functions from \directlib, and $96$ perturbed their versions.
Extensive experimental results are advantageous compared to state-of-the-art \direct-type global optimization.
New \halrect{} approaches offer high robustness across problems of different degrees of complexity, varying from simple -- uni-modal and low dimensional to complex -- multi-modal and higher dimensionality.
\end{abstract}

\section{Introduction}\label{intro}

Generally, global optimization approaches can be divided into two main classes: deterministic and stochastic \cite{Horst1995:book,Sergeyev2018}.
Deterministic algorithms theoretically guarantee that at least one global optimum can be found \cite{Floudas1999book}, while stochastic algorithms find the solution in the probability sense~\cite{Liberti2005}.
Various optimization problems in science and engineering (e.g., machine learning models~\cite{Bishop2006}, Boeing design~\cite{Booker1998}, etc.) are black-box, i.e., the analytic information about the objective and constraints functions is unavailable.
Therefore, the development of derivative-free optimization has been forced by the need to optimize various and often increasingly complex problems in practice.


In this paper we consider a box-constrained potentially black-box global optimization problem
\begin{equation}\label{eq:opt-problem}
	\begin{aligned}
		& \min_{\mathbf{x}\in D} && f(\mathbf{x}),
	\end{aligned}
\end{equation}
where $f:\mathbb{R}^n \rightarrow \mathbb{R}$ is a real-valued Lipschitz-continuous function, i.e., there exists a positive constant $0 < L < \infty$, such that
\begin{equation}\label{eq:LipschitzCondition}
	\abs{f(\mathbf{x}) - f(\mathbf{y})}  \le  L\| \mathbf{x} - \mathbf{y} \|, \quad \forall \mathbf{x}, \mathbf{y} \in D,
\end{equation}
and the feasible region is an $n$-dimensional hyper-rectangle
$ D = [ \mathbf{a},  \mathbf{b}] = \{ \mathbf{x} \in \mathbb{R}^n: a_j \leq x_j \leq b_j, j = 1, \dots, n\} $.
In a black-box optimization case, the objective function $f$ is unknown and any information can only be obtained by evaluating the function at feasible points.

The \direct{} algorithm, developed by Jones et al.~\cite{Jones1993}, is a well-known and widely used sampling-based solution technique for derivative-free global optimization with limitations in the box.
An algorithm is an extension of classical Lipschitz optimization (e.g.,~\cite{Paulavicius2006,Paulavicius2007,Pinter1996book,Piyavskii1967,Sergeyev2011,Shubert1972}), where the need to know the Lipschitz constant is eliminated.
The \direct{} algorithm has also been successfully extended to solve problems with various constraints.
Authors in~\cite{Paulavicius2016:ol} proposed an approach to tackle linearly constrained problems.
Other authors~\cite{Costa2017,Finkel2004,Jones2001,Liu2017,pillo2016,pillo2010,Stripinis2018b} introduced \direct{}-type algorithm for generally constrained or even the for problems with hidden constraints~\cite{Gablonsky2001:phd,Na2017,Stripinis2021}.
In addition, \direct{}-type algorithms appear more often in the parallel environment~\cite{He2008,Stripinis2020,Stripinis2021c}.

A decade-old comprehensive numerical benchmarking~\cite{Rios2013} showed the encouraging performance of \direct-type algorithms among other derivative-free global optimization methods.
Our recent extensive study~\cite{Stripinis2021c} revealed that new and potentially better \direct-type algorithms are available today.
In~\cite{Stripinis2021b}, we also demonstrated that even better \direct-type algorithms could be obtained by combining various already known candidate selection and partitioning techniques, leading to even more efficient \direct-type algorithms.
Therefore, continuous design and development of efficient \direct-type algorithms is important and motivated by practice needs.

Unfortunately, \direct-type algorithms are not without their drawbacks. 
Among them, two well-known ones are~\cite{Gablonsky2001,Jones2001,Jones2021,Paulavicius2016:jogo,Paulavicius2019:eswa,Sergeyev2008:book}: i) delayed discovery of the globally optimal solution, especially for multi-modal and greater dimensionality problems, and ii) slow fine-tuning of the solution to high accuracy.
This limits \direct{} applicability mainly to lower-dimensionality global optimization problems~\cite{Jones2021}.
The first drawback is possibly determined by the original sampling scheme based on one center point per hyper-rectangle.
If the hyper-rectangle containing the global solution has a bad objective value at the midpoint, it is undesirable for the selection, and his further subdivision is delayed.

To address this in~\cite{Paulavicius2016:jogo,Sergeyev2006}, the authors introduced two different diagonal sampling schemes using two points per hyper-rectangle.
In this way, new algorithms, \birect~\cite{Paulavicius2016:jogo} and \adc~\cite{Sergeyev2006}, intuitively reduce the chance of this situation occurring.
It would require evaluating two bad points in the hyper-rectangle containing the global optimum.
In~\cite{Jones2001}, the author observed that to reduce the curse of dimensionality, the division of hyper-rectangles along only one longest side instead of all has a very positive impact.
Moreover, various two-phase-based approaches (see, e.g., \cite{Paulavicius2014:jogo,Sergeyev2006}) and hybridized \direct-type methods (see, e.g. \cite{Holmstrom2010,Jones2001,Paulavicius2019:eswa,Stripinis2018b,Liuzzi2010:coaa,Liuzzi2016}) were proposed to address both these shortcomings.

This paper introduces a new \halrect{} (HALving RECTangles) algorithm based on a new multi-point sampling scheme efficiently combined with halving (bisection).
Each hyper-rectangle is represented by considering up to $2n + 1$ sampling points and halved using bisection instead of just one sampled midpoint and trisection traditionally used in most \direct-type algorithms.
Therefore, more comprehensive information about the objective function over each hyper-rectangle is captured, especially for higher-dimensionality problems, as more sampled points are considered in selecting potentially optimal hyper-rectangles.

The rest of the paper is organized as follows.
\Cref{rewiev} reviews relevant existing~\direct-type modifications and summarizes the most common selection schemes and partitioning strategies used in state-of-the-art \direct{} algorithms.
A description of the new \halrect{} algorithm and all its new variations is given in \Cref{sec:halrect}.
The extensive numerical investigation of twelve \halrect{} variations and comparison with twelve recently introduced \direct-type algorithms~\cite{Stripinis2021b} using $96$ box-constrained global optimization test problems and their perturbed versions from \directlib~\cite{DIRECTGOLibv11} is provided in \Cref{sec:results}.
Finally, we conclude the paper in \Cref{sec:conclusions}.

\section{Related literature review}\label{rewiev}

This section reviews some of the most relevant \direct-type modifications.
We begin with a recap of the original algorithm.
Reviewing other \direct-type algorithms, we mainly focus on the proposed candidate selection, sampling, and partitioning schemes.

\subsection{Original \direct{} algorithm}\label{sec:direct}

The original \direct{} algorithm is designed for box-constrained global optimization problems.
Initially, the algorithm normalizes the feasible region $D = [\mathbf{a}, \mathbf{b}]$ to a unit hyper-rectangle $ \bar{D} = [0, 1]^n$ and only refers to the original space $D$ when evaluating the objective function $f$.
Therefore, throughout this paper, when it says that the value of the objective function is evaluated at $f(\mathbf{c})$, where the midpoint $\mathbf{c} \in \bar{D}$, it is understood that the corresponding midpoint of the original domain ($\mathbf{x} \in D$) is used, i.e.,
\begin{equation}
	\label{eq:space_original}
	f(\mathbf{c}) = f(\mathbf{x}), \text{where } x_j = (b_j - a_j) c_j + a_j, j=1,\dots,n.
\end{equation}
In each iteration, certain hyper-rectangles are identified and selected as ``potentially optimal hyper-rectangles'' (POH) for further investigation.
\direct{} samples and evaluates the objective function at the midpoint of each POH and subdivides them (into smaller hyper-rectangles) using the trisection strategy.
The selection, sampling, and subdivision procedures continue until some predefined limits have not been reached.
\Cref{fig:divide} illustrates this process, showing the initialization and the first two subsequent iterations of \direct{} for the two-variable \textit{Bukin6} test problem.

\begin{figure}[ht]
	\resizebox{\textwidth}{!}{
		\begin{tikzpicture}
			\begin{axis}[
				width=0.6\textwidth,height=0.6\textwidth,
				xlabel = {$c_1$},
				ylabel = {$c_2$},
				enlargelimits=0.05,
				title={Initialization},
				legend style={draw=none},
				legend columns=1,
				legend style={at={(0.8,-0.2)},font=\LARGE},
				ylabel style={yshift=-0.1cm},
				xlabel style={yshift=-0.01cm},
				ytick distance=1/6,
				xtick distance=1/6,
				every axis/.append style={font=\LARGE},
				yticklabels={$0$, $0$,$\frac{1}{6}$, $\frac{1}{3}$, $\frac{1}{2}$, $\frac{2}{3}$, $\frac{5}{6}$, $1$},
				xticklabels={$0$, $0$,$\frac{1}{6}$, $\frac{1}{3}$, $\frac{1}{2}$, $\frac{2}{3}$, $\frac{5}{6}$, $1$},
				]
				\addlegendimage{only marks,mark=*,color=black}
				\addlegendentry{Sampling point}
				\addplot[thick,patch,mesh,draw,black,patch type=rectangle,line width=0.3mm] coordinates {(0,0) (1,0) (1,1) (0,1)} ;
				\draw [black, thick, mark size=0.05pt, fill=blue!40,opacity=0.4,line width=0.3mm] (axis cs:0,0) rectangle (axis cs:1,1);
				\addplot[only marks,mark=*, mark size=2pt,blue] coordinates {(1/2, 1/2)} node[yshift=-12pt] {\small $\mathbf{c}^1$} node[yshift=12pt] {\small $50.05$};
			\end{axis}
		\end{tikzpicture}
		\hspace{-0.25cm}
		\begin{tikzpicture}
			\begin{axis}[
				width=0.6\textwidth,height=0.6\textwidth,
				xlabel = {$c_1$},
				enlargelimits=0.05,
				title={Iteration $1$},
				legend style={draw=none},
				legend columns=1,
				legend style={at={(0.8,-0.2)},font=\LARGE},
				ylabel style={yshift=-0.1cm},
				xlabel style={yshift=-0.01cm},
				ytick distance=1/6,
				xtick distance=1/6,
				every axis/.append style={font=\LARGE},
				yticklabels={$0$, $0$,$\frac{1}{6}$, $\frac{1}{3}$, $\frac{1}{2}$, $\frac{2}{3}$, $\frac{5}{6}$, $1$},
				xticklabels={$0$, $0$,$\frac{1}{6}$, $\frac{1}{3}$, $\frac{1}{2}$, $\frac{2}{3}$, $\frac{5}{6}$, $1$},
				]
				\addlegendimage{area legend,blue!40,,fill=blue!40,opacity=0.4,line width=0.3mm}
				\addlegendentry{Selected POH}
				\addplot[thick,patch,mesh,draw,black,patch type=rectangle,line width=0.3mm] coordinates {(0,0) (1,0) (1,1) (0,1)} ;
				\addplot[thick,patch,mesh,draw,black,patch type=rectangle,line width=0.3mm] coordinates {(1/3, 0) (1/3, 1) (2/3, 1) (2/3, 0)};
				\addplot[thick,patch,mesh,draw,black,patch type=rectangle,line width=0.3mm] coordinates {(1/3, 1/3) (1/3, 2/3) (2/3, 2/3) (2/3, 1/3)};
				\draw [black, thick, mark size=0.1pt, fill=blue!40,opacity=0.4,line width=0.3mm] (axis cs:2/3,0) rectangle (axis cs:1,1);
				\addplot[only marks,mark=o, mark size=2pt,black] coordinates {(1/2, 1/2)} node[yshift=-12pt] {\small $\mathbf{c}^1$} node[yshift=12pt] {\small $50.05$};
				\addplot[only marks,mark=o, mark size=2pt,black] coordinates {(1/6, 1/2)} node[yshift=-12pt] {\small $\mathbf{c}^2$} node[yshift=12pt] {\small $116.68$};
				\addplot[only marks,mark=*, mark size=2pt,blue] coordinates {(5/6, 1/2)} node[yshift=-12pt] {\small $\mathbf{c}^3$} node[yshift=12pt] {\small $16.78$};
				\addplot[only marks,mark=o, mark size=2pt,black] coordinates {(1/2, 1/6)} node[yshift=-12pt] {\small $\mathbf{c}^4$} node[yshift=12pt] {\small $150.05$};
				\addplot[only marks,mark=o, mark size=2pt,black] coordinates {(1/2, 5/6)} node[yshift=-12pt] {\small $\mathbf{c}^5$} node[yshift=12pt] {\small $132.33$};
			\end{axis}
		\end{tikzpicture}
		\hspace{-0.25cm}
		\begin{tikzpicture}
			\begin{axis}[
				width=0.6\textwidth,height=0.6\textwidth,
				xlabel = {$c_1$},
				enlargelimits=0.05,
				title={Iteration $2$},
				legend style={draw=none},
				legend columns=1,
				legend style={at={(0.85,-0.2)},font=\LARGE},
				ylabel style={yshift=-0.1cm},
				xlabel style={yshift=-0.01cm},
				ytick distance=1/6,
				xtick distance=1/6,
				every axis/.append style={font=\LARGE},
				yticklabels={$0$, $0$,$\frac{1}{6}$, $\frac{1}{3}$, $\frac{1}{2}$, $\frac{2}{3}$, $\frac{5}{6}$, $1$},
				xticklabels={$0$, $0$,$\frac{1}{6}$, $\frac{1}{3}$, $\frac{1}{2}$, $\frac{2}{3}$, $\frac{5}{6}$, $1$},
				]
				\addlegendimage{area legend,black,fill=white,opacity=0.5}
				\addlegendentry{Unselected region}
				\addplot[thick,patch,mesh,draw,black,patch type=rectangle,line width=0.3mm] coordinates {(0, 0) (1, 0) (1, 1) (0, 1)} ;
				\addplot[thick,patch,mesh,draw,black,patch type=rectangle,line width=0.3mm] coordinates {(1/3, 0) (1/3, 1) (2/3, 1) (2/3, 0)};
				\addplot[thick,patch,mesh,draw,black,patch type=rectangle,line width=0.3mm] coordinates {(1/3, 1/3) (1/3, 2/3) (1, 2/3) (1, 1/3)};
				\draw [black, thick, mark size=0.1pt, fill=blue!40,opacity=0.4,line width=0.3mm] (axis cs:0, 0) rectangle (axis cs:1/3, 1);
				\draw [black, thick, mark size=0.1pt, fill=blue!40,opacity=0.4,line width=0.3mm] (axis cs:2/3, 1/3) rectangle (axis cs:1, 2/3);
				\addplot[only marks,mark=o, mark size=2pt,black] coordinates {(1/2, 1/2)} node[yshift=-12pt] {\small $\mathbf{c}^1$} node[yshift=12pt] {\small $50.05$};
				\addplot[only marks,mark=*, mark size=2pt,blue] coordinates {(1/6, 1/2)} node[yshift=-12pt] {\small $\mathbf{c}^2$} node[yshift=12pt] {\small $116.68$};
				\addplot[only marks,mark=*, mark size=2pt,blue] coordinates {(5/6, 1/2)} node[yshift=-12pt] {\small $\mathbf{c}^3$} node[yshift=12pt] {\small $16.78$};
				\addplot[only marks,mark=o, mark size=2pt,black] coordinates {(1/2, 1/6)} node[yshift=-12pt] {\small $\mathbf{c}^4$} node[yshift=12pt] {\small $150.05$};
				\addplot[only marks,mark=o, mark size=2pt,black] coordinates {(1/2, 5/6)} node[yshift=-12pt] {\small $\mathbf{c}^5$} node[yshift=12pt] {\small $132.33$};
				\addplot[only marks,mark=o, mark size=2pt,black] coordinates {(5/6, 1/6)} node[yshift=-12pt] {\small $\mathbf{c}^6$} node[yshift=12pt] {\small $142.51$};
				\addplot[only marks,mark=o, mark size=2pt,black] coordinates {(5/6, 5/6)} node[yshift=-12pt] {\small $\mathbf{c}^7$} node[yshift=12pt] {\small $140.55$};
			\end{axis}
	\end{tikzpicture}}
	\caption{Two-dimensional illustration of selection, central sampling, and trisection used in the original \direct{} algorithm~\cite{Jones1993} solving the \textit{Bukin6} test problem.}
	\label{fig:divide}
\end{figure}
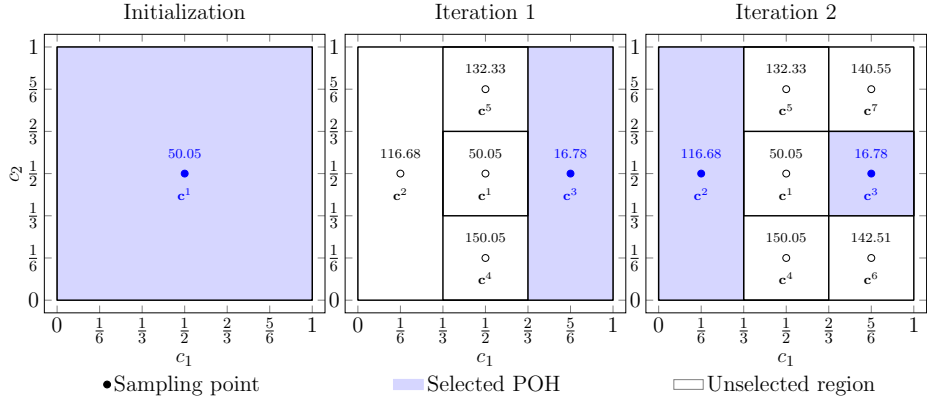

Regardless of the dimension, the first evaluation of the objective function is performed at the midpoint ($\mathbf{c}^1$).
Then, the \direct{} algorithm identifies and selects the POHs.
At initialization, the selection is trivial since only one hyper-rectangle $(\bar{D}^1)$ is available (see the left panel in \Cref{fig:divide}).
After selection, \direct{} samples new midpoints at positions
\begin{equation}
	\mathbf{c}^1 \pm \frac{1}{3}d^{\rm max} \mathbf{e}_j, j \in M,
\end{equation}
where $d^{\rm max}$ is equal to the maximum side length, $M$ is a set of dimensions with the maximum side length, and $\mathbf{e}_j$ is the $j$th unit vector.
The algorithm uses $n$-dimensional trisection, with the property that the objective function is evaluated at each hyper-rectangle only once --- at a midpoint.
The midpoint of the initial hyper-rectangle becomes the midpoint of the new smaller ``middle'' one.
Suppose the selected hyper-rectangle has more than one dimension with the maximum side length (as is the case for the initial hyper-rectangle).
In that case, \direct{} starts the trisection from the dimension with the lowest $w_j$ value
\begin{equation}
	w_j = \min \{ f(\mathbf{c}^1 + \frac{1}{3}d^{\rm max}\mathbf{e}_j),  f(\mathbf{c}^1 - \frac{1}{3}d^{\rm max}\mathbf{e}_j) \},  j \in M,
\end{equation}
and continues to the highest~\cite{Jones2021,Jones1993}.
In this way, the lower function values are placed in larger hyper-rectangles (see the middle panel in~\Cref{fig:divide}).
If all side lengths are equal, $2n + 1$ new smaller non-overlapping hyper-rectangles of $n$ distinct sizes are created.

Unlike initialization, in subsequent iterations, the selection of POHs is not trivial, as we have more than one candidate (see the middle and right panels in \Cref{fig:divide}).
Therefore, the selection procedure needs to be formalized.
Let the current partition in iteration $ k $ be defined as:
\begin{equation}\label{eq:partition_set}
	\mathcal{P}_k = \{ \bar{D}^i_k : i \in \indexsett_k \},
\end{equation}
where
\begin{equation}\label{eq:rectangle}
	\bar{D}^i_k = [\mathbf{a}_k^i, \mathbf{b}_k^i] = \{ \mathbf{x} \in \bar{D}: 0 \leq a_{k_j}^i \leq x_j \leq b_{k_j}^i \leq 1, j = 1,\dots, n, \forall i \in \indexsett_k \},
\end{equation}
and $ \indexsett_k $ is the index set that identifies the current partition $ \mathcal{P}_k $.
The next partition, $\mathcal{P}_{k+1}$, is obtained by subdividing selected POHs from the current partition $ \mathcal{P}_k $.
At the first iteration ($k=0$), there is always only one candidate, $\mathcal{P}_0 = \{ \bar{D}^1_0 \}$, which is automatically potentially optimal.
The formal requirement of potential optimality in subsequent iterations is stated in \Cref{def:potOptRect}.

\begin{definition}{(Original selection)}
	\label{def:potOptRect}
	Let $ \mathbf{c}^i $ denote the midpoint, $f(\mathbf{c}^i)$ objective function value $f(\mathbf{c}^i)$ obtained at the midpoint, and $ \delta^i_k $ be a measure (equivalently, sometimes called distance or size) of hyper-rectangle $ \bar{D}^i_k$.
	Let $ \varepsilon > 0 $ be a positive constant and $f^{\min}$ be the best currently found objective function value.
	A hyper-rectangle $ \bar{D}^h_k, h \in \indexsett_k $ is said to be potentially optimal if there exists some rate-of-change (Lipschitz) constant $ \tilde{L} > 0$ such that
	\begin{eqnarray}
		f(\mathbf{c}^h) - \tilde{L}\delta^h_k & \leq & f(\mathbf{c}^i) - \tilde{L}\delta^i_k, \quad \forall i \in \indexsett_k, \label{eqn:potOptRect1} \\
		f(\mathbf{c}^h) - \tilde{L}\delta^h_k & \leq & f^{\min} - \varepsilon \abs{f^{\min}}, \label{eqn:potOptRect2}
	\end{eqnarray}
	and the measure of the hyper-rectangle $ \bar{D}^i_k$ is
	\begin{equation}
		\label{eq:distance}
		\delta^i_k = \frac{1}{2} \left\| {\mathbf{b}}_k^i - {\mathbf{a}}_k^i \right\|.
	\end{equation}
\end{definition}

The hyper-rectangle $ \bar{D}^h_k $ is potentially optimal if the lower Lipschitz bound for the objective function computed on the left-hand side of \eqref{eqn:potOptRect1} is the lowest with some positive constant $\tilde{L}$ in the current partition $ \mathcal{P}_k $.
In~\eqref{eqn:potOptRect2}, the parameter $\varepsilon$ is used to protect against excessive refinement of the local minima~\cite{Jones1993,Paulavicius2014:jogo}.
Therefore, the lower Lipschitz bound on POH must be lower than the current minimum value $(f^{\rm min})$ by a considerable amount $(\ge \varepsilon \abs{f^{\min}})$.
In~\cite{Jones1993}, the authors obtained good performance using $\varepsilon$ values ranging from $10^{-3}$ to $10^{-7}$, and by default, the $\varepsilon = 10^{-4}$ value is suggested.

A geometrical interpretation of POH selection using \Cref{def:potOptRect} is illustrated in the right panel of \Cref{fig:poh}.
Here, each hyper-rectangle is represented as a dot whose horizontal coordinate is equal to the measure of the hyper-rectangle ($\delta^i_k$).
The vertical coordinate is equal to the value of the function at the midpoint $f(\mathbf{c}^i)$.
POHs satisfy both conditions of \Cref{def:potOptRect} and correspond to the lower-right convex hull of blue marked points in \Cref{fig:poh}.

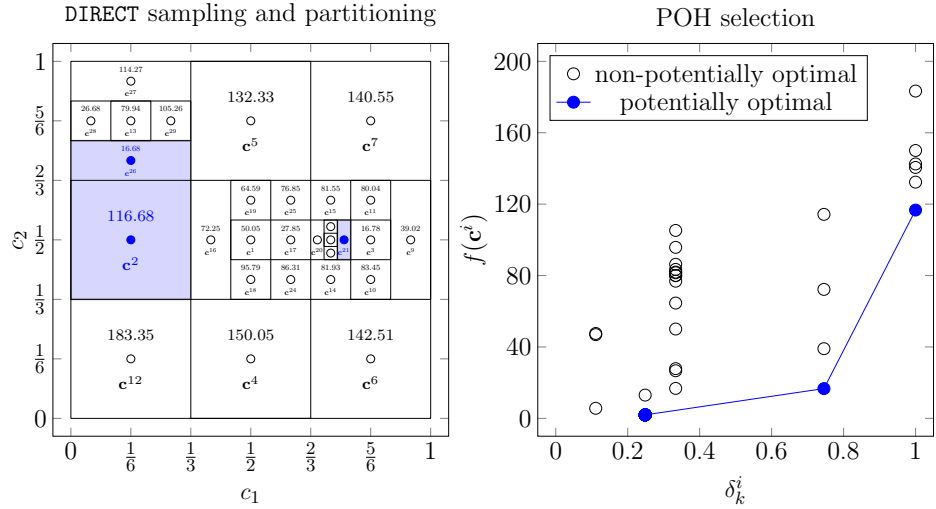
\begin{figure}
	\centering
	\resizebox{\textwidth}{!}{%
		\begin{tikzpicture}[baseline]
			\begin{axis}[
				width=0.7\textwidth,height=0.7\textwidth,
				legend style={draw=none},
				legend columns=1,
				legend style={at={(0.75,-0.15)},font=\LARGE},
				xlabel = {$c_1$},
				ylabel = {$c_2$},
				ymin=0,ymax=1,
				xmin=0,xmax=1,
				enlargelimits=0.05,
				title={\direct{} sampling and partitioning},
				ylabel style={yshift=-0.1cm},
				xlabel style={yshift=-0.1cm},
				ytick distance=1/6,
				xtick distance=1/6,
				every axis/.append style={font=\LARGE},
				yticklabels={$0$, $0$,$\frac{1}{6}$, $\frac{1}{3}$, $\frac{1}{2}$, $\frac{2}{3}$, $\frac{5}{6}$, $1$},
				xticklabels={$0$, $0$,$\frac{1}{6}$, $\frac{1}{3}$, $\frac{1}{2}$, $\frac{2}{3}$, $\frac{5}{6}$, $1$},
				]
				\draw [black, thick, mark size=0.1pt, fill=blue!40,opacity=0.4,line width=0.3mm] (axis cs:0, 1/3) rectangle (axis cs:1/3, 7/9);
				\draw [black, thick, mark size=0.1pt, fill=blue!40,opacity=0.4,line width=0.3mm] (axis cs:20/27, 4/9) rectangle (axis cs:21/27, 5/9);
				\addplot[only marks,mark=o, mark size=2pt,black] coordinates {(1/2, 1/2)} node[yshift=-6pt, scale=0.5] {\small $\mathbf{c}^1$} node[yshift=6pt, scale=0.5] {\small $50.05$};
				\addplot[only marks,mark=*, mark size=2pt,blue] coordinates {(1/6, 1/2)} node[yshift=-12pt] {\small $\mathbf{c}^2$} node[yshift=12pt] {\small $116.68$};
				\addplot[only marks,mark=o, mark size=2pt,black] coordinates {(5/6, 1/2)} node[yshift=-6pt, scale=0.5] {\small $\mathbf{c}^3$} node[yshift=6pt, scale=0.5] {\small $16.78$};
				\addplot[only marks,mark=o, mark size=2pt,black] coordinates {(1/2, 1/6)} node[yshift=-12pt] {\small $\mathbf{c}^4$} node[yshift=12pt] {\small $150.05$};
				\addplot[only marks,mark=o, mark size=2pt,black] coordinates {(1/2, 5/6)} node[yshift=-12pt] {\small $\mathbf{c}^5$} node[yshift=12pt] {\small $132.33$};
				\addplot[only marks,mark=o, mark size=2pt,black] coordinates {(5/6, 1/6)} node[yshift=-12pt] {\small $\mathbf{c}^6$} node[yshift=12pt] {\small $142.51$};
				\addplot[only marks,mark=o, mark size=2pt,black] coordinates {(5/6, 5/6)} node[yshift=-12pt] {\small $\mathbf{c}^7$} node[yshift=12pt] {\small $140.55$};
				\addplot[only marks,mark=o, mark size=2pt,black] coordinates {(13/18, 1/2)}; 
				\addplot[only marks,mark=o, mark size=2pt,black] coordinates {(17/18, 1/2)} node[yshift=-6pt, scale=0.5] {\small $\mathbf{c}^9$} node[yshift=6pt, scale=0.5] {\small $39.02$};
				\addplot[only marks,mark=o, mark size=2pt,black] coordinates {(5/6, 7/18)} node[yshift=-6pt, scale=0.5] {\small $\mathbf{c}^{10}$} node[yshift=6pt, scale=0.5] {\small $83.45$};
				\addplot[only marks,mark=o, mark size=2pt,black] coordinates {(5/6, 11/18)} node[yshift=-6pt, scale=0.5] {\small $\mathbf{c}^{11}$} node[yshift=6pt, scale=0.5] {\small $80.04$};
				\addplot[only marks,mark=o, mark size=2pt,black] coordinates {(1/6, 1/6)} node[yshift=-12pt] {\small $\mathbf{c}^{12}$} node[yshift=12pt] {\small $183.35$};
				\addplot[only marks,mark=o, mark size=2pt,black] coordinates {(1/6, 5/6)} node[yshift=-6pt, scale=0.5] {\small $\mathbf{c}^{13}$} node[yshift=6pt, scale=0.5] {\small $79.94$};
				\addplot[only marks,mark=o, mark size=2pt,black] coordinates {(13/18, 7/18)} node[yshift=-6pt, scale=0.5] {\small $\mathbf{c}^{14}$} node[yshift=6pt, scale=0.5] {\small $81.93$};
				\addplot[only marks,mark=o, mark size=2pt,black] coordinates {(13/18, 11/18)} node[yshift=-6pt, scale=0.5] {\small $\mathbf{c}^{15}$} node[yshift=6pt, scale=0.5] {\small $81.55$};
				\addplot[only marks,mark=o, mark size=2pt,black] coordinates {(7/18, 1/2)} node[yshift=-6pt, scale=0.5] {\small $\mathbf{c}^{16}$} node[yshift=6pt, scale=0.5] {\small $72.25$};
				\addplot[only marks,mark=o, mark size=2pt,black] coordinates {(11/18, 1/2)} node[yshift=-6pt, scale=0.5] {\small $\mathbf{c}^{17}$} node[yshift=6pt, scale=0.5] {\small $27.85$};
				\addplot[only marks,mark=o, mark size=2pt,black] coordinates {(1/2, 7/18)} node[yshift=-6pt, scale=0.5] {\small $\mathbf{c}^{18}$} node[yshift=6pt, scale=0.5] {\small $95.79$};
				\addplot[only marks,mark=o, mark size=2pt,black] coordinates {(1/2, 11/18)} node[yshift=-6pt, scale=0.5] {\small $\mathbf{c}^{19}$} node[yshift=6pt, scale=0.5] {\small $64.59$};
				\addplot[only marks,mark=o, mark size=2pt,black] coordinates {(37/54, 1/2)} node[yshift=-6pt, scale=0.5] {\small $\mathbf{c}^{20}$}; 
				\addplot[only marks,mark=*, mark size=2pt,blue] coordinates {(41/54, 1/2)} node[yshift=-6pt, scale=0.5] {\small $\mathbf{c}^{21}$}; 
				\addplot[only marks,mark=o, mark size=2pt,black] coordinates {(13/18, 25/54)}; 
				\addplot[only marks,mark=o, mark size=2pt,black] coordinates {(13/18, 29/54)}; 
				\addplot[only marks,mark=o, mark size=2pt,black] coordinates {(11/18, 7/18)} node[yshift=-6pt, scale=0.5] {\small $\mathbf{c}^{24}$} node[yshift=6pt, scale=0.5] {\small $86.31$};
				\addplot[only marks,mark=o, mark size=2pt,black] coordinates {(11/18, 11/18)} node[yshift=-6pt, scale=0.5] {\small $\mathbf{c}^{25}$} node[yshift=6pt, scale=0.5] {\small $76.85$};
				\addplot[only marks,mark=*, mark size=2pt,blue] coordinates {(1/6, 13/18)} node[yshift=-6pt, scale=0.5] {\small $\mathbf{c}^{26}$} node[yshift=6pt, scale=0.5] {\small $16.68$};
				\addplot[only marks,mark=o, mark size=2pt,black] coordinates {(1/6, 17/18)} node[yshift=-6pt, scale=0.5] {\small $\mathbf{c}^{27}$} node[yshift=6pt, scale=0.5] {\small $114.27$};
				\addplot[only marks,mark=o, mark size=2pt,black] coordinates {(1/18, 5/6)} node[yshift=-6pt, scale=0.5] {\small $\mathbf{c}^{28}$} node[yshift=6pt, scale=0.5] {\small $26.68$};
				\addplot[only marks,mark=o, mark size=2pt,black] coordinates {(5/18, 5/6)} node[yshift=-6pt, scale=0.5] {\small $\mathbf{c}^{29}$} node[yshift=6pt, scale=0.5] {\small $105.26$};
				\addplot[patch,mesh,draw,black,patch type=rectangle] coordinates {
					(0, 0) (1, 0) (1, 1) (0, 1)
					(0, 1/3) (1, 1/3) (1, 2/3) (0, 2/3)
					(1/3, 0) (1/3, 1) (2/3, 1) (2/3, 0)
					(7/9, 1/3) (7/9, 2/3) (8/9, 2/3) (8/9, 1/3)
					(4/9, 4/9) (4/9, 5/9) (8/9, 5/9) (8/9, 4/9)
					(4/9, 1/3) (4/9, 2/3) (5/9, 2/3) (5/9, 1/3)
					(1/9, 7/9) (1/9, 8/9) (2/9, 8/9) (2/9, 7/9)
					(0, 7/9) (3/9, 7/9) (3/9, 8/9) (0, 8/9)
					(19/27, 4/9) (19/27, 5/9) (20/27, 5/9) (20/27, 4/9)
					(19/27, 13/27) (19/27, 14/27) (20/27, 14/27) (20/27, 13/27)};
			\end{axis}
		\end{tikzpicture}
		\begin{tikzpicture}[baseline]
			\begin{axis}[
				width=0.7\textwidth,height=0.7\textwidth,
				every axis/.append style={font=\LARGE},
				xlabel = {$\delta^i_k$},
				ymin=0,ymax=200,
				xmin=0,xmax=1,
				xtick distance=0.2,
				ytick distance=40,
				title={POH selection},
				ylabel = {$f(\mathbf{c}^i)$},
				legend pos=north west,
				ylabel style={yshift=-0.1cm},
				xlabel style={yshift=-0.1cm},
				enlargelimits=0.05,
				]
				\addplot[only marks,mark=o,mark size=3pt,black] table[x=Y,y=X] {data/DPD.txt};
				\addplot[mark=*,mark size=3pt,blue] table[x=T,y=Z] {data/DPD.txt};
				\legend{non-potentially optimal,potentially optimal};
			\end{axis}
	\end{tikzpicture}}
	\caption{Visualization of selected potentially optimal rectangles in the fifth iteration of the \direct{} algorithm solving two-dimensional \textit{Bukin6} test problem.}
	\label{fig:poh}
\end{figure}

\subsection{Brief review of candidate selection schemes}
\label{sec:selection-schemes}
Typically, the \direct-type algorithms include three main steps: selection (of POHs), sampling, and partitioning (subdivision).
At each iteration, a specific \direct-type algorithm first selects the set of POHs before sampling and subdividing them.
In~\cite{Stripinis2021b}, we reviewed various improvements and new ideas introduced for the selection of POH proposed in the \direct{} literature.
The three most promising ones were extracted and used to construct new \direct-type algorithms, combining them with four different sampling and partitioning techniques.
For consistency, we give a brief description and a summary (see ~\Cref{tab:selection-schemes}) of the most often used selection schemes.
\Cref{sec:partitioning-schemes} briefly reviews sampling and partitioning techniques traditionally used in \direct-type algorithms.

\subsubsection{Improved original selection strategy}
It was observed that the original candidate selection strategy could be very inefficient on symmetric and other specific problems.
There may be many POHs with the same diameter $\delta^i_k$ and objective value, leading to a drastic increase in selected POHs per iteration.
To overcome this, the authors of \cite{Gablonsky2001} proposed an improvement by selecting only one of these many ``equivalent candidates''.
In \cite{Jones2021}, the authors showed that such modification could significantly increase the performance of the \direct{} algorithm.

\subsubsection{Aggressive Selection strategy}
In~\cite{Baker2000}, the authors relaxed the selection criteria of POHs and proposed an aggressive version.
The main idea is to select and divide at least one hyper-rectangle from each group of different diameters $(\delta_k^i)$ with the lowest value of the function.
\Cref{def:potOptRectAggr} formalizes the strategy for identifying an aggressive set of potentially optimal hyper-rectangles from the current partition.
\begin{definition}{(Aggressive selection)}
	\label{def:potOptRectAggr}
	Let $ \mathbf{c}^i $, $f(\mathbf{c}^i)$, and $ \delta^i_k $ be defined as in \Cref{def:potOptRect}.
	Let $\mathbb{I}_k^i \subseteq \mathbb{I}_k$ be the subset of indices corresponding to hyper-rectangles having the same measure $(\delta_k^i)$.
	The notation $\mathbb{I}^{\rm min}_k$ corresponds to the subset of hyper-rectangle indices that has the smallest measure $\delta_k^{\rm min}$, while $\mathbb{I}^{\rm max}_k $ has the largest measure $(\delta_k^{\rm max})$, and $\mathbb{I}_k =  \mathbb{I}^{\rm min}_k \cup \dots \cup \mathbb{I}^{\rm max}_k $.
	
	Then for each subset, $ \mathbb{I}_k^i \ (\min \le i \le \max) $, find hyper-rectangle(s) $ \bar{D}^h_k, h \in \indexsett_k^i $ with the lowest function value among all of the same measure $ (\delta^i_k) $, i.e.,
	\begin{eqnarray}
		f(\mathbf{c}^h) & \leq & f(\mathbf{c}^l), \quad \forall l \in \indexsett_k^i. \label{eqn:potOptRectAggr}
	\end{eqnarray}
\end{definition}

For the situation presented in~\Cref{fig:poh}, using \Cref{def:potOptRectAggr}, two additional hyper-rectangles from the groups where the original selection (see \Cref{def:potOptRect}) does not consider are selected.
From the Lipschitz optimization point of view, such an approach may seem less favorable since it explores non-potentially optimal hyper-rectangles.
There is no such positive constant $\tilde{L}$ value, using which the lower Lipschitz bound would have the lowest values for these additional candidates selected by an aggressive strategy.

\subsubsection{Improved aggressive selection strategy}
\label{ssec:ImprovedAggressiveSelection}
In \cite{He2008}, the authors introduced an improvement to aggressive selection. 
They showed that by limiting the refinement of the search space when the measure of hyper-rectangles $ (\delta^i_k )$ reached some prescribed limit $ \delta^{\rm limit} $, the memory usage might be reduced from $10 \%$ to $70 \%$.
Therefore, the improved aggressive version can run longer without memory allocation failure.
We note that in our experimental part (described in \Cref{sec:results}), the limit parameter $ (\delta^{\rm limit}) $ for algorithms using this selection scheme was set to the measure of a hyper-rectangle that has been subdivided $ 50 n $ times (same as in \cite{Stripinis2021b}).

\subsubsection{Two-step-based Pareto selection}
\label{ssec:ParetoSelection}
In our recent extension, \directgl{}~\cite{Stripinis2018a}, we introduced a new two-step-based approach to identify the extended set of POHs, formally stated in \Cref{def:potOptRectPareto}.

\begin{definition}{(Two-step Pareto selection)}
	\label{def:potOptRectPareto}
	Find all Pareto optimal hyper-rectangles that are non-dominated on size (the higher, the better) and center point function value (the lower, the better), and all non-dominated on size and distance from the current minimum point (the closer, the better).
	Then take the unique union of these two identified sets of candidates.
\end{definition}
Unlike the aggressive strategy (\Cref{def:potOptRectAggr}), using \Cref{def:potOptRectPareto}, hyper-rectangles from the groups where the minimum objective function value is higher than the minimum value from the larger groups are not selected.
Compared to the original selection (\Cref{def:potOptRect}), using \Cref{def:potOptRectPareto}, the set of POHs is enlarged by adding more medium-sized hyper-rectangles.
In this sense, Pareto selection may be more global.
Additionally, in the second step, the hyper-rectangles that are non-dominated with respect to the size and distance from the current minimum point are selected.
Therefore, the set of POHs is enlarged with various size hyper-rectangles nearest the current minimum point.

\begin{table}
	\begin{minipage}{\textwidth}
		\caption{Summary of selection schemes typically used in \direct-type algorithms (in ascending order of the year of publication)}
		\label{tab:selection-schemes}
		\begin{tabular}{p{0.25\textwidth}p{0.3\textwidth}p{0.35\textwidth}}
			\toprule
			\textbf{Notation $\&$ source} & Identification of POH & Final selection of POH \\
			\midrule
			\textbf{OS} (Jones et. al, 1993) & \textit{\textbf{O}riginal \textbf{S}election strategy} using \Cref{def:potOptRect}. & Selects all candidates which satisfies \Cref{def:potOptRect}. \\
			\midrule
			\textbf{AS} (Baker et. al, 2000) & \textit{\textbf{A}ggressive \textbf{S}election strategy} using \cref{def:potOptRectAggr}. & Selects all candidates which satisfies \Cref{def:potOptRectAggr}. \\
			\midrule
			\textbf{IO} (Gablonsky et. al, 2001) & \textit{\textbf{I}mproved \textbf{O}riginal selection strategy} using \Cref{def:potOptRect}. & Selects only one hyper-rectangle if there is a tie for the lowest function value in the same diameter group.\\
			\midrule
			\textbf{IA} (He et. al, 2008) & \textit{\textbf{I}mproved \textbf{A}ggressive selection strategy} using \cref{def:potOptRectAggr}, but limiting the selection of candidates to some prescribed limit $(\delta^{\rm limit})$. & Selects only one hyper-rectangle if there is a tie for the lowest function value in the same diameter group and $\delta^i_k \ge \delta^{\rm limit}$. \\
			\midrule
			\textbf{GL} (Stripinis et. al, 2018) & \textit{Two-step-based (\textbf{G}lobal-\textbf{L}ocal) Pareto selection} using \Cref{def:potOptRectPareto}. & Selects only one hyper-rectangle if there is a tie for the lowest function value or distance from the current minimum point.\\
			\bottomrule
		\end{tabular}
	\end{minipage}
\end{table}

\subsection{Brief review of sampling and partitioning schemes}
\label{sec:partitioning-schemes}
This subsection briefly reviews some of the primary sampling and partitioning techniques proposed in the \direct{} literature.
A summary of them is given in \Cref{tab:partitioning-strategies}, including illustrative examples.

\subsubsection{Hyper-rectangular partitioning based on \textbf{1-d}imensional \textbf{t}risection and \textbf{c}enter sampling}
In \cite{Jones2001}, Jones proposed a revised version of the original \direct{} algorithm.
One of the main algorithmic changes was made to the partitioning scheme.
The author suggested trisecting selected POHs only along the longest side (coordinate).
If there are several equal longest sides, the dimension that has been split the fewest times during the search procedure is selected.
If there is a tie on the latter criterion, the lowest indexed dimension is selected.
In \cite{Jones2021}, the authors experimentally justified that such modification can significantly improve the performance of the original \direct{} algorithm.

\begin{table}
	\begin{minipage}{\textwidth}
		\caption{Summary of sampling and partitioning schemes typically used in the \direct-type algorithms (in ascending order of the year of publication)}
		\label{tab:partitioning-strategies}
		\begin{tabular}{p{0.12\textwidth}p{0.24\textwidth}p{0.24\textwidth}p{0.25\textwidth}}
			\toprule
			\textbf{Notation \& Source} & Partitioning scheme & Sampling scheme & Illustrative example	\\
			\midrule
			\textbf{N-DTC} (Jones et. al, 1993) & Hyper-rectangular partitioning based on \textbf{N-D}imensional \textbf{T}risection. & Sampling points are located at the \textbf{C}enter points of each hyper-rectangle. &
			\raisebox{-0.8\totalheight}{
				\begin{tikzpicture}
					\begin{axis}[
						width=0.25\textwidth,height=0.25\textwidth,
						ytick=\empty,
						xtick=\empty,
						enlargelimits=0,
						]
						\addplot[thick,patch,mesh,draw,black,patch type=rectangle,line width=0.1mm] coordinates {(0, 0) (1, 0) (1, 1) (0, 1)} ;
						\addplot[thick,patch,mesh,draw,black,patch type=rectangle,line width=0.1mm] coordinates {(1/3, 0) (1/3, 1) (2/3, 1) (2/3, 0)};
						\addplot[thick,patch,mesh,draw,black,patch type=rectangle,line width=0.1mm] coordinates {(1/3, 1/3) (1/3, 2/3) (1, 2/3) (1, 1/3)};
						\draw [black, thick, mark size=0.1pt, fill=blue!40,opacity=0.4,line width=0.1mm] (axis cs:0, 0) rectangle (axis cs:1/3, 1);
						\draw [black, thick, mark size=0.1pt, fill=blue!40,opacity=0.4,line width=0.1mm] (axis cs:2/3, 1/3) rectangle (axis cs:1, 2/3);
						\addplot[only marks,mark=o, mark size=1.25pt,black] coordinates {(1/2, 1/2)};
						\addplot[only marks,mark=*, mark size=1.25pt,blue]  coordinates {(1/6, 1/2)};
						\addplot[only marks,mark=*, mark size=1.25pt,blue]  coordinates {(5/6, 1/2)};
						\addplot[only marks,mark=o, mark size=1.25pt,black] coordinates {(1/2, 1/6)};
						\addplot[only marks,mark=o, mark size=1.25pt,black] coordinates {(1/2, 5/6)};
						\addplot[only marks,mark=o, mark size=1.25pt,black] coordinates {(5/6, 1/6)};
						\addplot[only marks,mark=o, mark size=1.25pt,black] coordinates {(5/6, 5/6)};
					\end{axis}
				\end{tikzpicture}
				\begin{tikzpicture}
					\begin{axis}[
						width=0.25\textwidth,height=0.25\textwidth,
						ytick=\empty,
						xtick=\empty,
						enlargelimits=0,
						]
						\addplot[thick,patch,mesh,draw,black,patch type=rectangle,line width=0.1mm] coordinates {(0, 0) (1, 0) (1, 1) (0, 1)} ;
						\draw [black, thick, mark size=0.1pt,line width=0.1mm] (axis cs:0, 1/3) rectangle (axis cs:1, 2/3);
						\draw [black, thick, mark size=0.1pt,line width=0.1mm] (axis cs:7/9, 1/3) rectangle (axis cs:8/9, 2/3);
						\draw [black, thick, mark size=0.1pt,line width=0.1mm] (axis cs:7/9, 4/9) rectangle (axis cs:8/9, 5/9);
						\draw [black, thick, mark size=0.1pt,line width=0.1mm] (axis cs:1/3, 0) rectangle (axis cs:2/3, 1);
						\addplot[only marks,mark=o, mark size=1.25pt,black] coordinates {(1/2, 1/2)};
						\addplot[only marks,mark=o, mark size=1.25pt,black] coordinates {(1/6, 1/2)};
						\addplot[only marks,mark=o, mark size=1.25pt,black] coordinates {(5/6, 1/2)};
						\addplot[only marks,mark=o, mark size=1.25pt,black] coordinates {(1/2, 1/6)};
						\addplot[only marks,mark=o, mark size=1.25pt,black] coordinates {(1/2, 5/6)};
						\addplot[only marks,mark=o, mark size=1.25pt,black] coordinates {(5/6, 1/6)};
						\addplot[only marks,mark=o, mark size=1.25pt,black] coordinates {(5/6, 5/6)};
						\addplot[only marks,mark=o, mark size=1.25pt,black] coordinates {(1/6, 5/6)};
						\addplot[only marks,mark=o, mark size=1.25pt,black] coordinates {(1/6, 1/6)};
						\addplot[only marks,mark=o, mark size=1.25pt,black] coordinates {(5/6, 11/18)};
						\addplot[only marks,mark=o, mark size=1.25pt,black] coordinates {(5/6, 7/18)};
						\addplot[only marks,mark=o, mark size=1.25pt,black] coordinates {(17/18, 1/2)};
						\addplot[only marks,mark=o, mark size=1.25pt,black] coordinates {(13/18, 1/2)};
					\end{axis}
			\end{tikzpicture}} \\
			\midrule
			\textbf{1-DTC} (Jones et. al, 2001) & Hyper-rectangular partitioning based on \textbf{1-D}imensional \textbf{T}risection. & Sampling points are located at the \textbf{C}enter points of each hyper-rectangle. &
			\raisebox{-0.8\totalheight}{
				\begin{tikzpicture}
					\begin{axis}[
						width=0.25\textwidth,height=0.25\textwidth,
						ytick=\empty,
						xtick=\empty,
						enlargelimits=0,
						]
						\addplot[thick,patch,mesh,draw,black,patch type=rectangle,line width=0.1mm] coordinates {(0, 0) (1, 0) (1, 1) (0, 1)};
						\draw [black, thick, mark size=0.1pt,line width=0.1mm] (axis cs:1/3, 0) rectangle (axis cs:2/3, 1);
						\draw [black, thick, mark size=0.1pt, fill=blue!40,opacity=0.4,line width=0.1mm] (axis cs:0, 0) rectangle (axis cs:1/3, 1);
						\draw [black, thick, mark size=0.1pt, fill=blue!40,opacity=0.4,line width=0.1mm] (axis cs:2/3, 1/3) rectangle (axis cs:1, 2/3);
						\draw [black, thick, mark size=0.1pt,line width=0.1mm] (axis cs:2/3, 1/3) rectangle (axis cs:1, 2/3);
						\addplot[only marks,mark=o, mark size=1.25pt,black] coordinates {(1/2, 1/2)};
						\addplot[only marks,mark=*, mark size=1.25pt,blue]  coordinates {(1/6, 1/2)};
						\addplot[only marks,mark=*, mark size=1.25pt,blue]  coordinates {(5/6, 1/2)};
						\addplot[only marks,mark=o, mark size=1.25pt,black] coordinates {(5/6, 1/6)};
						\addplot[only marks,mark=o, mark size=1.25pt,black] coordinates {(5/6, 5/6)};
					\end{axis}
				\end{tikzpicture}
				\begin{tikzpicture}
					\begin{axis}[
						width=0.25\textwidth,height=0.25\textwidth,
						ytick=\empty,
						xtick=\empty,
						enlargelimits=0,
						]
						\addplot[thick,patch,mesh,draw,black,patch type=rectangle,line width=0.1mm] coordinates {(0, 0) (1, 0) (1, 1) (0, 1)} ;
						\draw [black, thick, mark size=0.1pt,line width=0.1mm] (axis cs:0, 1/3) rectangle (axis cs:1/3, 2/3);
						\draw [black, thick, mark size=0.1pt,line width=0.1mm] (axis cs:2/3, 1/3) rectangle (axis cs:1, 2/3);
						\draw [black, thick, mark size=0.1pt,line width=0.1mm] (axis cs:2/3, 4/9) rectangle (axis cs:1, 5/9);
						\draw [black, thick, mark size=0.1pt,line width=0.1mm] (axis cs:1/3, 0) rectangle (axis cs:2/3, 1);
						\addplot[only marks,mark=o, mark size=1.25pt,black] coordinates {(1/2, 1/2)};
						\addplot[only marks,mark=o, mark size=1.25pt,black] coordinates {(1/6, 1/2)};
						\addplot[only marks,mark=o, mark size=1.25pt,black] coordinates {(5/6, 1/2)};
						\addplot[only marks,mark=o, mark size=1.25pt,black] coordinates {(5/6, 1/6)};
						\addplot[only marks,mark=o, mark size=1.25pt,black] coordinates {(5/6, 5/6)};
						\addplot[only marks,mark=o, mark size=1.25pt,black] coordinates {(1/6, 5/6)};
						\addplot[only marks,mark=o, mark size=1.25pt,black] coordinates {(1/6, 1/6)};
						\addplot[only marks,mark=o, mark size=1.25pt,black] coordinates {(5/6, 11/18)};
						\addplot[only marks,mark=o, mark size=1.25pt,black] coordinates {(5/6, 7/18)};
					\end{axis}
			\end{tikzpicture}} \\
			\midrule
			\textbf{1-DTDV} (Sergeyev et. al, 2006) & Hyper-rectangular partitioning based on \textbf{1-D}imensional \textbf{T}risection. & Sampling points are located at two \textbf{D}iagonal \textbf{V}ertices of each hyper-rectangle. &
			\raisebox{-0.8\totalheight}{
				\begin{tikzpicture}
					\begin{axis}[
						width=0.25\textwidth,height=0.25\textwidth,
						ytick=\empty,
						xtick=\empty,
						enlargelimits=0,
						]
						\addplot[thick,patch,mesh,draw,black,patch type=rectangle,line width=0.1mm] coordinates {(0, 0) (1, 0) (1, 1) (0, 1)};
						\draw [black, thick, mark size=0.1pt,line width=0.1mm] (axis cs:1/3, 0) rectangle (axis cs:2/3, 1);
						\draw [black, thick, mark size=0.1pt, fill=blue!40,opacity=0.4,line width=0.1mm] (axis cs:0, 0) rectangle (axis cs:1/3, 1);
						\draw [black, thick, mark size=0.1pt, fill=blue!40,opacity=0.4,line width=0.1mm] (axis cs:2/3, 1/3) rectangle (axis cs:1, 2/3);
						\draw [black, thick, mark size=0.1pt,line width=0.1mm] (axis cs:2/3, 1/3) rectangle (axis cs:1, 2/3);
						\addplot[only marks,mark=*, mark size=1.25pt,blue]  coordinates {(0, 0)};
						\addplot[only marks,mark=*, mark size=1.25pt,blue]  coordinates {(1/3, 1)};
						\addplot[only marks,mark=o, mark size=1.25pt,black] coordinates {(2/3, 0)};
						\addplot[only marks,mark=o, mark size=1.25pt,black] coordinates {(1, 1)};
						\addplot[only marks,mark=*, mark size=1.25pt,blue]  coordinates {(2/3, 2/3)};
						\addplot[only marks,mark=*, mark size=1.25pt,blue]  coordinates {(1, 1/3)};
					\end{axis}
				\end{tikzpicture}
				\begin{tikzpicture}
					\begin{axis}[
						width=0.25\textwidth,height=0.25\textwidth,
						ytick=\empty,
						xtick=\empty,
						enlargelimits=0,
						]
						\addplot[thick,patch,mesh,draw,black,patch type=rectangle,line width=0.1mm] coordinates {(0, 0) (1, 0) (1, 1) (0, 1)} ;
						\draw [black, thick, mark size=0.1pt,line width=0.1mm] (axis cs:0, 1/3) rectangle (axis cs:1/3, 2/3);
						\draw [black, thick, mark size=0.1pt,line width=0.1mm] (axis cs:2/3, 1/3) rectangle (axis cs:1, 2/3);
						\draw [black, thick, mark size=0.1pt,line width=0.1mm] (axis cs:7/9, 1/3) rectangle (axis cs:8/9, 2/3);
						\draw [black, thick, mark size=0.1pt,line width=0.1mm] (axis cs:1/3, 0) rectangle (axis cs:2/3, 1);
						\addplot[only marks,mark=o, mark size=1.25pt,black] coordinates {(0, 0)};
						\addplot[only marks,mark=o, mark size=1.25pt,black] coordinates {(1/3, 1)};
						\addplot[only marks,mark=o, mark size=1.25pt,black] coordinates {(2/3, 0)};
						\addplot[only marks,mark=o, mark size=1.25pt,black] coordinates {(1, 1)};
						\addplot[only marks,mark=o, mark size=1.25pt,black] coordinates {(2/3, 2/3)};
						\addplot[only marks,mark=o, mark size=1.25pt,black] coordinates {(1, 1/3)};
						\addplot[only marks,mark=o, mark size=1.25pt,black] coordinates {(0, 2/3)};
						\addplot[only marks,mark=o, mark size=1.25pt,black] coordinates {(1/3, 1/3)};
						\addplot[only marks,mark=o, mark size=1.25pt,black] coordinates {(7/9, 1/3)};
						\addplot[only marks,mark=o, mark size=1.25pt,black] coordinates {(8/9, 2/3)};
					\end{axis}
			\end{tikzpicture}} \\
			\midrule
			\textbf{1-DTCS} (Paulavi\v{s}ius et. al, 2014) & Simplicial partitioning based on \textbf{1-D}imensional \textbf{T}risection. & Sampling points are located at the \textbf{C}enter points of each \textbf{S}implex. &
			\raisebox{-0.8\totalheight}{
				\begin{tikzpicture}
					\begin{axis}[
						width=0.25\textwidth,height=0.25\textwidth,
						ytick=\empty,
						xtick=\empty,
						enlargelimits=0,
						]
						\addplot[thick,patch,mesh,draw,black,patch type=rectangle,line width=0.1mm] coordinates {(0, 0) (1, 0) (1, 1) (0, 1)};
						\addplot[thick,patch,mesh,draw,black,line width=0.1mm] coordinates {(1/3, 1/3) (1, 0) (2/3, 2/3)};
						\addplot[thick,patch,mesh,draw,black,line width=0.1mm] coordinates {(1/3, 1/3) (0, 1) (2/3, 2/3)};
						\addplot[thick,patch,mesh,draw,black,line width=0.1mm] coordinates {(1/3, 0) (2/3, 0) (1/3, 1/3)};
						\addplot[thick,patch,mesh,draw,black,line width=0.1mm] coordinates {(2/3, 2/3) (1, 0) (1, 1)};
						\addplot[thick,patch,mesh,draw,black,line width=0.1mm] coordinates {(0, 0) (1/3, 0) (1/3, 1/3)};
						\addplot[thick,black,fill=blue!40,opacity=0.4,line width=0.1mm] coordinates {(2/3, 0) (1, 0) (1/3, 1/3)};
						\addplot[thick,patch,mesh,draw,black,line width=0.1mm] coordinates {(2/3, 0) (1, 0) (1/3, 1/3)};
						\addplot[thick,black,fill=blue!40,opacity=0.4,line width=0.1mm] coordinates {(0, 0) (0, 1) (1/3, 1/3)};
						\addplot[thick,patch,mesh,draw,black,line width=0.1mm] coordinates {(0, 0) (0, 1) (1/3, 1/3)};
						\addplot[thick,patch,mesh,draw,black,line width=0.1mm] coordinates {(2/3, 2/3) (0, 1) (1, 1)};
						\addplot[only marks,mark=o, mark size=1.25pt,black]  coordinates {(2/3, 1/3)};
						\addplot[only marks,mark=o, mark size=1.25pt,black]  coordinates {(1/3, 2/3)};
						\addplot[only marks,mark=o, mark size=1.25pt,black]  coordinates {(4/9, 1/9)};
						\addplot[only marks,mark=o, mark size=1.25pt,black]  coordinates {(8/9, 5/9)};
						\addplot[only marks,mark=o, mark size=1.25pt,black]  coordinates {(2/9, 1/9)};
						\addplot[only marks,mark=*, mark size=1.25pt,blue]   coordinates {(6/9, 1/9)};
						\addplot[only marks,mark=*, mark size=1.25pt,blue]   coordinates {(1/9, 4/9)};
						\addplot[only marks,mark=o, mark size=1.25pt,black]  coordinates {(5/9, 8/9)};
					\end{axis}
				\end{tikzpicture}
				\begin{tikzpicture}
					\begin{axis}[
						width=0.25\textwidth,height=0.25\textwidth,
						ytick=\empty,
						xtick=\empty,
						enlargelimits=0,
						]
						\addplot[thick,patch,mesh,draw,black,patch type=rectangle,line width=0.1mm] coordinates {(0, 0) (1, 0) (1, 1) (0, 1)};
						\addplot[thick,patch,mesh,draw,black,line width=0.1mm] coordinates {(1/3, 1/3) (1, 0) (2/3, 2/3)};
						\addplot[thick,patch,mesh,draw,black,line width=0.1mm] coordinates {(1/3, 1/3) (0, 1) (2/3, 2/3)};
						\addplot[thick,patch,mesh,draw,black,line width=0.1mm] coordinates {(1/3, 0) (2/3, 0) (1/3, 1/3)};
						\addplot[thick,patch,mesh,draw,black,line width=0.1mm] coordinates {(2/3, 2/3) (1, 0) (1, 1)};
						\addplot[thick,patch,mesh,draw,black,line width=0.1mm] coordinates {(0, 0) (1/3, 0) (1/3, 1/3)};
						\addplot[thick,patch,mesh,draw,black,line width=0.1mm] coordinates {(2/3, 0) (1, 0) (1/3, 1/3)};
						\addplot[thick,patch,mesh,draw,black,line width=0.1mm] coordinates {(0, 0) (0, 1) (1/3, 1/3)};
						\addplot[thick,patch,mesh,draw,black,line width=0.1mm] coordinates {(2/3, 2/3) (0, 1) (1, 1)};
						\addplot[thick,patch,mesh,draw,black,line width=0.1mm] coordinates {(0, 0) (0, 1/3) (1/3, 1/3)};
						\addplot[thick,patch,mesh,draw,black,line width=0.1mm] coordinates {(0, 2/3) (0, 1) (1/3, 1/3)};
						\addplot[thick,patch,mesh,draw,black,line width=0.1mm] coordinates {(2/3, 0) (1, 0) (7/9, 1/9)};
						\addplot[thick,patch,mesh,draw,black,line width=0.1mm] coordinates {(2/3, 0) (5/9, 2/9) (1/3, 1/3)};
						\addplot[only marks,mark=o, mark size=1.25pt,black]  coordinates {(2/3, 1/3)};
						\addplot[only marks,mark=o, mark size=1.25pt,black]  coordinates {(1/3, 2/3)};
						\addplot[only marks,mark=o, mark size=1.25pt,black]  coordinates {(4/9, 1/9)};
						\addplot[only marks,mark=o, mark size=1.25pt,black]  coordinates {(8/9, 5/9)};
						\addplot[only marks,mark=o, mark size=1.25pt,black]  coordinates {(2/9, 1/9)};
						\addplot[only marks,mark=o, mark size=1.25pt,black]  coordinates {(6/9, 1/9)};
						\addplot[only marks,mark=o, mark size=1.25pt,black]  coordinates {(1/9, 4/9)};
						\addplot[only marks,mark=o, mark size=1.25pt,black]  coordinates {(5/9, 8/9)};
						\addplot[only marks,mark=o, mark size=1.25pt,black]  coordinates {(1/9, 2/9)};
						\addplot[only marks,mark=o, mark size=1.25pt,black]  coordinates {(1/9, 6/9)};
						\addplot[only marks,mark=o, mark size=1.25pt,black]  coordinates {(22/27, 1/27)};
						\addplot[only marks,mark=o, mark size=1.25pt,black]  coordinates {(14/27, 5/27)};
					\end{axis}
			\end{tikzpicture}} \\
			\midrule
			\textbf{1-DBVS} (Paulavi\v{s}ius et. al, 2014) & Simplicial partitioning based on \textbf{1-D}imensional \textbf{B}isection. & Sampling points are located at \textbf{V}ertices of each \textbf{S}implex. &
			\raisebox{-0.8\totalheight}{
				\begin{tikzpicture}
					\begin{axis}[
						width=0.25\textwidth,height=0.25\textwidth,
						ytick=\empty,
						xtick=\empty,
						enlargelimits=0,
						]
						\addplot[thick,patch,mesh,draw,black,patch type=rectangle,line width=0.1mm] coordinates {(0, 0) (1, 0) (1, 1) (0, 1)};
						\addplot[thick,black,fill=blue!40,opacity=0.4,line width=0.1mm] coordinates {(0, 0) (1/2, 1/2) (0, 1)};
						\addplot[thick,patch,mesh,draw,black,line width=0.1mm] coordinates {(1/2, 1/2) (1, 1/2) (1, 1)};
						\addplot[thick,patch,mesh,draw,black,line width=0.1mm] coordinates {(1/2, 1/2) (0, 1) (1, 1)};
						\addplot[thick,black,fill=blue!40,opacity=0.4,line width=0.1mm] coordinates {(0, 0) (1/2, 0) (1/2, 1/2)};
						\addplot[thick,patch,mesh,draw,black,line width=0.1mm] coordinates {(1/2, 0) (1, 0) (1/2, 1/2)};
						\addplot[thick,patch,mesh,draw,black,line width=0.1mm] coordinates {(1/2, 1/2) (1, 0) (1, 1/2)};
						\addplot[thick,patch,mesh,draw,black,line width=0.1mm] coordinates {(0, 0) (0, 1) (1/2, 1/2)};
						\addplot[thick,patch,mesh,draw,black,line width=0.1mm] coordinates {(0, 0) (1/2, 0) (1/2, 1/2)};
						\addplot[only marks,mark=*, mark size=1.25pt,blue]  coordinates {(0, 0)};
						\addplot[only marks,mark=*, mark size=1.25pt,blue]  coordinates {(0, 1)};
						\addplot[only marks,mark=o, mark size=1.25pt,black] coordinates {(1, 0)};
						\addplot[only marks,mark=o, mark size=1.25pt,black] coordinates {(1, 1)};
						\addplot[only marks,mark=*, mark size=1.25pt,blue]  coordinates {(1/2, 1/2)};
						\addplot[only marks,mark=o, mark size=1.25pt,black] coordinates {(1, 1/2)};
						\addplot[only marks,mark=*, mark size=1.25pt,blue]  coordinates {(1/2, 0)};
					\end{axis}
				\end{tikzpicture}
				\begin{tikzpicture}
					\begin{axis}[
						width=0.25\textwidth,height=0.25\textwidth,
						ytick=\empty,
						xtick=\empty,
						enlargelimits=0,
						]
						\addplot[thick,patch,mesh,draw,black,patch type=rectangle,line width=0.1mm] coordinates {(0, 0) (1, 0) (1, 1) (0, 1)};
						\addplot[thick,patch,mesh,draw,black,line width=0.1mm] coordinates {(1/2, 1/2) (1, 1/2) (1, 1)};
						\addplot[thick,patch,mesh,draw,black,line width=0.1mm] coordinates {(1/2, 0) (1, 0) (1/2, 1/2)};
						\addplot[thick,patch,mesh,draw,black,line width=0.1mm] coordinates {(1/2, 1/2) (1, 0) (1, 1/2)};
						\addplot[thick,patch,mesh,draw,black,line width=0.1mm] coordinates {(0, 0) (0, 1) (1/2, 1/2)};
						\addplot[thick,patch,mesh,draw,black,line width=0.1mm] coordinates {(0, 0) (1/2, 0) (1/2, 1/2)};
						\addplot[thick,patch,mesh,draw,black,line width=0.1mm] coordinates {(0, 0) (0, 1/2) (1/2, 1/2)};
						\addplot[thick,patch,mesh,draw,black,line width=0.1mm] coordinates {(0, 0) (1/2, 0) (1/4, 1/4)};
						\addplot[only marks,mark=o, mark size=1.25pt,black]  coordinates {(0, 0)};
						\addplot[only marks,mark=o, mark size=1.25pt,black]  coordinates {(0, 1)};
						\addplot[only marks,mark=o, mark size=1.25pt,black]  coordinates {(1, 0)};
						\addplot[only marks,mark=o, mark size=1.25pt,black]  coordinates {(1, 1)};
						\addplot[only marks,mark=o, mark size=1.25pt,black]  coordinates {(1/2, 1/2)};
						\addplot[only marks,mark=o, mark size=1.25pt,black]  coordinates {(1, 1/2)};
						\addplot[only marks,mark=o, mark size=1.25pt,black]  coordinates {(1/2, 0)};
						\addplot[only marks,mark=o, mark size=1.25pt,black]  coordinates {(0, 1/2)};
						\addplot[only marks,mark=o, mark size=1.25pt,black]  coordinates {(1/4, 1/4)};
					\end{axis}
			\end{tikzpicture}} \\
			\midrule
			\textbf{1-DBDP} (Paulavi\v{s}ius et. al, 2016) & Hyper-rectangular partitioning based on \textbf{1-D}imensional \textbf{B}isection. & Sampling points are located at two \textbf{D}iagonals \textbf{P}oints equidistant between themselves and a diagonal's vertices.&
			\raisebox{-0.8\totalheight}{
				\begin{tikzpicture}
					\begin{axis}[
						width=0.25\textwidth,height=0.25\textwidth,
						ytick=\empty,
						xtick=\empty,
						enlargelimits=0,
						]
						\addplot[thick,patch,mesh,draw,black,patch type=rectangle,line width=0.1mm] coordinates {(0, 0) (1, 0) (1, 1) (0, 1)};
						\draw [black, thick,fill=blue!40,opacity=0.4, mark size=0.1pt,line width=0.1mm] (axis cs:0, 0) rectangle (axis cs:1/2, 1/2);
						\draw [black, thick,fill=blue!40,opacity=0.4, mark size=0.1pt,line width=0.1mm] (axis cs:0, 1/2) rectangle (axis cs:1/4, 1);
						\draw [black, thick, mark size=0.1pt,line width=0.1mm] (axis cs:0, 0) rectangle (axis cs:1/2, 1);
						\draw [black, thick, mark size=0.1pt,line width=0.1mm] (axis cs:0, 0) rectangle (axis cs:1, 1/2);
						\draw [black, thick, mark size=0.1pt,line width=0.1mm] (axis cs:0, 1/2) rectangle (axis cs:1/4, 1);
						\addplot[only marks,mark=*, mark size=1.25pt,blue]  coordinates {(1/3, 1/3)};
						\addplot[only marks,mark=*, mark size=1.25pt,blue]  coordinates {(1/6, 1/6)};
						\addplot[only marks,mark=o, mark size=1.25pt,black] coordinates {(5/6, 1/3)};
						\addplot[only marks,mark=o, mark size=1.25pt,black] coordinates {(2/3, 1/6)};
						\addplot[only marks,mark=o, mark size=1.25pt,black] coordinates {(1/3, 5/6)};
						\addplot[only marks,mark=o, mark size=1.25pt,black] coordinates {(5/12, 2/3)};
						\addplot[only marks,mark=*, mark size=1.25pt,blue]  coordinates {(1/12, 5/6)};
						\addplot[only marks,mark=*, mark size=1.25pt,blue]  coordinates {(1/6, 2/3)};
						\addplot[only marks,mark=o, mark size=1.25pt,black] coordinates {(5/6, 5/6)};
						\addplot[only marks,mark=o, mark size=1.25pt,black] coordinates {(2/3, 2/3)};
					\end{axis}
				\end{tikzpicture}
				\begin{tikzpicture}
					\begin{axis}[
						width=0.25\textwidth,height=0.25\textwidth,
						ytick=\empty,
						xtick=\empty,
						enlargelimits=0,
						]
						\addplot[thick,patch,mesh,draw,black,patch type=rectangle,line width=0.1mm] coordinates {(0, 0) (1, 0) (1, 1) (0, 1)};
						\draw [black, thick, mark size=0.1pt,line width=0.1mm] (axis cs:0, 0) rectangle (axis cs:1/2, 1);
						\draw [black, thick, mark size=0.1pt,line width=0.1mm] (axis cs:0, 0) rectangle (axis cs:1, 1/2);
						\draw [black, thick, mark size=0.1pt,line width=0.1mm] (axis cs:0, 0) rectangle (axis cs:1/4, 1);
						\draw [black, thick, mark size=0.1pt,line width=0.1mm] (axis cs:0, 1/2) rectangle (axis cs:1/4, 3/4);
						\addplot[only marks,mark=o, mark size=1.25pt,black] coordinates {(1/3, 1/3)};
						\addplot[only marks,mark=o, mark size=1.25pt,black] coordinates {(1/6, 1/6)};
						\addplot[only marks,mark=o, mark size=1.25pt,black] coordinates {(5/6, 1/3)};
						\addplot[only marks,mark=o, mark size=1.25pt,black] coordinates {(2/3, 1/6)};
						\addplot[only marks,mark=o, mark size=1.25pt,black] coordinates {(1/3, 5/6)};
						\addplot[only marks,mark=o, mark size=1.25pt,black] coordinates {(5/12, 2/3)};
						\addplot[only marks,mark=o, mark size=1.25pt,black] coordinates {(1/12, 5/6)};
						\addplot[only marks,mark=o, mark size=1.25pt,black] coordinates {(1/6, 2/3)};
						\addplot[only marks,mark=o, mark size=1.25pt,black] coordinates {(5/6, 5/6)};
						\addplot[only marks,mark=o, mark size=1.25pt,black] coordinates {(2/3, 2/3)};
						\addplot[only marks,mark=o, mark size=1.25pt,black] coordinates {(1/12, 1/3)};
						\addplot[only marks,mark=o, mark size=1.25pt,black] coordinates {(5/12, 1/6)};
						\addplot[only marks,mark=o, mark size=1.25pt,black] coordinates {(1/12, 7/12)};
						\addplot[only marks,mark=o, mark size=1.25pt,black] coordinates {(1/6, 11/12)};
					\end{axis}
			\end{tikzpicture}} \\
			\bottomrule
		\end{tabular}
	\end{minipage}
\end{table}

\subsubsection{Hyper-rectangular partitioning based on \textbf{1-d}imensional \textbf{t}risection and sampling on \textbf{d}iagonal \textbf{v}ertices}
Adaptive diagonal curves (\adc) based algorithm was proposed in \cite{Sergeyev2006}.
Independently of the dimension, \adc{} evaluates the objective function $f(\mathbf{x})$ on two vertices of the main diagonals.
By sampling two points per hyper-rectangle, such a partitioning scheme reduces the chance that the algorithm samples two bad points in the same hyper-rectangle containing an optimal solution.
Thus, better performance could be expected, especially on more complex problems.
Moreover, such a scheme has a significant advantage over center sampling methods when most solution coordinates are located on the boundaries~\cite{Stripinis2021b}.
As in the revised version of \direct{} \cite{Jones2001}, each selected POH is trisected along just one of the longest sides.

\subsubsection{Simplicial partitioning based on \textbf{1-d}imensional \textbf{t}risection/\textbf{b}isection and sampling at center/vertices}
In \texttt{DISIMPL}~\cite{Paulavicius2013:jogo}, simplicial partitions are considered instead of hyper-rectangles.
At the initialization step, the unit hyper-rectangle $\bar{D}$ is partitioned into $n!$ simplices by the standard face-to-face simplicial division based on combinatorial vertex triangulation~\cite{Paulavicius2013:jogo}.
After this, all simplices share the diagonal of the feasible region and have equal hyper-volume.
In~\cite{Paulavicius2013:jogo}, two different sampling and partitioning strategies were proposed: i) evaluating the objective function at the geometric center point of the simplex and trisecting them (1-DTCS); ii) evaluating the objective function on all unique vertices of the simplex and bisecting them (1-DBVS).
While simplicial partitions are very promising for symmetric~\cite{Paulavicius2013:jogo} and problems with linear constraints~\cite{Paulavicius2016:ol} for box-constrained, they are less appealing as the number of initial simplices increases speedily with the number of dimensions.

\subsubsection{Hyper-rectangular partitioning based on \textbf{1-d}imensional \textbf{b}isection and sampling at two diagonal points}
One of the most recent proposals, \birect{} (\texttt{BI}secting \texttt{RECT}angles)~\cite{Paulavicius2016:jogo}, is also motivated by the diagonal partitioning strategy~\cite{Sergeyev2006,Sergeyev2008:book,Sergeyev2017:book}.
However, the objective function is evaluated at two points on the diagonal that are equidistant between themselves and the vertices of the diagonal.
Such a sampling strategy enables the reuse of the sampling points in descendant hyper-rectangles.
Moreover, the bisection is used instead of the typical trisection for diagonal-based algorithms and most \direct-type algorithms.


\section{Description of the \halrect{} algorithm}\label{sec:halrect}

Unlike most \direct-type algorithms based on central sampling combined with trisection, \halrect{} (HALving RECTangles) is based on a unique multi-point sampling technique combined with a halving (bisection).
We first give a high-level illustration of the sampling and partitioning techniques used in the \halrect{} algorithm.
We illustrate them on a binary tree (see \Cref{fig:flowchart_halrect}).
Note that the value of the function was evaluated at more than one sampling point at each POH (except the initial hyper-rectangle).
The experimental part shows that much more comprehensive information about the objective function over hyper-rectangles can be exploited efficiently and positively impact the algorithm's performance.
In contrast to the authors of the original \direct, who proposed trisection, bisection combined with central sampling can also be a very efficient combination.
In the following subsections, we detail the main steps of the \halrect{} algorithm.

\begin{figure}
	\centering
	\resizebox{0.85\textwidth}{!}{
		\begin{tikzpicture}
			\node at (0,1.5) {Initialization};
			\node[draw,	minimum width=2cm, minimum height=2cm, line width=1pt] at (0,0) {};
			\filldraw [black] (0,0) circle (2pt);
			\draw[dashed, line width=1pt] (0,1) -- (0,-1);
			\node[draw,	minimum width=1cm, minimum height=2cm, line width=1pt] at (-3,-2) {};
			\filldraw [black] (-3,-2) circle (2pt);
			\filldraw [black] (-2.5,-2) circle (2pt);
			\draw[dashed, line width=1pt] (-3.5,-2) -- (-2.5,-2);
			\node[draw,	minimum width=1cm, minimum height=2cm, line width=1pt] at (3,-2) {};
			\filldraw [black] (3,-2) circle (2pt);
			\filldraw [black] (2.5,-2) circle (2pt);
			\draw[dashed, line width=1pt] (3.5,-2) -- (2.5,-2);
			\draw[dashed, line width=1pt] (-5,-4.5) -- (-5,-3.5);
			\node[draw,	minimum width=1cm, minimum height=1cm, line width=1pt] at (-5,-4) {};
			\node[draw,	minimum width=0.5cm, minimum height=1cm, line width=1pt] at (-5.25,-6) {};
			\node[draw,	minimum width=0.5cm, minimum height=1cm, line width=1pt] at (-2.75,-6) {};
			\node[draw,	minimum width=1cm, minimum height=1cm, line width=1pt] at (-1,-4) {};
			\node[draw,	minimum width=1cm, minimum height=1cm, line width=1pt] at (5,-4) {};
			\node[draw,	minimum width=1cm, minimum height=1cm, line width=1pt] at (1,-4) {};
			\filldraw [black] (-5.25,-6) circle (2pt);
			\filldraw [black] (-5,-6) circle (2pt);
			\filldraw [black] (-5,-5.5) circle (2pt);
			\filldraw [black] (-2.75,-6) circle (2pt);
			\filldraw [black] (-3,-6) circle (2pt);
			\filldraw [black] (-3,-5.5) circle (2pt);
			\filldraw [black] (-2.5,-5.5) circle (2pt);
			\filldraw [black] (-5,-4) circle (2pt);
			\filldraw [black] (-5,-3.5) circle (2pt);
			\filldraw [black] (-4.5,-3.5) circle (2pt);
			\filldraw [black] (-1,-4) circle (2pt);
			\filldraw [black] (-0.5,-4.5) circle (2pt);
			\filldraw [black] (-1,-4.5) circle (2pt);
			\filldraw [black] (5,-4) circle (2pt);
			\filldraw [black] (5,-4.5) circle (2pt);
			\filldraw [black] (4.5,-4.5) circle (2pt);
			\filldraw [black] (1,-4) circle (2pt);
			\filldraw [black] (0.5,-3.5) circle (2pt);
			\filldraw [black] (1,-3.5) circle (2pt);
			\draw[line width=1pt, ->] (-1,0) -- node[anchor=south,yshift=-0.15cm] {\rotatebox{26}{Iteration 1}} (-2.9,-0.9);
			\draw[line width=1pt, ->] (1,0)  -- node[anchor=south,yshift=-0.15cm] {\rotatebox{-26}{Iteration 1}} (2.9,-0.9);
			\draw[line width=1pt, ->] (-2.5,-1.5) -- node[anchor=south,xshift=0.15cm,yshift=-0.45cm] {\rotatebox{-53}{Iteration 2}} (-1.1,-3.4);
			\draw[line width=1pt, ->] (-3.5,-2.5) -- node[anchor=south,xshift=-0.1cm,yshift=-0.3cm] {\rotatebox{33}{Iteration 2}} (-4.9,-3.4);
			\draw[line width=1pt, ->] (3.5,-1.5) -- node[anchor=south,xshift=0.15cm,yshift=-0.4cm] {\rotatebox{-53}{Iteration 3}} (4.9,-3.4);
			\draw[line width=1pt, ->] (2.5,-2.5) -- node[anchor=south,xshift=-0.15cm,yshift=-0.35cm] {\rotatebox{33}{Iteration 3}} (1.1,-3.4);
			\draw[line width=1pt, ->] (-4.5,-4) -- node[anchor=south,xshift=0.2cm,yshift=-0.4cm] {\rotatebox{-38}{Iteration 3}} (-2.65,-5.4);
			\draw[line width=1pt, ->] (-5.25,-4.5) -- node[anchor=south,xshift=-0.6cm,yshift=-1cm] {\rotatebox{90}{Iteration 3}} (-5.25,-5.4);
	\end{tikzpicture}}
	\caption{Sampling and partitioning techniques used in the \halrect{} algorithm illustrated as a binary tree}
	\label{fig:flowchart_halrect}
\end{figure}
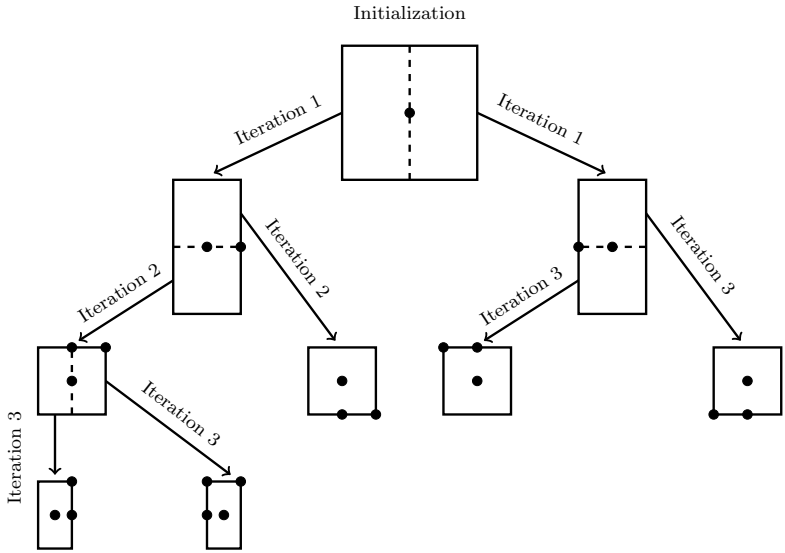

\subsection{Initialization phase}\label{sec:halrectinit}
Like others, the \halrect{} algorithm begins by scaling the feasible region $D$ to an $n$-dimensional unit hyper-rectangle $\bar{D}_0^1$.
It only refers to the initial space $D$ when evaluating the objective function $f(\mathbf{x})$.
The selection of POH in the initialization phase is trivial, as only one candidate is available.
However, in the subsequent iterations, the selection of  POHs is not trivial, and \Cref{ssec:selection-strategies} is devoted to formalizing this.

\subsection{Partitioning and sampling scheme} \label{partitioning}

Like other \direct-type algorithms, \halrect{} samples and evaluates the objective function at midpoints (in the initial phase at $\mathbf{c}^1 \in \bar{D}_0^1 $).
However, unlike most other algorithms, \halrect{} uses bisection instead of trisection.
As a result, midpoints, after bisection, shift in different facets of the hyper-rectangle.
Moreover, all these sampling points can be involved in the POHs selection process.
This way, more detailed information about each hyper-rectangle is considered.
\Cref{fig:overall} illustrates the selection, sampling, and subdivision procedures in the initialization and the subsequent first two iterations of \halrect{} for two- and three-dimensional test problems.

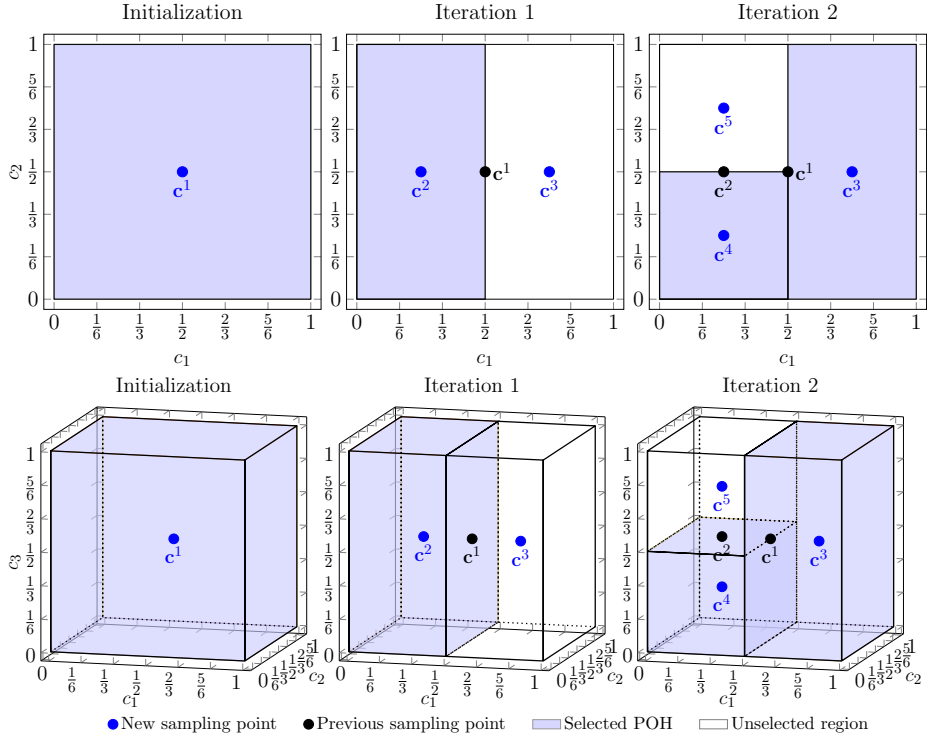
\begin{figure}[ht]
	\resizebox{\textwidth}{!}{
		\begin{tikzpicture}
			\begin{groupplot}[
				group style={
					group size=3 by 1,
					x descriptions at=edge bottom,
					y descriptions at=edge left,
					vertical sep=0pt,
					horizontal sep=15pt,
				},
				height=0.6\textwidth,width=0.6\textwidth,
				]
				\nextgroupplot[
				xlabel = {$c_1$},
				ylabel = {$c_2$},
				enlargelimits=0.04,
				title={Initialization},
				ylabel style={yshift=-0.1cm},
				xlabel style={yshift=-0.1cm},
				ytick distance=1/6,
				xtick distance=1/6,
				every axis/.append style={font=\LARGE},
				yticklabels={$0$, $0$,$\frac{1}{6}$, $\frac{1}{3}$, $\frac{1}{2}$, $\frac{2}{3}$, $\frac{5}{6}$, $1$},
				xticklabels={$0$, $0$,$\frac{1}{6}$, $\frac{1}{3}$, $\frac{1}{2}$, $\frac{2}{3}$, $\frac{5}{6}$, $1$},
				]
				\draw [black, thick, mark size=0.1pt,fill=blue!40,opacity=0.4] (axis cs:0,0) rectangle (axis cs:1,1);
				\addplot[only marks,mark=*,mark size=3pt,color=blue] coordinates {(0.5, 0.5) } node[yshift=-10pt] {$\mathbf{c}^1$} ;
				\addplot[patch,mesh,draw,black,patch type=rectangle] coordinates {
					(0,0) (1,0) (1,1) (0,1)};
				\nextgroupplot[
				xlabel = {$c_1$},
				enlargelimits=0.04,
				title={Iteration $1$},
				ylabel style={yshift=-0.6cm},
				xlabel style={yshift=-0.1cm},
				ytick distance=1/6,
				xtick distance=1/6,
				every axis/.append style={font=\LARGE},
				yticklabels={$0$, $0$,$\frac{1}{6}$, $\frac{1}{3}$, $\frac{1}{2}$, $\frac{2}{3}$, $\frac{5}{6}$, $1$},
				xticklabels={$0$, $0$,$\frac{1}{6}$, $\frac{1}{3}$, $\frac{1}{2}$, $\frac{2}{3}$, $\frac{5}{6}$, $1$},
				]
				\draw [black, thick, mark size=0.1pt,fill=blue!40,opacity=0.4] (axis cs:0,0) rectangle (axis cs:0.5,1);
				\addplot[only marks,mark=*,mark size=3pt,color=black] coordinates {(0.5, 0.5) } node[xshift=10pt] {$\mathbf{c}^1$} ;
				\addplot[only marks,mark=*,mark size=3pt,color=blue] coordinates {(0.25, 0.5) } node[yshift=-10pt] {$\mathbf{c}^2$} ;
				\addplot[only marks,mark=*,mark size=3pt,color=blue] coordinates {(0.75, 0.5) } node[yshift=-10pt] {$\mathbf{c}^3$} ;
				\addplot[patch,mesh,draw,black,patch type=rectangle] coordinates {
					(0,0) (1,0) (1,1) (0,1)
					(0,0) (0.5,0) (0.5,1) (0,1)};
				\nextgroupplot[
				xlabel = {$c_1$},
				enlargelimits=0.04,
				title={Iteration $2$},
				ylabel style={yshift=-0.6cm},
				xlabel style={yshift=-0.1cm},
				ytick distance=1/6,
				xtick distance=1/6,
				every axis/.append style={font=\LARGE},
				yticklabels={$0$, $0$,$\frac{1}{6}$, $\frac{1}{3}$, $\frac{1}{2}$, $\frac{2}{3}$, $\frac{5}{6}$, $1$},
				xticklabels={$0$, $0$,$\frac{1}{6}$, $\frac{1}{3}$, $\frac{1}{2}$, $\frac{2}{3}$, $\frac{5}{6}$, $1$},
				]
				\draw [black, thick, mark size=0.1pt,fill=blue!40,opacity=0.4] (axis cs:1,0) rectangle (axis cs:0.5,1);
				\draw [black, thick, mark size=0.1pt,fill=blue!40,opacity=0.4] (axis cs:0,0) rectangle (axis cs:0.5,0.5);
				\addplot[only marks,mark=*,mark size=3pt,color=black] coordinates {(0.5, 0.5) } node[xshift=10pt] {$\mathbf{c}^1$} ;
				\addplot[only marks,mark=*,mark size=3pt,color=black] coordinates {(0.25, 0.5) } node[yshift=-10pt] {$\mathbf{c}^2$} ;
				\addplot[only marks,mark=*,mark size=3pt,color=blue] coordinates {(0.75, 0.5) } node[yshift=-10pt] {$\mathbf{c}^3$} ;
				\addplot[only marks,mark=*,mark size=3pt,color=blue] coordinates {(0.25, 0.25) } node[yshift=-10pt] {$\mathbf{c}^4$} ;
				\addplot[only marks,mark=*,mark size=3pt,color=blue] coordinates {(0.25, 0.75) } node[yshift=-10pt] {$\mathbf{c}^5$} ;
				\addplot[patch,mesh,draw,black,patch type=rectangle] coordinates {
					(0,0) (1,0) (1,1) (0,1)
					(0,0) (0.5,0) (0.5,1) (0,1)
					(0,0) (0.5,0) (0.5,0.5) (0,0.5)};
			\end{groupplot}
	\end{tikzpicture}}
	\resizebox{\textwidth}{!}{
		\begin{tikzpicture}
			\begin{groupplot}[
				group style={
					group size=3 by 1,
					x descriptions at=edge bottom,
					vertical sep=0pt,
					horizontal sep=20pt,
				},
				height=0.6\textwidth,width=0.6\textwidth,
				]
				\nextgroupplot[
				view={15}{10},
				title={Initialization},
				xlabel = {$c_1$},
				ylabel = {$c_2$},
				zlabel = {$c_3$},
				enlargelimits=0.04,
				xlabel style={xshift=0.2cm,yshift=0.2cm},
				ylabel style={yshift=0.3cm},
				zlabel style={yshift=-0.2cm},
				ytick distance=1/6,
				xtick distance=1/6,
				ztick distance=1/6,
				every axis/.append style={font=\LARGE},
				yticklabels={$0$, $0$,$\frac{1}{6}$, $\frac{1}{3}$, $\frac{1}{2}$, $\frac{2}{3}$, $\frac{5}{6}$, $1$},
				xticklabels={$0$, $0$,$\frac{1}{6}$, $\frac{1}{3}$, $\frac{1}{2}$, $\frac{2}{3}$, $\frac{5}{6}$, $1$},
				zticklabels={$0$, $0$,$\frac{1}{6}$, $\frac{1}{3}$, $\frac{1}{2}$, $\frac{2}{3}$, $\frac{5}{6}$, $1$},
				]
				\addplot3[patch,patch type=rectangle,fill=blue!20,opacity=0.4] coordinates {
					(0,0,0) (1,0,0) (1,1,0) (0,1,0)
					(0,0,0) (1,0,0) (1,0,1) (0,0,1)
					(1,0,0) (1,1,0) (1,1,1) (1,0,1)
					(1,1,0) (1,1,1) (0,1,1) (0,1,0)
					(0,1,0) (0,1,1) (0,0,1) (0,0,0)
					(0,0,1) (1,0,1) (1,1,1) (0,1,1)
				};
				\addplot3[style=thick,mark=none,color=black] coordinates {(0,0,0) (0,0,1) (0,1,1) (1,1,1) (1,1,0) (1,0,0) (0,0,0) (0,0,1) (1,0,1) (1,0,0) (1,0,1) (1,1,1)};
				\addplot3[style=dotted,style=thick,mark=none,color=black] coordinates {(0,0,0) (0,1,0) (0,1,1) (0,1,0) (1,1,0)};
				\addplot3[only marks,mark=*,mark size=3pt,color=blue] coordinates {(0.5, 0.5, 0.5) } node[yshift=-10pt] {$\mathbf{c}^1$};
				\nextgroupplot[
				view={15}{10},
				title={Iteration $1$},
				xlabel = {$c_1$},
				ylabel = {$c_2$},
				enlargelimits=0.04,
				xlabel style={xshift=0.2cm,yshift=0.2cm},
				ylabel style={yshift=0.3cm},
				zlabel style={yshift=-0.2cm},
				ytick distance=1/6,
				xtick distance=1/6,
				ztick distance=1/6,
				every axis/.append style={font=\LARGE},
				yticklabels={$0$, $0$,$\frac{1}{6}$, $\frac{1}{3}$, $\frac{1}{2}$, $\frac{2}{3}$, $\frac{5}{6}$, $1$},
				xticklabels={$0$, $0$,$\frac{1}{6}$, $\frac{1}{3}$, $\frac{1}{2}$, $\frac{2}{3}$, $\frac{5}{6}$, $1$},
				zticklabels={$0$, $0$,$\frac{1}{6}$, $\frac{1}{3}$, $\frac{1}{2}$, $\frac{2}{3}$, $\frac{5}{6}$, $1$},
				]
				\addplot3[patch,patch type=rectangle,fill=blue!20,opacity=0.4] coordinates {
					(0,0,0) (0.5,0,0) (0.5,1,0) (0,1,0)
					(0,0,0) (0.5,0,0) (0.5,0,1) (0,0,1)
					(0.5,0,0) (0.5,1,0) (0.5,1,1) (0.5,0,1)
					(0.5,1,0) (0.5,1,1) (0,1,1) (0,1,0)
					(0,1,0) (0,1,1) (0,0,1) (0,0,0)
					(0,0,1) (0.5,0,1) (0.5,1,1) (0,1,1)
				};
				\addplot3[style=thick,mark=none,color=black] coordinates {
					(0,0,0)   (0,0,1)   (0,1,1) (0.5,1,1)
					(0.5,0,1) (0.5,0,0) (0.5,0,1) (0.5,1,1)
					(1,1,1)   (1,1,0)   (1,0,0)   (0,0,0)
					(0,0,1)   (1,0,1)   (1,0,0)   (1,0,1)
					(1,1,1)};
				\addplot3[style=dotted,style=thick,mark=none,color=black] coordinates {
					(0,0,0)   (0,1,0)   (0,1,1)   (0,1,0)
					(0.5,1,0) (0.5,0,0) (0.5,1,0) (0.5,1,1)
					(0.5,1,0) (1,1,0)};
				\addplot3[only marks,mark=*,mark size=3pt,color=black] coordinates {(0.5, 0.5, 0.5) } node[yshift=-10pt] {$\mathbf{c}^1$};
				\addplot3[only marks,mark=*,mark size=3pt,color=blue] coordinates {(0.25, 0.5, 0.5) } node[yshift=-10pt] {$\mathbf{c}^2$};
				\addplot3[only marks,mark=*,mark size=3pt,color=blue] coordinates {(0.75, 0.5, 0.5) } node[yshift=-10pt] {$\mathbf{c}^3$};
				\nextgroupplot[
				view={15}{10},
				title={Iteration $2$},
				xlabel = {$c_1$},
				ylabel = {$c_2$},
				enlargelimits=0.04,
				xlabel style={xshift=0.2cm,yshift=0.2cm},
				ylabel style={yshift=0.3cm},
				zlabel style={yshift=-0.2cm},
				ytick distance=1/6,
				xtick distance=1/6,
				ztick distance=1/6,
				every axis/.append style={font=\LARGE},
				yticklabels={$0$, $0$,$\frac{1}{6}$, $\frac{1}{3}$, $\frac{1}{2}$, $\frac{2}{3}$, $\frac{5}{6}$, $1$},
				xticklabels={$0$, $0$,$\frac{1}{6}$, $\frac{1}{3}$, $\frac{1}{2}$, $\frac{2}{3}$, $\frac{5}{6}$, $1$},
				zticklabels={$0$, $0$,$\frac{1}{6}$, $\frac{1}{3}$, $\frac{1}{2}$, $\frac{2}{3}$, $\frac{5}{6}$, $1$},
				legend style={draw=none},
				legend columns=4,
				legend style={at={(0.9,-0.15)},font=\large},
				legend style={/tikz/every even column/.append style={column sep=0.5cm}}]
				]
				\addlegendimage{only marks,mark=*,mark size=3pt,color=blue}
				\addlegendentry{New sampling point}
				\addlegendimage{only marks,mark=*,mark size=3pt,color=black}
				\addlegendentry{Previous sampling point}
				\addlegendimage{area legend,fill=blue!40,opacity=0.4}
				\addlegendentry{Selected POH}
				\addlegendimage{area legend,black,fill=white,opacity=0.5}
				\addlegendentry{Unselected region}
				\addplot3[patch,patch type=rectangle,fill=blue!20,opacity=0.4] coordinates {
					(0,0,0) (0.5,0,0) (0.5,1,0) (0,1,0)
					(0,0,0) (0.5,0,0) (0.5,0,0.5) (0,0,0.5)
					(0.5,0,0) (0.5,1,0) (0.5,1,0.5) (0.5,0,0.5)
					(0.5,1,0) (0.5,1,0.5) (0,1,0.5) (0,1,0)
					(0,1,0) (0,1,0.5) (0,0,0.5) (0,0,0)
					(0,0,0.5) (0.5,0,0.5) (0.5,1,0.5) (0,1,0.5)
				};
				\addplot3[patch,patch type=rectangle,fill=blue!20,opacity=0.4] coordinates {
					(0.5,0,0) (1,0,0) (1,1,0) (0.5,1,0)
					(0.5,0,0) (1,0,0) (1,0,1) (0.5,0,1)
					(1,0,0) (1,1,0) (1,1,1) (1,0,1)
					(1,1,0) (1,1,1) (0.5,1,1) (0.5,1,0)
					(0.5,1,0) (0.5,1,1) (0.5,0,1) (0.5,0,0)
					(0.5,0,1) (1,0,1) (1,1,1) (0.5,1,1)
				};
				\addplot3[style=thick,mark=none,color=black] coordinates {
					(0,0,0)   (0,0,1)     (0,1,1)   (0.5,1,1)
					(0.5,0,1) (0.5,0,0.5) (0,0,0.5) (0.5,0,0.5)
					(0.5,0,0) (0.5,0,1)   (0.5,1,1) (1,1,1)
					(1,1,0)   (1,0,0)     (0,0,0)   (0,0,1)
					(1,0,1)   (1,0,0)     (1,0,1)   (1,1,1)};
				\addplot3[style=dotted,style=thick,mark=none,color=black] coordinates {
					(0,0,0)     (0,1,0)   (0,1,1)     (0,1,0.5)
					(0,0,0.5)   (0,1,0.5) (0.5,1,0.5) (0.5,0,0.5)
					(0.5,1,0.5) (0,1,0.5) (0,1,0)     (0.5,1,0)
					(0.5,0,0)   (0.5,1,0) (0.5,1,1)   (0.5,1,0)
					(1,1,0)};
				
				\addplot3[only marks,mark=*,mark size=3pt,color=black] coordinates {(0.5, 0.5, 0.5) } node[yshift=-10pt] {$\mathbf{c}^1$};
				\addplot3[only marks,mark=*,mark size=3pt,color=black] coordinates {(0.25, 0.5, 0.5) } node[yshift=-10pt] {$\mathbf{c}^2$};
				\addplot3[only marks,mark=*,mark size=3pt,color=blue] coordinates {(0.75, 0.5, 0.5) } node[yshift=-10pt] {$\mathbf{c}^3$};
				\addplot3[only marks,mark=*,mark size=3pt,color=blue] coordinates {(0.25, 0.5, 0.25) } node[yshift=-10pt] {$\mathbf{c}^4$};
				\addplot3[only marks,mark=*,mark size=3pt,color=blue] coordinates {(0.25, 0.5, 0.75) } node[yshift=-10pt] {$\mathbf{c}^5$};
				
			\end{groupplot}
	\end{tikzpicture}}
	\caption{Illustration of selection, sampling and partitioning schemes used in the \halrect{} algorithm on two-dimensional (upper part) and three-dimensional (lower part) test problems.}
	\label{fig:overall}
\end{figure}

Now let us formalize the sampling and partitioning schemes used in \halrect.
In iteration $k$, the current partition ($\mathcal{P}_k$) and hyper-rectangle $(\bar{D}^i_k)$ are defined as in \cref{eq:partition_set,eq:rectangle}, where $ \indexsett_k$ is the index set of the current partition.
Additionally, for each hyper-rectangle, we define the representative sampling index set $\mathbb{H}^i_k$ storing the indices $(i)$ of all sampled points $(\mathbf{c}^i)$ within the hyper-rectangle at which the objective function has already been evaluated.
We note that initially sampled midpoints, after subdivision (bisection), change their position and later are located on facets of hyper-rectangles (see \Cref{fig:overall}).

Using these notations, at the initial $(k=0)$ and the first two iterations, the current partition ($\mathcal{P}_k$) and the representative sampling index sets $(\mathbb{H}^i_k)$ are
\begin{align*}
	\mathcal{P}_0 &= \{ \bar{D}^1_0 \},  \mathbb{H}^1_0 = \{ 1 \},\\
	\mathcal{P}_1 &= \{ \bar{D}^2_1,  \bar{D}^3_1\},  \mathbb{H}^2_1 = \{ 1, 2 \},  \mathbb{H}^3_1 = \{ 1, 3 \},\\
	\mathcal{P}_2 &= \{ \bar{D}^3_2,  \bar{D}^4_2, \bar{D}^5_2 \},
	\mathbb{H}^3_2 = \{ 1, 3 \}, \mathbb{H}^4_2 = \{ 1, 2, 4 \}, \mathbb{H}^5_2 = \{ 1, 2, 5 \}.
\end{align*}

Selected POHs (\cref{ssec:selection-strategies} describes the selection process) are bisected only along one coordinate with the maximum side length.
\Cref{alg:branching-coordinate} describes the procedure used in \halrect{} to select the branching variable, i.e., coordinate index $(br \in \{1,...,n \})$.

\begin{algorithm}[ht]
	\caption{Branching coordinate index selection}\label{alg:branching-coordinate}
	\begin{algorithmic}[1]
		\Require {Selected POH $(\bar{D}^i_k)$, new sampling point $(\mathbf{c}^{i} \in \bar{D}^i_k)$, current minimum point $(\mathbf{c}^{\rm min})$; }
		\Ensure {Branching coordinate index $(br)$ ;}
		\algrule
		\State Find all the longest sides (indices of corresponding coordinates)
		\begin{equation}
			\lambda_1 = \argmax_{j=1,...,n} \left\{ d^i_j = \mid \bar{b}^i_j - \bar{a}^i_j \mid \right\}; \algorithmiccomment{\text{See~\cref{eq:rectangle}}}
		\end{equation}
		\State Find the furthest coordinate(s) from~$\mathbf{c}^{i}$ to $\mathbf{c}^{\rm min}$
		\begin{equation}
			\lambda_2 = \argmax_{j \in \lambda_1} \left\{ \mid c^i_{j} - c^{\rm min}_{j} \mid \right\};
		\end{equation}
		\State Select the coordinate with the smallest index
		\begin{equation}
			br = \min_{j \in \lambda_2}{j}.
		\end{equation}
		\textbf{Return} $br$.
	\end{algorithmic}
\end{algorithm}

\begin{example}
	In \Cref{fig:example}, an illustration of branching variable selection is given in the \halrect{} algorithm moving from the second to the third iteration.
	In the second iteration ($k = 2$), there are two POHs ($\bar{D}^3_2$ and $\bar{D}^4_2$).
	For $\bar{D}^3_2$ there is only one longest side (coordinate $j = 2$ with the side length $d^3_2 = 1$), therefore \Cref{alg:branching-coordinate} returns $br = 2$.
	However, for $\bar{D}^3_2$, at Step 1 of \Cref{alg:branching-coordinate}, both sides are equal, and therefore $\lambda_1 = \{ 1, 2\}$.
	Since the midpoint $\mathbf{c}^{4}$ is also a current minimum point $(\mathbf{c}^{\rm min})$, after Step 2, the set $\lambda_2 = \{ 1, 2\}$.
	Finally, the coordinate with the smallest index value ($br = 1$) is selected in the third step and returned.
	
	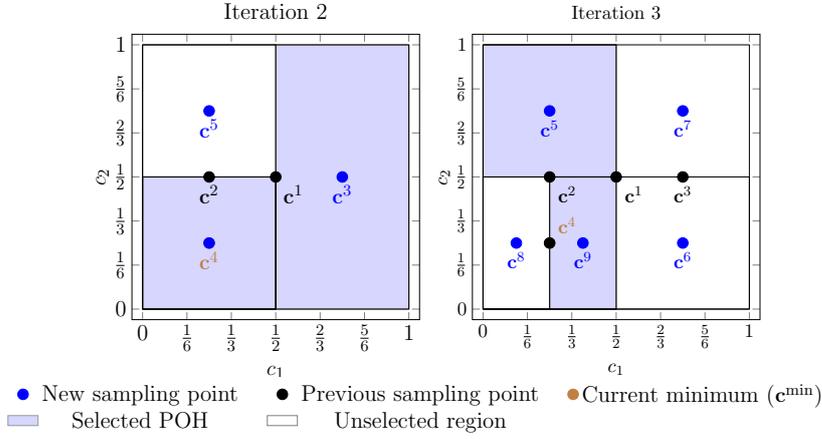
\begin{figure}[ht]
		\centering
		\resizebox{0.9\textwidth}{!}{
			\begin{tikzpicture}
				\begin{groupplot}[
					group style={
						group size=2 by 1,
						x descriptions at=edge bottom,
						y descriptions at=edge left,
						vertical sep=0pt,
						horizontal sep=30pt,
					},
					height=0.6\textwidth,width=0.6\textwidth,
					]
					\nextgroupplot[
					xlabel = {$c_1$},
					ylabel = {$c_2$},
					enlargelimits=0.04,
					title={Iteration $2$},
					ylabel style={yshift=-0.1cm},
					xlabel style={yshift=-0.1cm},
					ytick distance=1/6,
					xtick distance=1/6,
					every axis/.append style={font=\LARGE},
					yticklabels={$0$, $0$,$\frac{1}{6}$, $\frac{1}{3}$, $\frac{1}{2}$, $\frac{2}{3}$, $\frac{5}{6}$, $1$},
					xticklabels={$0$, $0$,$\frac{1}{6}$, $\frac{1}{3}$, $\frac{1}{2}$, $\frac{2}{3}$, $\frac{5}{6}$, $1$},
					]
					\draw [black, thick, mark size=0.1pt,fill=blue!40,opacity=0.4] (axis cs:1,0) rectangle (axis cs:0.5,1);
					\draw [black, thick, mark size=0.1pt,fill=blue!40,opacity=0.4] (axis cs:0,0) rectangle (axis cs:0.5,0.5);
					\addplot[only marks,mark=*,mark size=3pt,color=black] coordinates {(0.5, 0.5) } node[yshift=-10pt,xshift=10pt] {$\mathbf{c}^1$};
					\addplot[only marks,mark=*,mark size=3pt,color=black] coordinates {(0.25, 0.5) } node[yshift=-10pt] {$\mathbf{c}^{2}$} ;
					\addplot[only marks,mark=*,mark size=3pt,color=blue] coordinates {(0.75, 0.5) } node[yshift=-10pt] {$\mathbf{c}^3$} ;
					\addplot[only marks,mark=*,mark size=3pt,color=brown] coordinates {(0.25, 0.25) } node[yshift=-10pt] {$\mathbf{c}^4$} ;
					\addplot[only marks,mark=*,mark size=3pt,color=blue] coordinates {(0.25, 0.25) } ;
					\addplot[only marks,mark=*,mark size=3pt,color=blue] coordinates {(0.25, 0.75) } node[yshift=-10pt] {$\mathbf{c}^5$} ;
					\addplot[patch,mesh,draw,black,patch type=rectangle] coordinates {
						(0,0) (1,0) (1,1) (0,1)
						(0,0) (0.5,0) (0.5,1) (0,1)
						(0,0) (0.5,0) (0.5,0.5) (0,0.5)};
					\nextgroupplot[
					xlabel = {$c_1$},
					ylabel = {$c_2$},
					enlargelimits=0.04,
					title={Iteration $3$},
					ylabel style={yshift=-0.1cm},
					xlabel style={yshift=-0.1cm},
					ytick distance=1/6,
					xtick distance=1/6,
					every axis/.append style={font=\large},
					yticklabels={$0$, $0$,$\frac{1}{6}$, $\frac{1}{3}$, $\frac{1}{2}$, $\frac{2}{3}$, $\frac{5}{6}$, $1$},
					xticklabels={$0$, $0$,$\frac{1}{6}$, $\frac{1}{3}$, $\frac{1}{2}$, $\frac{2}{3}$, $\frac{5}{6}$, $1$},
					legend style={draw=none},
					legend columns=3,
					legend style={at={(1.25,-0.2)},font=\LARGE},
					legend style={/tikz/every even column/.append style={column sep=0.5cm}}
					]
					\addlegendimage{only marks,mark=*,mark size=3pt,color=blue}
					\addlegendentry{New sampling point}
					\addlegendimage{only marks,mark=*,mark size=3pt,color=black}
					\addlegendentry{Previous sampling point}
					\addlegendimage{only marks,mark=*,mark size=3pt,color=brown}
					\addlegendentry{Current minimum $(\mathbf{c}^{\rm min})$}
					\addlegendimage{area legend,fill=blue!40,opacity=0.4}
					\addlegendentry{Selected POH}
					\addlegendimage{area legend,black,fill=white,opacity=0.5}
					\addlegendentry{Unselected region}
					\draw [black, thick, mark size=0.1pt,fill=blue!40,opacity=0.4] (axis cs:0,0.5) rectangle (axis cs:0.5,1);
					\draw [black, thick, mark size=0.1pt,fill=blue!40,opacity=0.4] (axis cs:0.25,0) rectangle (axis cs:0.5,0.5);
					\addplot[only marks,mark=*,mark size=3pt,color=black] coordinates {(0.5, 0.5) }    node[yshift=-10pt, xshift=10pt] {$\mathbf{c}^1$} ;
					\addplot[only marks,mark=*,mark size=3pt,color=black] coordinates {(0.25, 0.5) }   node[yshift=-10pt, xshift=10pt] {$\mathbf{c}^{2}$} ;
					\addplot[only marks,mark=*,mark size=3pt,color=black] coordinates {(0.75, 0.5) }   node[yshift=-10pt] {$\mathbf{c}^3$} ;
					\addplot[only marks,mark=*,mark size=3pt,color=brown] coordinates {(0.25, 0.25) }  node[yshift=10pt, xshift=10pt] {$\mathbf{c}^4$} ;
					\addplot[only marks,mark=*,mark size=3pt,color=black] coordinates {(0.25, 0.25) } ;
					\addplot[only marks,mark=*,mark size=3pt,color=blue]  coordinates {(0.25, 0.75) }  node[yshift=-10pt] {$\mathbf{c}^5$} ;
					\addplot[only marks,mark=*,mark size=3pt,color=blue]  coordinates {(0.75, 0.25) }  node[yshift=-10pt] {$\mathbf{c}^6$} ;
					\addplot[only marks,mark=*,mark size=3pt,color=blue]  coordinates {(0.75, 0.75) }  node[yshift=-10pt] {$\mathbf{c}^7$} ;
					\addplot[only marks,mark=*,mark size=3pt,color=blue]  coordinates {(0.125, 0.25) } node[yshift=-10pt] {$\mathbf{c}^8$} ;
					\addplot[only marks,mark=*,mark size=3pt,color=blue]  coordinates {(0.375, 0.25) } node[yshift=-10pt] {$\mathbf{c}^9$} ;
					\addplot[patch,mesh,draw,black,patch type=rectangle]  coordinates {
						(0,0) (1,0) (1,1) (0,1)
						(0,0) (0.25, 0) (0.25,0.5) (0,0.5)
						(0,0) (0.5,0) (0.5,1) (0,1)
						(0,0) (0,0.5) (1,0.5) (1,0)};
				\end{groupplot}
		\end{tikzpicture}}
		\caption{Illustration of sampling and partitioning schemes used in the \halrect{} algorithm on a two-dimensional example moving from the second to the third iteration.}
		\label{fig:example}
	\end{figure}
\end{example}

When the branching coordinate $(br)$ is identified, each POH $(\bar{D}^i_k)$ is bisected into two equal smaller hyper-rectangles $\bar{D}^{\rm left}_k$ and $\bar{D}^{\rm right}_k$.
The new midpoints ($\mathbf{c}^{\rm left} \in \bar{D}^{\rm left}_k$ and $\mathbf{c}^{\rm right} \in \bar{D}^{\rm right}_k$) are located at the following positions:
\begin{equation}
	\label{eq:leftpoint}
	\mathbf{c}^{\rm left} = ( c_1^i, ..., c_{br}^i - \frac{d^i_{br}}{4} ,...,c_n^i ),
\end{equation}
\begin{equation}
	\label{eq:rightpoint}
	\mathbf{c}^{\rm right} = ( c_1^i, ..., c_{br}^i + \frac{d^i_{br}}{4} ,...,c_n^i ),
\end{equation}
where $\mathbf{c}^i \in \bar{D}^i_k$.
We note that naming new hyper-rectangles and midpoints as the ``left'' and the ``right'' is only relative.

Continuing in the same vein, after bisection of $\bar{D}^3_2$, new midpoints are located at:
\[
\mathbf{c}^{\rm left} = \mathbf{c}^{6} = \left( c_1^3, c_{2}^3 - \frac{d^3_{2}}{4} \right) = \left( \frac{3}{4}, \frac{1}{4} \right),
\]
\[
\mathbf{c}^{\rm right} = \mathbf{c}^{7} = \left( c_1^3, c_{2}^3 + \frac{d^3_{2}}{4} \right) = \left( \frac{3}{4}, \frac{3}{4} \right).
\]
After bisection of $\bar{D}^4_2$, new sampling points are located at:
\[
\mathbf{c}^{\rm left} = \mathbf{c}^{8} = \left( c_{1}^4 - \frac{d^4_{1}}{4}, c_2^4 \right) = \left( \frac{1}{8}, \frac{1}{4} \right),
\]
\[
\mathbf{c}^{\rm right} = \mathbf{c}^{9} = \left( c_{1}^4 + \frac{d^4_{1}}{4}, c_2^4 \right) = \left( \frac{3}{8}, \frac{3}{4} \right).
\]

The illustration of sampled search space  after ten iterations using the \halrect{} algorithm on \textit{Sum\_of\_Powers} function is given in \Cref{fig:sampling}.

\begin{figure}[htbp]
	\centering
	\resizebox{0.9\textwidth}{!}{
		\begin{tikzpicture}
			\begin{axis}[
				width=0.8\textwidth,height=0.8\textwidth,
				xlabel = {$c_1$},
				ylabel = {$c_2$},
				title={2-dimensional case},
				enlargelimits=0.05,
				ylabel style={yshift=-0.1cm},
				xlabel style={yshift=-0.1cm},
				ytick distance=1/6,
				xtick distance=1/6,
				every axis/.append style={font=\LARGE},
				yticklabels={$0$, $0$,$\frac{1}{6}$, $\frac{1}{3}$, $\frac{1}{2}$, $\frac{2}{3}$, $\frac{5}{6}$, $1$},
				xticklabels={$0$, $0$,$\frac{1}{6}$, $\frac{1}{3}$, $\frac{1}{2}$, $\frac{2}{3}$, $\frac{5}{6}$, $1$},
				]
				\addplot[thick,patch,mesh,draw,black,patch type=rectangle,line width=0.25mm] coordinates {(0,0)(1,0)(1,1) (0,1)};
				\addplot[thick,patch,mesh,draw,black,patch type=line,dashed,line width=0.25mm] coordinates {(0.25,0.5)(0.75,0.5)};
				\addplot[thick,patch,mesh,draw,black,patch type=line,dashed,line width=0.25mm] coordinates {(0.25,0.75)(0.25,0.25)};
				\addplot[thick,patch,mesh,draw,black,patch type=line,dashed,line width=0.25mm] coordinates {(0.75,0.75)(0.75,0.25)};
				\addplot[thick,patch,mesh,draw,black,patch type=line,dashed,line width=0.25mm] coordinates {(0.865,0.75)(0.635,0.75)};
				\addplot[thick,patch,mesh,draw,black,patch type=line,dashed,line width=0.25mm] coordinates {(0.865,0.25)(0.635,0.25)};
				\addplot[thick,patch,mesh,draw,black,patch type=line,dashed,line width=0.25mm] coordinates {(0.625,0.365)(0.625,0.135)};
				\addplot[thick,patch,mesh,draw,black,patch type=line,dashed,line width=0.25mm] coordinates {(0.375,0.25)(0.125,0.25)};
				\addplot[thick,patch,mesh,draw,black,patch type=line,dashed,line width=0.25mm] coordinates {(0.365,0.75)(0.135,0.75)};
				\addplot[thick,patch,mesh,draw,black,patch type=line,dashed,line width=0.25mm] coordinates {(0.125,0.635)(0.125,0.865)};
				\addplot[thick,patch,mesh,draw,black,patch type=line,dashed,line width=0.25mm] coordinates {(0.125,0.365)(0.125,0.135)};
				\addplot[thick,patch,mesh,draw,black,patch type=line,dashed,line width=0.25mm] coordinates {(0.375,0.365)(0.375,0.135)};
				\addplot[thick,patch,mesh,draw,black,patch type=line,dashed,line width=0.25mm] coordinates {(0.3225,0.125)(0.4275,0.125)};
				\addplot[thick,patch,mesh,draw,black,patch type=line,dashed,line width=0.25mm] coordinates {(0.0725,0.125)(0.1775,0.125)};
				\addplot[thick,patch,mesh,draw,black,patch type=line,dashed,line width=0.25mm] coordinates {(0.0725,0.375)(0.1775,0.375)};
				\addplot[thick,patch,mesh,draw,black,patch type=line,dashed,line width=0.25mm] coordinates {(0.0625,0.0725)(0.0625,0.1775)};
				\addplot[thick,patch,mesh,draw,black,patch type=line,dashed,line width=0.25mm] coordinates {(0.1875,0.0725)(0.1875,0.1775)};
				\addplot[thick,patch,mesh,draw,black,patch type=line,dashed,line width=0.25mm] coordinates {(0.04125, 0.0625)(0.08375, 0.0625)};
				\addplot[only marks,mark=*,mark size=2pt,color=black] coordinates {(0.5, 0.5)};
				\addplot[only marks,mark=*,mark size=2pt,color=black] coordinates {(0.25, 0.5)};
				\addplot[only marks,mark=*,mark size=2pt,color=black] coordinates {(0.75, 0.5)};
				\addplot[only marks,mark=*,mark size=2pt,color=black] coordinates {(0.25, 0.25)};
				\addplot[only marks,mark=*,mark size=2pt,color=black] coordinates {(0.25, 0.75)};
				\addplot[only marks,mark=*,mark size=2pt,color=black] coordinates {(0.125, 0.25)};
				\addplot[only marks,mark=*,mark size=2pt,color=black] coordinates {(0.375, 0.25)};
				\addplot[only marks,mark=*,mark size=2pt,color=black] coordinates {(0.75, 0.25)};
				\addplot[only marks,mark=*,mark size=2pt,color=black] coordinates {(0.75, 0.75)};
				\addplot[only marks,mark=*,mark size=2pt,color=black] coordinates {(0.125, 0.75)};
				\addplot[only marks,mark=*,mark size=2pt,color=blue] coordinates {(0.375, 0.75)};
				\addplot[only marks,mark=*,mark size=2pt,color=black] coordinates {(0.125, 0.125)};
				\addplot[only marks,mark=*,mark size=2pt,color=black] coordinates {(0.125, 0.375)};
				\addplot[only marks,mark=*,mark size=2pt,color=black] coordinates {(0.625, 0.25)};
				\addplot[only marks,mark=*,mark size=2pt,color=blue] coordinates {(0.875, 0.25)};
				\addplot[only marks,mark=*,mark size=2pt,color=black] coordinates {(0.375, 0.125)};
				\addplot[only marks,mark=*,mark size=2pt,color=blue] coordinates {(0.375, 0.375)};
				\addplot[only marks,mark=*,mark size=2pt,color=black] coordinates {(0.0625, 0.125)};
				\addplot[only marks,mark=*,mark size=2pt,color=black] coordinates {(0.1875, 0.125)};
				\addplot[only marks,mark=*,mark size=2pt,color=blue] coordinates {(0.625, 0.75)};
				\addplot[only marks,mark=*,mark size=2pt,color=blue] coordinates {(0.875, 0.75)};
				\addplot[only marks,mark=*,mark size=2pt,color=blue] coordinates {(0.125, 0.625)};
				\addplot[only marks,mark=*,mark size=2pt,color=blue] coordinates {(0.125, 0.875)};
				\addplot[only marks,mark=*,mark size=2pt,color=blue] coordinates {(0.0625, 0.375)};
				\addplot[only marks,mark=*,mark size=2pt,color=blue] coordinates {(0.1875, 0.375)};
				\addplot[only marks,mark=*,mark size=2pt,color=black] coordinates {(0.0625, 0.0625)};
				\addplot[only marks,mark=*,mark size=2pt,color=blue] coordinates {(0.0625, 0.1875)};
				\addplot[only marks,mark=*,mark size=2pt,color=blue] coordinates {(0.625, 0.125)};
				\addplot[only marks,mark=*,mark size=2pt,color=blue] coordinates {(0.625, 0.375)};
				\addplot[only marks,mark=*,mark size=2pt,color=blue] coordinates {(0.3125, 0.125)};
				\addplot[only marks,mark=*,mark size=2pt,color=blue] coordinates {(0.4375, 0.125)};
				\addplot[only marks,mark=*,mark size=2pt,color=blue] coordinates {(0.1875, 0.0625)};
				\addplot[only marks,mark=*,mark size=2pt,color=blue] coordinates {(0.1875, 0.1875)};
				\addplot[only marks,mark=*,mark size=2pt,color=blue] coordinates {(0.03125, 0.0625)};
				\addplot[only marks,mark=*,mark size=2pt,color=blue] coordinates {(0.09375, 0.0625)};
			\end{axis}
		\end{tikzpicture}
		\begin{tikzpicture}
			\begin{axis}[
				view={15}{10},
				title={3-dimensional case},
				width=0.8\textwidth,height=0.8\textwidth,
				xlabel = {$c_1$},
				ylabel = {$c_2$},
				zlabel = {$c_3$},
				enlargelimits=0.04,
				xlabel style={xshift=0.2cm,yshift=0.2cm},
				ylabel style={yshift=0.3cm},
				zlabel style={yshift=-0.2cm},
				ytick distance=1/6,
				xtick distance=1/6,
				ztick distance=1/6,
				every axis/.append style={font=\LARGE},
				yticklabels={$0$, $0$,$\frac{1}{6}$, $\frac{1}{3}$, $\frac{1}{2}$, $\frac{2}{3}$, $\frac{5}{6}$, $1$},
				xticklabels={$0$, $0$,$\frac{1}{6}$, $\frac{1}{3}$, $\frac{1}{2}$, $\frac{2}{3}$, $\frac{5}{6}$, $1$},
				zticklabels={$0$, $0$,$\frac{1}{6}$, $\frac{1}{3}$, $\frac{1}{2}$, $\frac{2}{3}$, $\frac{5}{6}$, $1$},
				]
				\addplot3[thick,patch,mesh,draw,black,patch type=rectangle,line width=0.4mm] coordinates {
					(0,0,0) (1,0,0) (1,1,0) (0,1,0)
					(0,0,0) (1,0,0) (1,0,1) (0,0,1)
					(1,0,0) (1,1,0) (1,1,1) (1,0,1)
					(1,1,0) (1,1,1) (0,1,1) (0,1,0)
					(0,1,0) (0,1,1) (0,0,1) (0,0,0)
					(0,0,1) (1,0,1) (1,1,1) (0,1,1)
				};
				\addplot3[thick,patch,mesh,draw,black,patch type=line,dashed,line width=0.25mm] coordinates {(0.25,0.5,0.5) (0.75,0.5,0.5)};
				\addplot3[thick,patch,mesh,draw,black,patch type=line,dashed,line width=0.25mm] coordinates {(0.25,0.5,0.25) (0.25,0.5,0.75)};
				\addplot3[thick,patch,mesh,draw,black,patch type=line,dashed,line width=0.25mm] coordinates {(0.25,0.25,0.25) (0.25,0.71,0.25)};
				\addplot3[thick,patch,mesh,draw,black,patch type=line,dashed,line width=0.25mm] coordinates {(0.75,0.25,0.5) (0.75,0.75,0.5)};
				\addplot3[thick,patch,mesh,draw,black,patch type=line,dashed,line width=0.25mm] coordinates {(0.25, 0.25, 0.125) (0.25, 0.25, 0.365)};
				\addplot3[thick,patch,mesh,draw,black,patch type=line,dashed,line width=0.25mm] coordinates {(0.75, 0.25, 0.25) (0.75, 0.25, 0.74)};
				\addplot3[thick,patch,mesh,draw,black,patch type=line,dashed,line width=0.25mm] coordinates {(0.75, 0.75, 0.26) (0.75, 0.75, 0.74)};
				\addplot3[thick,patch,mesh,draw,black,patch type=line,dashed,line width=0.25mm] coordinates {(0.125, 0.25, 0.125) (0.36, 0.25, 0.125)};
				\addplot3[thick,patch,mesh,draw,black,patch type=line,dashed,line width=0.25mm] coordinates {(0.125, 0.125, 0.125) (0.125, 0.335, 0.125)};
				\addplot3[thick,patch,mesh,draw,black,patch type=line,dashed,line width=0.25mm] coordinates {(0.125, 0.125, 0.0625) 	(0.125, 0.125, 0.1775)};
				\addplot3[thick,patch,mesh,draw,black,patch type=line,dashed,line width=0.25mm] coordinates {(0.25, 0.25, 0.75) 	(0.25, 0.71, 0.75)};
				\addplot3[thick,patch,mesh,draw,black,patch type=line,dashed,line width=0.25mm] coordinates {(0.25, 0.25, 0.635) 	(0.25, 0.25, 0.865)};
				\addplot3[thick,patch,mesh,draw,black,patch type=line,dashed,line width=0.25mm] coordinates {(0.125, 0.0625, 0.0625) 	(0.125, 0.1875, 0.0625)};
				\addplot3[thick,patch,mesh,draw,black,patch type=line,dashed,line width=0.25mm] coordinates {(0.0725, 0.0625, 0.0625) 	(0.1775, 0.0625, 0.0625)};
				\addplot3[thick,patch,mesh,draw,black,patch type=line,dashed,line width=0.25mm] coordinates {(0.635, 0.25, 0.25) 	(0.865, 0.25, 0.25)};
				\addplot3[only marks,mark=*,mark size=2pt,color=black] coordinates {(0.5, 0.5, 0.5) };
				\addplot3[only marks,mark=*,mark size=2pt,color=black] coordinates {(0.25, 0.5, 0.5) };
				\addplot3[only marks,mark=*,mark size=2pt,color=black] coordinates {(0.75, 0.5, 0.5) };
				\addplot3[only marks,mark=*,mark size=2pt,color=black] coordinates {(0.25, 0.5, 0.25) };
				\addplot3[only marks,mark=*,mark size=2pt,color=black] coordinates {(0.25, 0.5, 0.75) };
				\addplot3[only marks,mark=*,mark size=2pt,color=black] coordinates {(0.25, 0.25, 0.25) };
				\addplot3[only marks,mark=*,mark size=2pt,color=blue] coordinates {(0.25, 0.75, 0.25) };
				\addplot3[only marks,mark=*,mark size=2pt,color=black] coordinates {(0.75, 0.75, 0.5) };
				\addplot3[only marks,mark=*,mark size=2pt,color=black] coordinates {(0.75, 0.25, 0.5) };
				\addplot3[only marks,mark=*,mark size=2pt,color=black] coordinates {(0.25, 0.25, 0.125) };
				\addplot3[only marks,mark=*,mark size=2pt,color=blue] coordinates {(0.25, 0.25, 0.375) };
				\addplot3[only marks,mark=*,mark size=2pt,color=black] coordinates {(0.75, 0.25, 0.25) };
				\addplot3[only marks,mark=*,mark size=2pt,color=blue] coordinates {(0.75, 0.25, 0.75) };
				\addplot3[only marks,mark=*,mark size=2pt,color=blue] coordinates {(0.75, 0.75, 0.25) };
				\addplot3[only marks,mark=*,mark size=2pt,color=blue] coordinates {(0.75, 0.75, 0.75) };
				\addplot3[only marks,mark=*,mark size=2pt,color=black] coordinates {(0.125, 0.25, 0.125) };
				\addplot3[only marks,mark=*,mark size=2pt,color=blue] coordinates {(0.375, 0.25, 0.125) };
				\addplot3[only marks,mark=*,mark size=2pt,color=black] coordinates {(0.125, 0.125, 0.125) };
				\addplot3[only marks,mark=*,mark size=2pt,color=blue] coordinates {(0.125, 0.375, 0.125) };
				\addplot3[only marks,mark=*,mark size=2pt,color=black] coordinates {(0.125, 0.125, 0.0625) };
				\addplot3[only marks,mark=*,mark size=2pt,color=blue] coordinates {(0.125, 0.125, 0.1875) };
				\addplot3[only marks,mark=*,mark size=2pt,color=black] coordinates {(0.25, 0.25, 0.75) };
				\addplot3[only marks,mark=*,mark size=2pt,color=blue] coordinates {(0.25, 0.75, 0.75) };
				\addplot3[only marks,mark=*,mark size=2pt,color=blue] coordinates {(0.25, 0.25, 0.625) };
				\addplot3[only marks,mark=*,mark size=2pt,color=blue] coordinates {(0.25, 0.25, 0.875) };
				\addplot3[only marks,mark=*,mark size=2pt,color=blue] coordinates {(0.125, 0.0625, 0.0625) };
				\addplot3[only marks,mark=*,mark size=2pt,color=blue] coordinates {(0.125, 0.1875, 0.0625) };
				\addplot3[only marks,mark=*,mark size=2pt,color=blue] coordinates {(0.0625, 0.0625, 0.0625) };
				\addplot3[only marks,mark=*,mark size=2pt,color=blue] coordinates {(0.1875, 0.0625, 0.0625) };
				\addplot3[only marks,mark=*,mark size=2pt,color=blue] coordinates {(0.625, 0.25, 0.25) };
				\addplot3[only marks,mark=*,mark size=2pt,color=blue] coordinates {(0.875, 0.25, 0.25) };
			\end{axis}
	\end{tikzpicture}}
	\caption{The illustration of sampled points after $10$ iterations of the \halrect{} algorithm using two and three-dimensional \textit{Sum\_of\_Powers} functions.}
	\label{fig:sampling}
\end{figure}
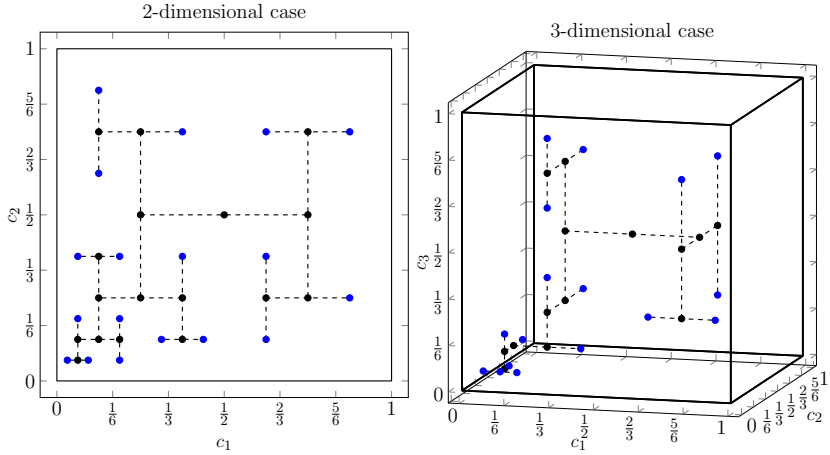

After subdivision, each POH $(\bar{D}^i_k)$ is removed, and two new ones are added to the list that describes the current partition:
\[
\mathcal{P}_{k + 1} = (\mathcal{P}_k \setminus \bar{D}^i_k ) \cup \bar{D}^{\rm left}_k \cup \bar{D}^{\rm right}_k.
\]
Therefore, moving from iteration two to three, hyper-rectangles $\bar{D}^{3}_2$ and $\bar{D}^{4}_2$ are removed from the partition $(\mathcal{P}_{2})$, and new ones are included:
\[
\mathcal{P}_{3} = \{\bar{D}^{5}_3, \bar{D}^{6}_3, \bar{D}^{7}_3, \bar{D}^{8}_3, \bar{D}^{9}_3 \}.
\]
New vectors of the representative index sets $\mathbb{H}^{\rm left}_k$ and $\mathbb{H}^{\rm right}_k$ are constructed based on the set $\mathbb{H}^i_k$ corresponding to the subdivided hyper-rectangle ($\bar{D}^i_k$).
The following rules are used to create them:
\begin{equation}
	\label{eq:leftrep}
	\mathbb{H}^{\rm left}_k = \{h \in \mathbb{H}^i_k : c_{br}^i \geq c_{br}^h \} \cup \{{\rm left}\},
\end{equation}
\begin{equation}
	\label{eq:rightrep}
	\mathbb{H}^{\rm right}_k = \{h \in \mathbb{H}^i_k : c_{br}^i \leq c_{br}^h \} \cup \{{\rm right}\}.
\end{equation}

\begin{example}
	Let us consider the subdivided hyper-rectangle $\bar{D}^{4}_2$, whose representative sampling index set $\mathbb{H}^{4}_2 = \{ 1, 2, 4\}$ (see \Cref{fig:example}).
	Then $\mathbb{H}^{\rm left}_3 $ and $\mathbb{H}^{\rm right}_3 $ consist of:
	\[
	\mathbb{H}^{\rm left}_3 = \mathbb{H}^{8}_3 = \left\{ h \in \mathbb{H}^4_2 : c_{1}^4 \geq c_{1}^h \right\} \cup \{8\} = \{ 2, 4, 8\},
	\]
	\[
	\mathbb{H}^{\rm right}_3 = \mathbb{H}^{9}_3 = \left\{ h \in \mathbb{H}^4_2 : c_{1}^4 \leq c_{1}^h \right\} \cup \{9\} = \{ 1, 2, 4, 9\}.
	\]
\end{example}

In the following subsection, we will show how these representative index sets $(\mathbb{H}^i_k)$ are used to select potentially optimal hyper-rectangles by taking into account up to $2 \times n+1$ objective function values over each hyper-rectangle.
But first we prove that the cardinality of $\mathbb{H}^i_k$ cannot exceed $2n + 1$.

\begin{corollary}\label{coroll1}
	The cardinality of any representative sampling index set $\mathbb{H}^i_k$ is less than or equal to $2 \times n+1$, i.e.,
	\begin{equation}\label{eq:coroll}
		\max_{i \in \indexsett_k, \forall k}{{\rm card}(\mathbb{H}^i_k) \leq 2 \times n+1}.
	\end{equation}
\end{corollary}
\begin{proof}
	In \halrect{}, selected POHs are bisected only along one coordinate with the maximum side length.
	Without loss of generality, assume that $br = 1$, i.e., the branching (bisection) on the $x_1$ variable takes place.
	As a result, the midpoint, after bisection, shifts on the ``left'' and on the ``right'' facet of two newly created hyper-rectangles (see the middle part for two and three-dimensional illustrations in \Cref{fig:overall}).
	This way, each subdivided hyper-rectangle cuts off the old facet and replaces it with a new one.
	Therefore, only one point can appear on one facet concerning the branching variable.
	
	Throughout the search process (as the number of iterations $k$ increases), all this will be applied to other branching variables ($x_2, \dots, x_n$) too. 
	From this follows, that the set $\mathbb{H}^i_k$ is constructed only by points located in the hyper-rectangular facets and one midpoint.
	As each hyper-rectangle contains $2 \times n$ facets, the maximal number of $2 \times n + 1$ points can be included in $\mathbb{H}^i_k$.
	
	%
	%
\end{proof}

\subsection{Selection of potentially optimal hyper-rectangles}
\label{ssec:selection-strategies}

Since the objective function in the \halrect{} algorithm is evaluated at multiple points, more comprehensive information about the objective function values can be efficiently integrated into the selection scheme.
In \Cref{def:potOptRectN}, we introduce four different selection schemes implemented in the new \halrect{} algorithm,
where the main difference is how the value $\mathcal{F}^i_k$ is calculated (see \cref{eq:eq1,eq:eq2,eq:eq3,eq:eq4}).

\begin{definition}{(\halrect{} selection)}
	\label{def:potOptRectN}
	Let \textcolor{blue}{$ \mathbf{c}^i$} $\in \bar{D}^i_k$ denote the midpoint, $ \mathbf{c}^j \in \bar{D}^i_k, j \in \mathbb{H}^i_k$ denote all sampling points (including \textcolor{blue}{$ \mathbf{c}^i $}) of hyper-rectangle $(\bar{D}^i_k)$, $ \textrm{card}( \mathbb{H}^i_k ) $ -- the cardinality of $(\mathbb{H}^i_k)$, $ \delta^i_k $ be a measure of $ \bar{D}^i_k$, and $\mathcal{F}^i_k$ -- aggregated value based on objective function values attained at sampling point(s) whose indices belong to $ \mathbb{H}^i_k $.
	Let $ \varepsilon > 0 $ be a positive constant and $f^{\min}$ be the best currently found objective function value.
	A hyper-rectangle $ \bar{D}^h_k, h \in \indexsett_k $ is said to be potentially optimal if there exists some rate-of-change (Lipschitz) constant $ \tilde{L} > 0$ such that
	\begin{eqnarray}
		\mathcal{F}^h_k - \tilde{L}\delta^h_k & \leq & \mathcal{F}^i_k - \tilde{L}\delta^i_k, \quad \forall i \in \indexsett_k, \label{eq:potOptRect1N} \\
		\mathcal{F}^h_k - \tilde{L}\delta^h_k & \leq & f^{\min} - \varepsilon \abs{f^{\min}}, \label{eq:potOptRect2N}
	\end{eqnarray}
	where the measure of the hyper-rectangle $ \bar{D}^i_k$ is
	\begin{equation}
		\label{eq:distanceN}
		\delta^i_k = \left\| {\mathbf{b}}_k^i - {\mathbf{a}}_k^i \right\|,
	\end{equation}
	and $\mathcal{F}^i_k$ is defined in one of the following four ways
	\begin{subequations}
		\label{eq:main}
		\begin{align}
			\label{eq:eq1}
			\mathcal{F}^i_k &= f(\textcolor{blue}{\mathbf{c}^i})\\
			\label{eq:eq2}
			\mathcal{F}^i_k &= \min_{j \in \mathbb{H}^i_k} f(\mathbf{c}^j)\\
			\label{eq:eq3}
			\mathcal{F}^i_k &= \dfrac{1}{\textrm{card}( \mathbb{H}^i_k )} \sum_{j=1}^{\textrm{card}( \mathbb{H}^i_k )} f(\mathbf{c}^j)\\
			\label{eq:eq4}
			\mathcal{F}^i_k &= \frac{1}{2}\left( \min_{j \in \mathbb{H}^i_k} f(\mathbf{c}^j) +  f(\textcolor{blue}{\mathbf{c}^i}) \right)
		\end{align}
	\end{subequations}
\end{definition}

\subsubsection{Midpoint value based selection}
\label{ssec:midpoint-value-selection}
\Cref{def:potOptRect} is typically used in most existing \direct-type algorithmic modifications to select POHs.
Geometrical visualization of the selection scheme used in \direct{} was shown in \Cref{fig:poh}.
The same selection scheme could be directly applied using the new sampling and partitioning strategy proposed in \halrect, as the midpoint is always included in the sampling set.
It is obtained by using \cref{eq:eq1} in \Cref{def:potOptRectN}.

For the illustrative comparison of all selection schemes, we will use partitioned space in the seventh iteration of the \halrect{} algorithm solving the two-dimensional \textit{Bukin6} test function. 
The selected POHs using this selection scheme are shown in part  (a) on the right panel of \Cref{fig:divides}.
$Y$-axis shows the objective function values attained at the midpoints $f(\textcolor{blue}{\mathbf{c}^i})$ of hyper-rectangles belonging to the current partition.
These values can also be seen on the left panel of \Cref{fig:divides}.
The apparent drawback is that the midpoints of previously partitioned hyper-rectangles (see black dots on the left panel of \Cref{fig:divides}) are not involved in POH selection anymore.

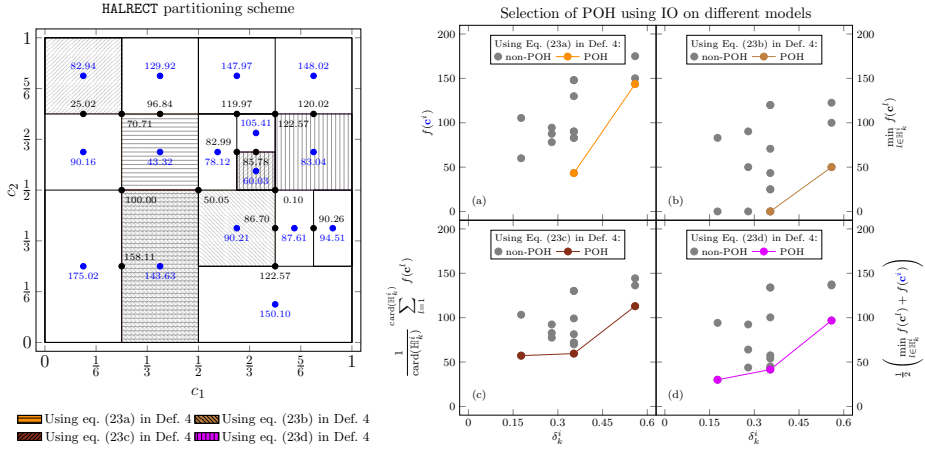
\begin{figure}[ht]
	\resizebox{\textwidth}{!}{
		\begin{tikzpicture}
			\begin{axis}[
				width=0.8\textwidth,height=0.8\textwidth,
				legend style={draw=none},
				legend columns=2,
				legend style={at={(1.05,-0.175)},font=\normalsize},
				xlabel = {$c_1$},
				ylabel = {$c_2$},
				ymin=0,ymax=1,
				xmin=0,xmax=1,
				enlargelimits=0.029,
				title={\large  \halrect{} partitioning scheme},
				ylabel style={yshift=-0.1cm},
				xlabel style={yshift=-0.1cm},
				ytick distance=1/6,
				xtick distance=1/6,
				every axis/.append style={font=\LARGE},
				yticklabels={$0$, $0$,$\frac{1}{6}$, $\frac{1}{3}$, $\frac{1}{2}$, $\frac{2}{3}$, $\frac{5}{6}$, $1$},
				xticklabels={$0$, $0$,$\frac{1}{6}$, $\frac{1}{3}$, $\frac{1}{2}$, $\frac{2}{3}$, $\frac{5}{6}$, $1$},
				]
				\addlegendimage{area legend,black,fill=princetonorange,postaction={pattern=horizontal lines}}
				\addlegendentry{Using \cref{eq:eq1} in Def. \ref{def:potOptRectN}}
				\addlegendimage{area legend,black,fill=brown,postaction={pattern=north west lines}}
				\addlegendentry{Using \cref{eq:eq2} in Def. \ref{def:potOptRectN}}
				\addlegendimage{area legend,black,fill=sienna,postaction={pattern=north east lines}}
				\addlegendentry{Using \cref{eq:eq3} in Def. \ref{def:potOptRectN}}
				\addlegendimage{area legend,black,fill=psychedelicpurple,postaction={pattern=vertical lines}}
				\addlegendentry{Using \cref{eq:eq4} in Def. \ref{def:potOptRectN}}
				\draw [princetonorange, thick, mark size=0.01pt,line width=0.01mm,postaction={pattern=horizontal lines},opacity=0.3] (axis cs:1/4, 0) rectangle (axis cs:1/2, 1/2);
				\draw [princetonorange, thick, mark size=0.01pt,line width=0.01mm,postaction={pattern=horizontal lines},opacity=0.4] (axis cs:1/4, 1/2) rectangle (axis cs:1/2, 3/4);
				
				\draw [brown, thick, mark size=0.01pt,line width=0.01mm,postaction={pattern=north west lines},opacity=0.3] (axis cs:1/4, 0) rectangle (axis cs:1/2, 1/2);
				\draw [brown, thick, mark size=0.01pt,line width=0.01mm,postaction={pattern=north west lines},opacity=0.4] (axis cs:1/2, 1/4) rectangle (axis cs:3/4, 1/2);
				
				\draw [sienna, thick, mark size=0.01pt,line width=0.01mm,postaction={pattern=north east lines},opacity=0.4] (axis cs:0, 3/4) rectangle (axis cs:1/4, 1);
				\draw [sienna, thick, mark size=0.01pt,line width=0.01mm,postaction={pattern=north east lines},opacity=0.4] (axis cs:5/8, 1/2) rectangle (axis cs:3/4, 5/8);
				\draw [sienna, thick, mark size=0.01pt,line width=0.01mm,postaction={pattern=north east lines},opacity=0.3] (axis cs:1/4, 0) rectangle (axis cs:1/2, 1/2);
				
				\draw [psychedelicpurple, thick, mark size=0.01pt,line width=0.01mm,postaction={pattern=vertical lines},opacity=0.4] (axis cs:3/4, 1/2) rectangle (axis cs:1, 3/4);
				\draw [psychedelicpurple, thick, mark size=0.01pt,line width=0.01mm,postaction={pattern=vertical lines},opacity=0.4] (axis cs:5/8, 1/2) rectangle (axis cs:3/4, 5/8);
				\draw [psychedelicpurple, thick, mark size=0.01pt,line width=0.01mm,postaction={pattern=vertical lines},opacity=0.3] (axis cs:1/4, 0) rectangle (axis cs:1/2, 1/2);
				
				\addplot[only marks,mark=*,mark size=2pt,black] coordinates {(1/2, 1/2)} node[yshift=-7pt, xshift=13pt, font=\scriptsize] { $50.05$}; 
				\addplot[only marks,mark=*,mark size=2pt,black] coordinates {(1/4, 1/2)} node[yshift=-7pt, xshift=14pt, font=\scriptsize] { $100.00$}; 
				\addplot[only marks,mark=*,mark size=2pt,black] coordinates {(3/4, 1/2)} node[yshift=-7pt, xshift=13pt, font=\scriptsize] { $0.10$};
				\addplot[only marks,mark=*,mark size=2pt,black] coordinates {(3/4, 1/4)} node[yshift=-7pt, font=\scriptsize] { $122.57$};
				\addplot[only marks,mark=*,mark size=2pt,black] coordinates {(3/4, 3/4)} node[yshift=-7pt, xshift=13pt, font=\scriptsize] { $122.57$};
				\addplot[only marks,mark=*,mark size=2pt,black] coordinates {(5/8, 3/4)} node[yshift=7pt, font=\scriptsize] { $119.97$};
				\addplot[only marks,mark=*,mark size=2pt,black] coordinates {(7/8, 3/4)} node[yshift=7pt, font=\scriptsize] { $120.02$};
				\addplot[only marks,mark=*,mark size=2pt,black] coordinates {(1/4, 1/4)} node[yshift=7pt, xshift=13pt, font=\scriptsize] { $158.11$}; 
				\addplot[only marks,mark=*,mark size=2pt,black] coordinates {(1/4, 3/4)} node[yshift=-7pt, xshift=13pt, font=\scriptsize] { $70.71$}; 
				\addplot[only marks,mark=*,mark size=2pt,black] coordinates {(1/8, 3/4)} node[yshift=7pt, font=\scriptsize] { $25.02$};
				\addplot[only marks,mark=*,mark size=2pt,black] coordinates {(3/8, 3/4)} node[yshift=7pt, font=\scriptsize] { $96.84$};
				\addplot[only marks,mark=*,mark size=2pt,blue] coordinates {(1/8, 5/8)} node[yshift=-7pt, font=\scriptsize] { $90.16$};
				\addplot[only marks,mark=*,mark size=2pt,blue] coordinates {(1/8, 7/8)} node[yshift=7pt, font=\scriptsize] { $82.94$};
				\addplot[only marks,mark=*,mark size=2pt,blue] coordinates {(3/4, 1/8)} node[yshift=-7pt, font=\scriptsize] { $150.10$};
				\addplot[only marks,mark=*,mark size=2pt,black] coordinates {(3/4, 3/8)} node[yshift=7pt, xshift=-13pt, font=\scriptsize] { $86.70$};
				\addplot[only marks,mark=*,mark size=2pt,blue] coordinates {(5/8, 3/8)} node[yshift=-7pt, font=\scriptsize] { $90.21$};
				\addplot[only marks,mark=*,mark size=2pt,black] coordinates {(7/8, 3/8)} node[yshift=7pt, xshift=13pt, font=\scriptsize] { $90.26$};
				\addplot[only marks,mark=*,mark size=2pt,blue] coordinates {(26/32, 3/8)} node[yshift=-7pt, font=\scriptsize] { $87.61$};
				\addplot[only marks,mark=*,mark size=2pt,blue] coordinates {(30/32, 3/8)} node[yshift=-7pt, font=\scriptsize] { $94.51$};
				\addplot[only marks,mark=*,mark size=2pt,black] coordinates {(5/8, 5/8)} node[yshift=7pt, xshift=-13pt, font=\scriptsize] { $82.99$};
				\addplot[only marks,mark=*,mark size=2pt,blue] coordinates {(5/8, 7/8)} node[yshift=7pt, font=\scriptsize] { $147.97$};
				\addplot[only marks,mark=*,mark size=2pt,blue] coordinates {(1/8, 1/4)} node[yshift=-7pt, font=\scriptsize] { $175.02$}; 
				\addplot[only marks,mark=*,mark size=2pt,blue] coordinates {(3/8, 1/4)} node[yshift=-7pt, font=\scriptsize] { $143.63$}; 
				\addplot[only marks,mark=*,mark size=2pt,blue] coordinates {(7/8, 5/8)} node[yshift=-7pt, font=\scriptsize] { $83.04$};
				\addplot[only marks,mark=*,mark size=2pt,blue] coordinates {(7/8, 7/8)} node[yshift=7pt, font=\scriptsize] { $148.02$};
				\addplot[only marks,mark=*,mark size=2pt,blue] coordinates {(18/32, 5/8)} node[yshift=-7pt, font=\scriptsize] { $78.12$};
				\addplot[only marks,mark=*,mark size=2pt,black] coordinates {(22/32, 5/8)} node[yshift=-7pt, font=\scriptsize] { $85.78$};
				\addplot[only marks,mark=*,mark size=2pt,blue] coordinates {(22/32, 18/32)} node[yshift=-7pt, font=\scriptsize] { $60.03$};
				\addplot[only marks,mark=*,mark size=2pt,blue] coordinates {(22/32, 22/32)} node[yshift=7pt, font=\scriptsize] { $105.41$};
				\addplot[only marks,mark=*,mark size=2pt,blue] coordinates {(3/8, 5/8)} node[yshift=-7pt, font=\scriptsize] { $43.32$};
				\addplot[only marks,mark=*,mark size=2pt,blue] coordinates {(3/8, 7/8)} node[yshift=7pt, font=\scriptsize] { $129.92$};
				\addplot[patch,mesh,draw,black,patch type=rectangle] coordinates {
					(0, 0) (1, 0) (1, 1) (0, 1)
					(0, 1/2) (1, 1/2) (1, 0) (0, 0)
					(1/2, 0) (1/2, 1) (1, 1) (1, 0)
					(1/4, 0) (1/4, 1) (0, 1) (0, 0)
					(1/4, 3/4) (1/4, 1) (0, 1) (0, 3/4)
					(1/4, 3/4) (1/4, 1) (1/2, 1) (1/2, 3/4)
					(3/4, 1/4) (3/4, 3/4) (1/2, 3/4) (1/2, 1/4)
					(5/8, 1/2) (5/8, 3/4) (3/4, 3/4) (3/4, 1/2)
					(5/8, 1/2) (5/8, 5/8) (3/4, 5/8) (3/4, 1/2)
					(1, 1/4) (1, 3/4) (1/2, 3/4) (1/2, 1/4)
					(3/4, 3/4) (3/4, 1) (1/2, 1) (1/2, 3/4)
					(7/8, 1/4) (7/8, 1/2) (1, 1/2) (1, 1/4)};
			\end{axis}
		\end{tikzpicture}
		\begin{tikzpicture}
			\begin{groupplot}[
				group style={
					group size=2 by 2,
					x descriptions at=edge bottom,
					vertical sep=0pt,
					horizontal sep=0pt,
				},
				height=0.535\textwidth,width=0.535\textwidth,
				]
				\nextgroupplot[
				ymin=0,ymax=200,
				xmin=0,xmax=0.6,
				xtick distance=0.15,
				ytick distance=50,
				enlargelimits=0.049,
				ylabel = {$f(\textcolor{blue}{\mathbf{c}}^i)$},
				ylabel style={yshift=-0.1cm},
				]
				\addplot[only marks,mark=*,mark size=2.5pt,black!50] coordinates {(-0.1,-0.1)} ;
				\label{p123}
				\addplot[mark=*,mark size=2.5pt,princetonorange] coordinates {(-0.1,-0.1)} ;
				\label{p321}
				\addplot[only marks,mark=*,mark size=2.5pt,black!50] table[x=T,y=Xa] {data/POH.txt};
				\addplot[mark=*,mark size=2.5pt,princetonorange] coordinates {(0.5590, 143.6391) (0.3536, 43.3263)};
				\node [draw,fill=white] at (rel axis cs: 0.5, 0.875) {\shortstack[l]{
						{\scriptsize {Using} \cref{eq:eq1} {in Def.} \ref{def:potOptRectN}}: \\
						\ref{p123} {\scriptsize non-POH} \ref{p321} {\scriptsize POH} }};
				\node [fill=white] at (rel axis cs: 0.1, 0.1) {\shortstack[l]{ (a) }};
				\nextgroupplot[
				ymin=0,ymax=200,
				xmin=0,xmax=0.6,
				xtick distance=0.15,
				ytick distance=50,
				ylabel = {$\displaystyle\min_{l \in \mathbb{H}^i_k} f(\mathbf{c}^l)$},
				ylabel near ticks, yticklabel pos=right,
				enlargelimits=0.049,
				]
				\addplot[mark=*,mark size=2.5pt,brown] coordinates {(-0.1,-0.1)} ;
				\label{p322}
				\addplot[only marks,mark=*,mark size=2.5pt,black!50] table[x=T,y=Xb] {data/POH.txt};
				\addplot[mark=*,mark size=2.5pt,brown] coordinates {(0.5590, 50.0500) (0.3536, 0.1000)};
				\node [draw,fill=white] at (rel axis cs: 0.5, 0.875) {\shortstack[l]{
						{\scriptsize {Using} \cref{eq:eq2} {in Def.} \ref{def:potOptRectN}}: \\
						\ref{p123} {\scriptsize non-POH} \ref{p322} {\scriptsize POH} }};
				\node [fill=white] at (rel axis cs: 0.1, 0.1) {\shortstack[l]{ (b) }};
				\nextgroupplot[
				ymin=0,ymax=200,
				xmin=0,xmax=0.6,
				xtick distance=0.15,
				ytick distance=50,
				enlargelimits=0.049,
				xlabel = {$\delta^i_k$},
				ylabel = {$\dfrac{1}{\textrm{card}( \mathbb{H}^i_k )} \displaystyle\sum_{l=1}^{\textrm{card}( \mathbb{H}^i_k )} f(\mathbf{c}^l)$},
				ylabel style={yshift=0.1cm},
				xlabel style={yshift=0.1cm},
				]
				\addplot[mark=*,mark size=2.5pt,sienna] coordinates {(-0.1,-0.1)} ;
				\label{p323}
				\addplot[only marks,mark=*,mark size=2.5pt,black!50] table[x=T,y=Xc] {data/POH.txt};
				\addplot[mark=*,mark size=2.5pt,sienna] coordinates {(0.5590, 112.9507) (0.3536, 59.5588) (0.1768, 57.2273)};
				\node [draw,fill=white] at (rel axis cs: 0.5, 0.875) {\shortstack[l]{
						{\scriptsize {Using} \cref{eq:eq3} {in Def.} \ref{def:potOptRectN}}: \\
						\ref{p123} {\scriptsize non-POH} \ref{p323} {\scriptsize POH} }};
				\node [fill=white] at (rel axis cs: 0.1, 0.1) {\shortstack[l]{ (c) }};
				\nextgroupplot[
				ymin=0,ymax=200,
				xmin=0,xmax=0.6,
				xtick distance=0.15,
				ytick distance=50,
				enlargelimits=0.049,
				xlabel = {$\delta^i_k$},
				ylabel = {$\frac{1}{2}\left( \displaystyle\min_{l \in \mathbb{H}^i_k} f(\mathbf{c}^l) +  f(\textcolor{blue}{\mathbf{c}^i}) \right)$},
				ylabel near ticks, yticklabel pos=right,
				xlabel style={yshift=0.1cm},
				]
				\addplot[mark=*,mark size=2.5pt,psychedelicpurple] coordinates {(-0.1,-0.1)} ;
				\label{p333}
				\addplot[only marks,mark=*,mark size=2.5pt,black!50] table[x=T,y=Xd] {data/POH.txt};
				\addplot[mark=*,mark size=2.5pt,psychedelicpurple] coordinates {(0.5590, 96.8445) (0.3536, 41.5703) (0.1768, 30.0677)};
				\node [draw,fill=white] at (rel axis cs: 0.5, 0.875) {\shortstack[l]{
						{\scriptsize {Using} \cref{eq:eq4} {in Def.} \ref{def:potOptRectN}}: \\
						\ref{p123} {\scriptsize non-POH} \ref{p333} {\scriptsize POH} }};
				\node [fill=white] at (rel axis cs: 0.1, 0.1) {\shortstack[l]{ (d) }};
			\end{groupplot}
			\node (title) at ($(group c1r1.center)!0.5!(group c2r1.center)+(0,2.75cm)$) {\large Selection of POH using IO on different models};
	\end{tikzpicture}}
	\caption{Two-dimensional illustration (in the seventh iteration of \halrect{} on \textit{Bukin6} test problem) of four different POH selection scheme variations (see \Cref{def:potOptRectN}) implemented in the \halrect{} algorithm and controlled by \cref{eq:eq1,eq:eq2,eq:eq3,eq:eq4}.}
	\label{fig:divides}
\end{figure}

\subsubsection{Minimum value based selection}
\label{ssec:minimum-value-selection}
The second selection scheme in \halrect{} is motivated by the  \birect{} algorithm~\cite{Paulavicius2016:jogo}.
Instead of objective function evaluation at midpoints, the sampling and evaluation on the diagonal points equidistant between themselves and the endpoints of a diagonal are used.
Then, in the selection of POHs, the minimum of these two points is used.
In the \halrect{} case, the best (minimum) function value is used at all the points sampled in the hyper-rectangle $(\bar{D}^i_k)$.
It is obtained by using \cref{eq:eq2} in \Cref{def:potOptRectN}.

As more sampling points are used in the lower Lipschitz bound calculation, more information about the objective function is exploited for POH identification, likely to result in faster convergence.
Therefore, on the vertical $y$-axis, instead of function values obtained at the current midpoints, the minimum values attained at all sampled points over a hyper-rectangle $(\min_{j \in \mathbb{H}^i_k} f(\mathbf{c}^j))$ are used (see part (b) on the right side of \Cref{fig:divides}).

\begin{corollary}\label{coroll2}
	For each hyper-rectangle $ \bar{D}^i_k $ the following condition holds
	\begin{equation}\label{eq:coroll2}
		\min_{j \in \mathbb{H}^i_k} f(\mathbf{c}^j) \le f(\textcolor{blue}{\mathbf{c}^i})
	\end{equation}
\end{corollary}
\begin{proof}
	It follows directly from the definition of $\mathbb{H}^i_k$ (see \Cref{def:potOptRectN}).
\end{proof}


\subsubsection{Mean value based selection scheme}
\label{ssec:mean-value-selection}
The third selection scheme implemented in \halrect{} is motivated by the mean value obtained at diagonal sampling points and proposed in~\cite{Sergeyev2006}.
In the \halrect{} case, the mean function value is calculated from all sampled points on the hyper-rectangle $(\bar{D}^i_k)$.
It is obtained by using \cref{eq:eq3} in \Cref{def:potOptRectN}.
Using this selection scheme, on the vertical $y$-axis, the mean values calculated from objective function values attained at all sampled points over a hyper-rectangle are used (see part (c) on the right side of \Cref{fig:divides}).

\subsubsection{Midpoint and minimum values based selection scheme}
\label{ssec:midpoint-minimum-value-selection}
The final selection scheme (see \cref{eq:eq4}) implemented in \halrect{} combines ideas used in \cref{eq:eq1} and \cref{eq:eq3} and takes the mean of these two values.
On the vertical $y$-axis, the mean values calculated for each hyper-rectangle using two values: i) the midpoint value $f(\textcolor{blue}{\mathbf{c}^i})$, and ii) the minimum value $\min_{j \in \mathbb{H}^i_k} f(\mathbf{c}^j)$ are used (see part (d) on the right side of \Cref{fig:divides}).

The impact of these four selection schemes on the performance of \halrect{} is explored in \Cref{sec:model}.

\subsubsection{Reducing the set of selected POHs}
\label{ssec:reducing-POH}
It was stated in~\Cref{sec:selection-schemes} that sometimes, e.g., using \Cref{def:potOptRect} on symmetric problems, there might exist many POHs with the same measure $\delta_k$ and objective function value, leading to a significant increase of selected ``equivalent'' POHs per iteration.
This situation can also arise in \halrect, mainly when \cref{eq:eq2} is used.
Then a good objective function value attained at the vertex can be shared up to $2^n$ hyper-rectangles.

Many authors (see, e.g.,~\cite{Jones2001,Jones2021,Gablonsky2001,Stripinis2021b,Stripinis2021c}) observed that selecting only one from many ``equivalent'' candidates can significantly increase the performance of  \direct-type algorithms.
Some authors (see, e.g.,~\cite{Baker2000,Gablonsky2001,Jones2001}) did not specify how the only candidate should be selected, while in \cite{Stripinis2018a,Stripinis2021c}, the authors selected a hyper-rectangle with the largest index value among them.
In the \halrect{} algorithm, as more sampling points per hyper-rectangle are available, we use a unique strategy to select ``the most promising candidate''.
Specifically, we sort in ascending order the objective function values attained at the points belonging to the hyper-rectangle.
Then, if there are two hyper-rectangles of the same size with the same minimum value, we compare the second smallest values and choose the hyper-rectangle with the smaller value.
If the second smallest values are equal, we compare the subsequent ones.

\subsection{Algorithmic steps}
\label{sec:directhalrect}

The complete description of the \halrect{} algorithm is shown in Algorithm~\ref{alg:halrect}.
The inputs for the algorithm are the problem ($f$), optimization domain ($D$), and one (or few) stopping criteria: required tolerance ($\varepsilon_{\rm pe}$), the maximal number of function evaluations ($\texttt{M}_{\rm max}$), and the maximal number of iterations ($\texttt{K}_{\rm max}$).
After termination, \halrect{} returns the value of the objective function found $(f^{\min})$ and the solution point $(\mathbf{x}^{\min})$ together with algorithmic performance measures: final tolerance~-- percent error~$(pe)$, the number of function evaluations~$(m)$, and the number of iterations~$(k)$.

\begin{algorithm}
	\caption{The \halrect{} algorithm}\label{alg:halrect}
	\begin{algorithmic}[1]
		\State \halrect($f$,$D$,opt);
		\Require {Problem $f$, search domain $D$, and adjustable algorithmic parameters $opt$: tolerance ($\varepsilon_{\rm pe}$), the maximal number of function evaluations ($\texttt{M}_{\rm max}$) and the maximal number of iterations ($\texttt{K}_{\rm max}$);}
		\Ensure {The best objective function value $f^{\min}$, minimum point $\mathbf{c}^{\min}$, and algorithmic performance measures $pe$, $k$, $m$;}
		\algrule
		\State Normalize the search domain $D$ to be the unit hyper-rectangle $\bar{D}$; \label{alg:initialization_begins}
		\State Initialize: $\mathbf{c}^1 = (\frac{1}{2},\dots,\frac{1}{2})$, $k=1$, $m=1$ and $pe$;  \Comment{\textit{pe} defined in \cref{eq:pe}}
		\State Evaluate $f^1 = f(\mathbf{c}^1)$, and set $f^{\min} = f^1$, $\mathbf{c}^{\min} = \mathbf{c}^1$, $\mathbb{H}^1_1 = \{ 1\}$;\label{alg:initialization_ends}
		\While{$pe > \varepsilon_{\rm pe}$ \textbf{and} $m < \texttt{M}_{\rm max}$ \textbf{and} $k < \texttt{K}_{\rm max}$ } \label{alg:iterbegin}
		\State Identify the set $S_k \subseteq \mathcal{P}_k$ of POHs applying \cref{def:potOptRectN}; \label{alg:selection_begins}
		\ForEach{$\bar{D}^j_k \in {S}_k$}
		\State Find the branching coordinate index $(br)$ using \cref{alg:branching-coordinate};
		\State Bisect $\bar{D}^j_k$ into two new  hyper-rectangles $\bar{D}^{m+1}_k$ and $\bar{D}^{m+2}_k$; \label{alg:subdivisionbegin}
		\State Create new midpoints $\mathbf{c}^{m+1}$ and $\mathbf{c}^{m+2}$; \Comment{see~\cref{eq:leftpoint,eq:rightpoint}}\label{alg:samplingend}
		\State Construct $\mathbb{H}^{m+1}, \mathbb{H}^{m+2}$; \Comment{see \cref{eq:leftrep,eq:rightrep}}
		\State Update the partition set: $\mathcal{P}_k = \mathcal{P}_k \setminus \bar{D}^j_k \cup \bar{D}^{m+1}_k \cup \bar{D}^{m+2}_k$;\label{alg:subdivisionend}
		\If{$f(\mathbf{c}^{m+1}) \leq f^{\min}$ \textbf{or} $f(\mathbf{c}^{m+2}) \leq f^{\min}$ } \label{alg:dimen_begin}
		\State Update $f^{\min}, \mathbf{c}^{\min}$;
		\EndIf
		\State Update performance measures: $k$, $m$ and $pe$;
		\EndFor
		\EndWhile \label{alg:iterend} \\
		\textbf{Return} $f^{\min}, \mathbf{c}^{\min}$, and algorithmic performance measures: $k$, $m$ and $pe$.
	\end{algorithmic}
\end{algorithm}

Like almost all \direct-type algorithms, \halrect{} performs initialization: normalization of the feasible region, initial evaluation of the objective function at the midpoint, setting initial values for performance measures, and specifying stopping conditions (see Algorithm~\ref{alg:halrect}, lines \ref{alg:initialization_begins}--\ref{alg:initialization_ends}).
The main \textit{while} loop (see Algorithm~\ref{alg:halrect}, lines \ref{alg:iterbegin}--\ref{alg:iterend}) is executed until any specified stopping condition is satisfied.
At the beginning of each iteration, the \halrect{} algorithm identifies the set of POHs (see Algorithm~\ref{alg:halrect}, line \ref{alg:selection_begins}).
As noted in the previous section, the \halrect{} algorithm uses four different approaches controlled by \cref{eq:eq1,eq:eq2,eq:eq3,eq:eq4}.
Then, the \halrect{} algorithm bisects all POHs, samples at new midpoints of created hyper-rectangles and updates performance measures.
In the end, the solution is found, and the performance measures are returned.

\subsection{Convergence properties of the \halrect{} algorithm}\label{sec:converg}

The convergence properties of \direct-type algorithms are broadly reviewed and investigated~(see, e.g., \cite{Finkel2006,Jones1993,Paulavicius2016:jogo,Paulavicius2014:jogo,Sergeyev2006}).
Typically, they belong to the class of ``divide the best'' methods and have the ``everywhere-dense'' type of convergence, that is, convergence to each point of the feasible region.
The continuity of the objective function (at least in the neighborhood of global minima) is the only assumption required to ensure convergence.

The convergence of \halrect{} follows from a logic similar to that of other \direct-type algorithms.
Let us state it formally in \Cref{theorem}, when the maximal allowed number of generated trial points, or the maximal number of function evaluations, $\texttt{M}_{\rm max} \rightarrow \infty$.


\begin{theorem}
	\label{theorem}
	For any global minima $\mathbf{x}^* \in \bar{D}$ and any $\epsilon > 0$ there exists an iteration number $k_{\epsilon} \ge 1$ and a sampling point $ \mathbf{c}^j \in \bar{D}^{i^{*}}_k \subseteq \bar{D}$, such that
	\begin{equation}
		\label{theorem_eq}
		\max_{j \in \mathbb{H}^{i^{*}}_k}{ \{ \| \mathbf{c}^j - \mathbf{x}^* \| \} } \leq \epsilon.
	\end{equation}
\end{theorem}
\begin{proof}
	In the selection scheme developed in \halrect {} (see \Cref{def:potOptRectN}), every iteration $(k)$ always selects at least one hyper-rectangle $\bar{D}^{\max}_k \in {S}_k \subseteq \mathcal{P}_k$ from the group of hyper-rectangles with the most extensive measure $\delta_k^{\rm max}$ (see the right panel of \Cref{fig:divides})
	\begin{equation}
		\label{theorem_eqau}
		\delta_k^{\rm max} = \max_{i \in \indexsett_k}{\{\delta^i_k  \}}.
	\end{equation}
	From \cref{theorem_eqau} follows, the hyper-rectangle $\bar{D}^{\max}_k$ with the largest measure $\delta_k^{\rm max}$ will be bisected through the longest coordinate (see \Cref{partitioning}) in each \halrect{} iteration.
	Since each group $\delta_k$ of distinct measures contains only a finite number of hyper-rectangles, all hyper-rectangles of the group $\delta_k^{\rm max}$ will be partitioned after a sufficient number of iterations. 
	
	The procedure will be repeated with a new group of the largest hyper-rectangles.
	As a result, after the finite number of iterations, the current partition $\mathcal{P}_k, k \ge k_{\epsilon}$ will have only hyper-rectangles measured $\delta_k^{\rm max} \le \epsilon$, i.e.,
	
	\begin{equation}
		\label{theorem_eqaus}
		\| \mathbf{b}_k^{i^{\max}} - \mathbf{a}_k^{i^{\max}} \| \leq \epsilon.
	\end{equation}
	From \cref{theorem_eqaus} follows, the measure $\delta_k^{i^{*}}$ of the hyper-rectangle containing the global minimum $\mathbf{x}^* \in \bar{D}_k^{i^{*}}$ also does not exceed $\epsilon$
	\begin{equation}
		\label{theorem_eqs}
		\| \mathbf{b}_k^{i^{*}} - \mathbf{a}_k^{i^{*}} \| \leq \epsilon.
	\end{equation}
	Moreover, it is clear, that 
	\begin{equation}
		\label{theorem_eqs2}
		\max_{j \in \mathbb{H}^{i^{*}}_k}{ \{ \| \mathbf{c}^j - \mathbf{x}^* \| \} } \leq \delta_k^{i^{*}}.
	\end{equation}
	Thus, from \cref{theorem_eqs,theorem_eqs2} follows \cref{theorem_eq}.
\end{proof}

\section{Experimental results}
\label{sec:results}

This section describes the numerical experiments conducted to evaluate the performance of the newly introduced \halrect{} algorithm and all its modifications by comparing them with other well-known and relevant \direct-type approaches.
In total, we examine twelve variations of \halrect{}.
We compared them with twelve recently introduced \direct-type algorithms~\cite{Stripinis2021b} available in the most recent version of \directgo{} \cite{DIRECTGOv1.1.0} using $96$ box-constrained global optimization test problems and their perturbed versions from \directlib~\cite{DIRECTGOLibv1.1,DIRECTGOLibv11}  (listed in \Cref{tab:test} in Appendix \ref{apendixas}).

In our recent study~\cite{Stripinis2021b}, we stress that the optimization domains $(D)$ for certain test problems were redesigned to eliminate the dominance of particular partitioning schemes.
The exact modified domains are also considered in this paper.
Note that different subsets (e.g., low dimensional problems $(n \le 4)$, non-convex problems, etc.) of the entire set were used to deepen the investigation.
All the problems and algorithms used in this section are implemented in the Matlab R2022a environment.
All computations were performed using an Intel R Core$^\textit{TM}$ i5-10400 @ 2.90GHz processor and $16$ GB RAM.
All algorithms were tested using a limit of M$_{\rm max} = 10^6$ function evaluations in each run.
For the $96$ analytical test cases with a priori known global optima $ f^* $, one of the used stopping criteria is based on the percent error:
\begin{equation}
	\label{eq:pe}
	\ pe = 100 \times
	\begin{cases}
		\frac{f({\mathbf{x}}) - f^*}{\mid f^* \mid},  & f^* \neq 0, \\
		f({\mathbf{x}}),  & f^* = 0,
	\end{cases}
\end{equation}
where $ f^* $ is the known global optimum.
Thus if not specified differently, tested algorithms were stopped when the percent error became smaller than the prescribed value equal to $\varepsilon_{\rm pe} = 10^{-2}$ or when the number of function evaluations exceeded the prescribed limit of $10^6$.

Testing results shown in this article are also available in digital form in the \texttt{Results/JOGO2} directory of the GitHub repository~\cite{DIRECTGOv1.1.0}.
The \texttt{Scripts/JOGO2} directory of the same GitHub repository~\cite{DIRECTGOv1.1.0} provides the \texttt{MATLAB} script for cycling through all \directlib{} test problems used in this article.
The script can reproduce the results presented here.
In addition, they can be used to compare and evaluate newly developed algorithms.

\subsection{Comparison of different selection strategies in \halrect}
\label{sec:model}
In this section, the impact and comparison of three different selection schemes: Lipschitz-based (using \Cref{def:potOptRectN}), improved aggressive (using \Cref{def:potOptRectAggr}), and two-step based Pareto (using \Cref{def:potOptRectPareto}), and four different strategies to obtain an aggregated objective function information over hyper-rectangles (controlled by \cref{eq:eq1,eq:eq2,eq:eq3,eq:eq4}) in the performance of \halrect{} is investigated.
In total, twelve different variations of \halrect{} are compared.

The results obtained on the entire set of $96$ \directlib{} test problems are summarized in \Cref{tab:results}.
The best results are highlighted in bold.
In the upper part of this table, the performance of \halrect{} is given using four strategies to obtain the aggregated information about the objective function ($\mathcal{F}^i_k$).
As can be seen, there is no single superior strategy.
The best average results are obtained with the first strategy based on a single midpoint value (see~\cref{eq:eq1}).
However, the overall lowest number of unsolved problems (7/96) was obtained with the second strategy.
It is based on the minimum value attained at all sampled points belonging to a certain hyper-rectangle (see~\cref{eq:eq2}).
Moreover, it performed significantly better on average than the other strategies on low-dimensional $(n \le 4)$ problems.
It can also be seen that the third strategy, based on the mean value (see~\cref{eq:eq3}), was the worst for practically all summarized  cases.
The best median results were obtained with the fourth strategy (see~\cref{eq:eq4}), which combines all three strategies, using the arithmetic mean of the estimates used in the first two strategies.

Our recent work~\cite{Stripinis2021b} showed that combining existing partition and selection schemes into \direct-type algorithms could lead to more efficient ones.
Motivated by this, we have created two different \halrect{} algorithmic versions, \halrect\texttt{-IA} and \halrect\texttt{-GL}, where the original partition strategy is used, but the selection scheme is changed.
Specifically, in \halrect\texttt{-IA}, the original Lipschitz lower bounds-based selection scheme (\Cref{def:potOptRectN}) is replaced with the \textit{improved aggressive selection} (\Cref{ssec:ImprovedAggressiveSelection}) using newly introduced \cref{eq:eq1,eq:eq2,eq:eq3,eq:eq4} for the information about the objective function.
Similarly, in \halrect\texttt{-GL}, the original \halrect{} selection scheme is replaced with a \textit{two-step-based (Global-Local) Pareto selection} (\Cref{ssec:ImprovedAggressiveSelection}).
Consequently, the results obtained on the same testbed are summarized in the middle and bottom parts of \Cref{tab:results}.
Comparing the influence of \cref{eq:eq1,eq:eq2,eq:eq3,eq:eq4} on the performance of three different versions of \halrect{}, we observe that for both \halrect\texttt{-IA} and \halrect\texttt{-GL}, the best results for practically all cases are obtained when \cref{eq:eq4} is used.
However, in the case of \halrect\texttt{-IA} and \halrect\texttt{-GL}, we no longer observe that \cref{eq:eq3} is always the worst, as was the case for \halrect.
Comparing \halrect, \halrect\texttt{-IA}, and \halrect\texttt{-GL}, we observe that the lowest number of unsolved problems ($2/96$) is obtained using \halrect\texttt{-GL}.
It was the best for almost all criteria, except for the median value, where \halrect{} with \cref{eq:eq4} performed the best.

\begin{table}[ht]
	\caption{Comparison of \halrect{} versions based on three different selection schemes: Lipschitz-based (used in \halrect), improved aggressive (used in \halrect\texttt{-IA}), and two-step based Pareto (used in \halrect\texttt{-GL}) and four different strategies to obtain an aggregated objective function information (controlled by \cref{eq:eq1,eq:eq2,eq:eq3,eq:eq4}). The performance measured as the number of function evaluations. The best results are marked in bold.}
	\footnotesize
	\resizebox{\textwidth}{!}{
		\begin{tabular}[tb]{@{\extracolsep{\fill}}clcrrrr}
			\toprule
			Alg. & Criteria & $\#$ of cases & \cref{eq:eq1} & \cref{eq:eq2} & \cref{eq:eq3} & \cref{eq:eq4}  \\
			\midrule
			\multirow{9}{*}{\rotatebox{90}{\halrect}} & $\#$ of failed problems	& $96$ & $8$ & $\mathbf{7}$ & $15$ & $12$ \\
			& Median results & $96$ & $1,419$ & $2,119$ & $3,581$ & $\mathbf{976}$ \\
			& Average results & $96$ & $\mathbf{127,562}$ & $143,909$ & $216,933$ & $142,403$ \\
			& Average ($n \leq 4$) & $51$ & $30,792$ & $\mathbf{7,248}$ & $31,456$ & $48,456$ \\
			& Average ($n > 4$) & $45$ & $\mathbf{237,918}$ & $298,952$ & $427,839$ & $249,953$ \\
			& Average (convex) & $30$ & $93,018$ & $167,616$ & $236,137$ & $\mathbf{83,835}$ \\
			& Average (non-convex) & $66$ & $143,263$ & $\mathbf{133,133}$ & $208,204$ & $169,024$ \\
			& Average (uni-modal) & $15$ & $62,774$ & $159,110$ & $221,099$ & $\mathbf{60,046}$ \\
			& Average (multi-modal) & $81$ & $142,513$ & $\mathbf{140,401}$ & $215,972$ & $161,408$ \\
			\midrule
			\multirow{9}{*}{\rotatebox{90}{\halrect\texttt{-IA}}} & $\#$ of failed problems & $96$ & $9$ & $9$ & $15$ & $\mathbf{5}$ \\
			& Median results & $96$ & $1,826$ & $1,880$ & $2,737$ & $\mathbf{1,581}$ \\
			& Average results & $96$ & $114,222$ & $124,552$ & $194,832$ & $\mathbf{62,874}$ \\
			& Average ($n \leq 4$) & $51$ & $43,203$  & $10,634$ & $13,223$ & $\mathbf{3,762}$ \\
			& Average ($n > 4$) & $45$ & $195,668$ & $253,895$ & $400,948$ & $\mathbf{129,952}$ \\
			& Average (convex) & $30$ & $105,963$ & $139,058$ & $234,394$ & $\mathbf{51,166}$ \\
			& Average (non-convex) & $66$ & $117,975$ & $117,958$ & $176,849$ & $\mathbf{68,197}$ \\
			& Average (uni-modal) & $15$ & $65,698$ & $127,693$ & $250,054$ & $\mathbf{62,332}$ \\
			& Average (multi-modal) & $81$ & $125,419$ & $123,827$ & $182,084$ & $\mathbf{63,000}$ \\
			\midrule
			\multirow{9}{*}{\rotatebox{90}{\halrect\texttt{-GL}}} & $\#$ of failed problems & $96$ & $5$ & $7$ & $5$ & $\mathbf{2}$ \\
			& Median results & $96$ & $\mathbf{1,404}$ & $2,564$ & $2,185$ & $1,520$ \\
			& Average results & $96$ & $64,275$ & $107,127$ & $65,271$ & $\mathbf{41,061}$ \\
			& Average ($n \leq 4$) & $51$ & $25,847$ & $10,831$ & $4,360$ & $\mathbf{3,301}$ \\
			& Average ($n > 4$) & $45$ & $108,401$ & $216,503$ & $134,399$ & $\mathbf{83,929}$ \\
			& Average (convex) & $30$ & $40,343$ & $79,374$ & $13,521$ & $\mathbf{7,055}$ \\
			& Average (non-convex) & $66$ & $75,153$ & $119,742$ & $88,793$ & $\mathbf{56,519}$ \\
			& Average (uni-modal) & $15$ & $24,714$ & $116,545$ & $65,538$ & $\mathbf{18,646}$ \\
			& Average (multi-modal) & $81$ & $73,405$ & $104,953$ & $65,209$ & $\mathbf{46,234}$ \\
			\bottomrule
	\end{tabular}}
	\label{tab:results}
\end{table}

Additionally, the operational characteristics~\cite{Grishagin1978,Strongin2000:book} using all $96$ test problems from \directlib{} are reported in \Cref{fig:perf-l1}.
Operational characteristics provide the proportion of test problems that can be solved within a given budget of function evaluations.
\Cref{fig:perf-l1} reveals that all \halrect{} algorithms based on three different selection schemes and four different strategies for $\mathcal{F}^i_k$ (\cref{eq:eq1,eq:eq2,eq:eq3,eq:eq4}) perform similarly when the budget given for the evaluations of objective functions is relatively small ($m \leq 1,000$).
Within this budget, all versions of \halrect{} could solve approximately half of the test problems.
However, as the number of function evaluations increases (as more complex problems are considered), the dominance of \cref{eq:eq4} based versions (especially \halrect\texttt{-GL}) begins to emerge.
At the same time, the worst results come from versions based on \cref{eq:eq3}.

\begin{figure}[ht]
	\resizebox{\textwidth}{!}{
		\begin{tikzpicture}
			\begin{axis}[
				legend pos=north west,
				title  = {Operational characteristics},
				xlabel = {Function evaluations},
				xmode=log,
				ymin=-0.02,ymax=1.02,
				ytick distance=0.1,
				xmode=log,
				xmin=50,
				xmax=1000000,
				xtick distance=10,
				ylabel = {Proportion of solved problems},
				legend style={font=\tiny,xshift=-0.5em},
				legend cell align={left},
				legend columns=1,
				height=0.6\textwidth,width=\textwidth,
				every axis plot/.append style={very thick},
				]
				\addplot[mark=*,black,mark options={scale=1.5, fill=princetonorange}, only marks,line width=0.75pt] coordinates {(0.1,0.1)} ;
				\label{p1}
				\addplot[mark=square*,black,mark options={scale=1.5, fill=brown}, only marks,line width=0.75pt] coordinates {(0.1,0.1)} ;
				\label{p2}
				\addplot[mark=diamond*,black,mark options={scale=1.5, fill=sienna}, only marks,line width=0.75pt] coordinates {(0.1,0.1)} ;
				\label{p3}
				\addplot[mark=triangle*,black,mark options={scale=1.5, fill=psychedelicpurple}, only marks,line width=0.75pt] coordinates {(0.1,0.1)} ;
				\label{p4}
				\node [draw,fill=white] at (rel axis cs: 0.85,0.25) {\shortstack[l]{
						\ref{p1} {\scriptsize \cref{eq:eq1}} \\
						\ref{p2} {\scriptsize \cref{eq:eq2}} \\
						\ref{p3} {\scriptsize \cref{eq:eq3}} \\
						\ref{p4} {\scriptsize \cref{eq:eq4}}}};
				
				\addplot[postaction={decoration={markings,mark=between positions 0 and 1 step 0.1 with {\node[circle,draw=black,fill=princetonorange,inner sep=1.5pt,solid] {};}},decorate,},princetonorange,line width=0.75pt] table[x=T,y=DDA] {data/Overallass.txt};
				\addplot[postaction={decoration={markings,mark=between positions 0 and 1 step 0.1 with {\node[mark=square,draw=black,fill=brown,inner sep=1.5pt,solid] {};}},decorate,},brown,line width=0.75pt] table[x=T,y=DRA] {data/Overallass.txt};
				\addplot[postaction={decoration={markings,mark=between positions 0 and 1 step 0.1 with 	{\node[diamond,draw=black,fill=sienna,inner sep=1.5pt,solid] {};}},decorate,},sienna,line width=0.75pt] table[x=T,y=BIA] {data/Overallass.txt};
				\addplot[postaction={decoration={markings,mark=between positions 0 and 1 step 0.1 with {\node[regular 	polygon,regular polygon sides=3,draw=black,fill=psychedelicpurple,inner sep=1pt,solid] {};}},decorate,},psychedelicpurple,line width=0.75pt] table[x=T,y=ADA] {data/Overallass.txt};
				
				\addplot[postaction={decoration={markings,mark=between positions 0 and 1 step 0.09 with {\node[circle,draw=black,fill=princetonorange,inner sep=1.5pt,solid] {};}},decorate,},princetonorange,line width=0.75pt, loosely dashed] table[x=T,y=DDO] {data/Overallass.txt};
				\addplot[postaction={decoration={markings,mark=between positions 0 and 1 step 0.09 with {\node[mark=square,draw=black,fill=brown,inner sep=1.5pt,solid] {};}},decorate,},brown,line width=0.75pt, loosely dashed] table[x=T,y=DRO] {data/Overallass.txt};
				\addplot[postaction={decoration={markings,mark=between positions 0 and 1 step 0.09 with 	{\node[diamond,draw=black,fill=sienna,inner sep=1.5pt,solid] {};}},decorate,},sienna,line width=0.75pt, loosely dashed] table[x=T,y=BIO] {data/Overallass.txt};
				\addplot[postaction={decoration={markings,mark=between positions 0 and 1 step 0.09 with {\node[regular 	polygon,regular polygon sides=3,draw=black,fill=psychedelicpurple,inner sep=1pt,solid] {};}},decorate,},psychedelicpurple,line width=0.75pt, loosely dashed] table[x=T,y=ADO] {data/Overallass.txt};
				
				\addplot[postaction={decoration={markings,mark=between positions 0 and 1 step 0.11 with {\node[circle,draw=black,fill=princetonorange,inner sep=1.5pt,solid] {};}},decorate,},princetonorange,line width=0.75pt, loosely dotted] table[x=T,y=DDG] {data/Overallass.txt};
				\addplot[postaction={decoration={markings,mark=between positions 0 and 1 step 0.11 with {\node[mark=square,draw=black,fill=brown,inner sep=1.5pt,solid] {};}},decorate,},brown,line width=0.75pt, loosely dotted] table[x=T,y=DRG] {data/Overallass.txt};
				\addplot[postaction={decoration={markings,mark=between positions 0 and 1 step 0.11 with 	{\node[diamond,draw=black,fill=sienna,inner sep=1.5pt,solid] {};}},decorate,},sienna,line width=0.75pt, loosely dotted] table[x=T,y=BIG] {data/Overallass.txt};
				\addplot[postaction={decoration={markings,mark=between positions 0 and 1 step 0.11 with {\node[regular 	polygon,regular polygon sides=3,draw=black,fill=psychedelicpurple,inner sep=1pt,solid] {};}},decorate,},psychedelicpurple,line width=0.75pt, loosely dotted] table[x=T,y=ADG] {data/Overallass.txt};
				
				\addplot[line width=0.75pt, black] coordinates {(0.1,0.1)} ;
				\label{ps1}
				\addplot[line width=0.75pt, loosely dashed, black] coordinates {(0.1,0.1)} ;
				\label{ps2}
				\addplot[line width=0.75pt, loosely dotted, black] coordinates {(0.1,0.1)} ;
				\label{ps3}
				\node [draw,fill=white] at (rel axis cs: 0.15,0.8) {\shortstack[l]{
						{\scriptsize Algorithm} \\
						\ref{ps2} {\scriptsize \halrect} \\
						\ref{ps1} {\scriptsize \halrect\texttt{-IA}} \\
						\ref{ps3} {\scriptsize \halrect\texttt{-GL}}}};
			\end{axis}
	\end{tikzpicture}}
	\caption{Operational characteristics of \halrect, \halrect\texttt{-IA}, \halrect\texttt{-GL} algorithms based on \cref{eq:eq1,eq:eq2,eq:eq3,eq:eq4} (used in the selection scheme) on the whole set of \directlib{} test problems.}
	\label{fig:perf-l1}
\end{figure}
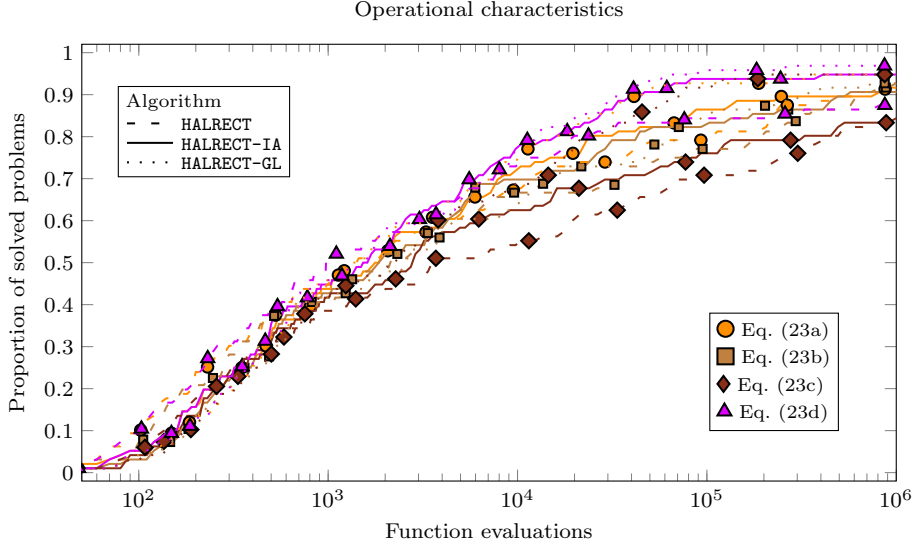

\subsection{Comparison of three \halrect{} variations vs. twelve recent \direct-type algorithms}\label{sec:comparison}

Based on the results presented in the previous section, the three most promising variations of \halrect{} algorithms (all based on \cref{eq:eq4}) are considered and  compared with twelve different \direct-type global optimization variations introduced in~\cite{Stripinis2021b}.
These twelve \direct-type algorithms have been created by newly combining three known selection schemes: i) \textbf{I}mproved \textbf{O}riginal (\textbf{IO}), ii) \textbf{I}mproved \textbf{A}ggressive (\textbf{IA}), and iii) two-step-based (\textbf{G}lobal-\textbf{L}ocal) Pareto (\textbf{GL}) (see \Cref{tab:selection-schemes}), and four partitioning techniques: i) Hyper-rectangular partitioning based on \textbf{N}-\textbf{D}imensional \textbf{T}risection and objective function evaluations at \textbf{C}enter points (\textbf{N-DTC}), ii)
Hyper-rectangular partitioning based on \textbf{1}-\textbf{D}imensional \textbf{T}risection and objective function evaluations at \textbf{C}enter points (\textbf{1-DTC}), iii) Hyper-rectangular partitioning based on \textbf{1}-\textbf{D}imensional \textbf{T}risection and objective function evaluations at two \textbf{D}iagonal \textbf{V}ertices (\textbf{1-DTDV}), and iv)  Hyper-rectangular partitioning based on \textbf{1}-\textbf{D}imensional \textbf{B}isection and objective function evaluations at two \textbf{D}iagonal \textbf{P}oints (\textbf{1-DBDP}) (see \Cref{tab:partitioning-strategies}).

\cref{tab:total} shows the summarized comparative results on the whole set of $96$ box-constrained test problems from \directlib.
In \Cref{tab:total}, each column corresponds to a \direct-type algorithm based on a different partitioning scheme.
Since each partition scheme was run on $96$ problems using $3$ different selection methods (rows of \Cref{tab:total}), each \direct-type algorithm based on a certain partition scheme was involved in solving $3 \times 96 = 288$ problems.
As previously, the best results are marked in bold.
We note that the original \halrect{} algorithm does not have the purpose of adapting the IO scheme designed to reduce the number of ``equivalent'' hyper-rectangles.
As described in \cref{ssec:reducing-POH}, the \halrect{} algorithm internally uses an innovative approach to deal with such cases.

Regardless of the chosen POH selection scheme (IO, IA, GL), the smallest number of unsolved problems was achieved using the \halrect{} partitioning scheme-based algorithms (\halrect, \halrect\texttt{-IA}, \halrect\texttt{-GL}).
Summing the results, \halrect{} partitioning scheme-based algorithms did not solve ($19/288$) of the test cases, while the second and third best partitioning schemes (1-DBDP and N-DTC) based algorithms did not solve ($28/288$) and ($29/282$) cases accordingly.
Naturally, a higher number of solved problems leads to better performance of the \halrect{} partitioning scheme-based algorithms.
Consequently, the three \halrect{} partitioning scheme-based algorithms required approximately $31 \%$ and $36 \%$ evaluations of fever functions in comparison to the other two algorithms based on the best partition schemes (1-DBDP and N-DTC) based algorithms.
The most notable difference in the \halrect{} partitioning scheme was observed when the IA selection scheme was used.
The \halrect\texttt{-IA} algorithm required approximately $57 \%$ and $61 \%$ compared to the two best algorithms (1-DBDP-IA and 1-DTC-IA).

\begin{sidewaystable}
	\caption{The number of function evaluations and the execution time (in seconds) of three \halrect{} versions based on \cref{eq:eq4} vs. twelve \direct-type algorithms (introduced in \cite{Stripinis2021b}) on \directlib{} test problems. The best results are marked in bold.}
	\resizebox{\textwidth}{!}{
		\begin{tabular}{lcrrrrrrrrrrr}
			\toprule
			\multirow{2}{*}{Criteria / Algorithms} & \multirow{2}{*}{$\#$ of cases} & \multicolumn{5}{c}{Function evaluations} && \multicolumn{5}{c}{Execution time (in seconds)} \\
			\cmidrule{3-7}\cmidrule{9-13}
			& & \halrect & N-DTC-IO & 1-DTC-IO & 1-DBDP-IO & 1-DTDV-IO && \halrect & N-DTC-IO & 1-DTC-IO & 1-DBDP-IO & 1-DTDV-IO \\
			\midrule
			$\#$ of failed problems & $96$ & $\mathbf{12}$ & $\mathbf{12}$ & $18$ & $\mathbf{12}$ & $21$ & & $-$ & $-$ & $-$ & $-$ & $-$ \\
			Average results & $96$ & $142,403$ & $\mathbf{142,277}$ & $211,463$ & $146,133$ & $227,455$ & & $ 652.97 $ & $ \mathbf{184.77} $ & $ 435.49 $ & $ 306.43 $ & $ 9,533.28 $ \\
			Average ($n \leq 4$) & $51$ & $48,456$ & $43,832$ & $42,633$ & $\mathbf{41,602}$ & $41,990$ & & $ 317.96 $ & $ 180.41 $ & $ 321.23 $ & $ \mathbf{169.54} $ & $ 1,730.39 $ \\
			Average ($n > 4$) & $45$ & $\mathbf{249,953}$ & $254,819$ & $403,749$ & $265,522$ & $438,574$ & & $ 1,039.72 $ & $ \mathbf{193.73} $ & $ 572.13 $ & $ 465.33 $ & $ 18,414.99 $ \\
			Average (convex) & $30$ & $83,835$ & $111,817$ & $170,675$ & $\mathbf{80,490}$ & $171,868$ & & $ 105.17 $ & $ 99.51 $ & $ 292.37 $ & $ \mathbf{89.43} $ & $ 7,204.96 $ \\
			Average (non-convex) & $66$ & $169,024$ & $\mathbf{156,122}$ & $230,004$ & $175,971$ & $252,722$ & & $ 901.97 $ & $ \mathbf{223.53} $ & $ 500.55 $ & $ 405.06 $ & $ 10,591.60 $ \\
			Average (uni-modal) & $15$ & $60,046$ & $60,100$ & $\mathbf{57,360}$ & $62,016$ & $111,547$ & & $ 55.23 $ & $ \mathbf{27.32} $ & $ 100.38 $ & $ 62.29 $ & $ 4, 800.28 $ \\
			Average (multi-modal) & $81$ & $161,408$ & $\mathbf{161,240}$ & $247,026$ & $165,545$ & $254,203$ & & $ 790.91 $ & $ \mathbf{221.11} $ & $ 512.83 $ & $ 362.77 $ & $ 10, 625.50 $ \\
			Median results & $96$ & $976$ & $\mathbf{771}$ & $1,198$ & $953$ & $847$ & & $ 0.68 $ & $ \mathbf{0.17} $ & $ 0.27 $ & $ 0.30 $ & $ 0.75 $ \\
			\midrule
			Criteria / Algorithms & $\#$ of cases & \halrect\texttt{-IA} & N-DTC-IA & 1-DTC-IA & 1-DBDP-IA & 1-DTDV-IA && \halrect\texttt{-IA} & N-DTC-IA & 1-DTC-IA & 1-DBDP-IA & 1-DTDV-IA \\
			\midrule
			$\#$ of failed problems & $96$ & $\mathbf{5}$ & $13$ & $13$ & $11$ & $18$ & & $-$ & $-$ & $-$ & $-$ & $-$ \\
			Average results & $96$ & $\mathbf{62,874}$ & $172,805$ & $160,691$ & $146,887$ & $202,694$ & & $ \mathbf{18.67} $ & $ 21.43 $ & $ 69.06 $ & $ 33.59 $ & $8,580.81$ \\
			Average ($n \leq 4$) & $51$ & $\mathbf{3,762}$ & $25,968$ & $23,638$ & $45,643$ & $9,785$ & & $ \mathbf{1.09} $ & $ 2.74 $ & $ 4.10 $ & $ 14.03 $ & $16.80$ \\
			Average ($n > 4$) & $45$ & $\mathbf{129,952}$ & $339,791$ & $316,541$ & $262,640$ & $421,539$ & & $ \mathbf{38.61} $ & $ 42.67 $ & $ 142.76 $ & $ 56.07 $ & $18,287.05$ \\
			Average (convex) & $30$ & $\mathbf{51,166}$ & $149,711$ & $126,030$ & $109,374$ & $153,594$ & & $ 19.06 $ & $ \mathbf{17.65} $ & $ 51.27 $ & $ 24.62 $ & $7,204.24$ \\
			Average (non-convex) & $66$ & $\mathbf{68,197}$ & $183,302$ & $176,446$ & $163,939$ & $225,012$ & & $ \mathbf{18.49} $ & $ 23.15 $ & $ 77.14 $ & $ 37.67 $ & $9,206.52$ \\
			Average (uni-modal) & $15$ & $\mathbf{62,332}$ & $108,068$ & $78,226$ & $73,957$ & $111,805$ & & $ 20.90 $ & $ \mathbf{9.61} $ & $ 30.50 $ & $ 15.28 $ & $4,800.43$ \\
			Average (multi-modal) & $81$ & $\mathbf{63,000}$ & $187,744$ & $179,722$ & $163,717$ & $223,668$ & & $ \mathbf{18.15} $ & $ 24.15 $ & $ 77.95 $ & $ 37.82 $ & $9,453.20$ \\
			Median results & $96$ & $1,581$ & $7,608$ & $\mathbf{1,287}$ & $2,108$ & $1,586$ & & $ 0.41 $ & $ 0.41 $ & $ \mathbf{0.21} $ & $ \mathbf{0.21} $ & $0.53$ \\
			\midrule
			Criteria / Algorithms & $\#$ of cases & \halrect\texttt{-GL} & N-DTC-GL & 1-DTC-GL & 1-DBDP-GL & 1-DTDV-GL && \halrect\texttt{-GL} & N-DTC-GL & 1-DTC-GL & 1-DBDP-GL & 1-DTDV-GL \\
			\midrule
			$\#$ of failed problems & $96$ & $\mathbf{2}$ & $4$ & $5$ & $5$ & $5$ & & $-$ & $-$ & $-$ & $-$ & $-$ \\
			Average results & $96$ & $\mathbf{41,061}$ & $71,488$ & $62,475$ & $65,442$ & $71,319$ & & $ 16.31 $ & $ \mathbf{9.10} $ & $ 38.12 $ & $ 21.25 $ & $ 3,907.79 $ \\
			Average ($n \leq 4$) & $51$ & $\mathbf{3,301}$ & $9,675$ & $7,073$ & $41,300$ & $5,772$ & & $ \mathbf{0.63} $ & $ 1.08 $ & $ 1.09 $ & $ 14.33 $ & $ 10.68 $ \\
			Average ($n > 4$) & $45$ & $\mathbf{83,929}$ & $141,753$ & $125,417$ & $93,714$ & $145,733$ & & $ 34.09 $ & $ \mathbf{18.21} $ & $ 80.11 $ & $ 29.41 $ & $ 8,324.75 $ \\
			Average (convex) & $30$ & $\mathbf{7,055}$ & $55,320$ & $45,520$ & $42,326$ & $8,950$ & & $ \mathbf{1.48} $ & $ 6.94 $ & $ 22.66 $ & $ 13.39 $ & $ 20.79 $ \\
			Average (non-convex) & $66$ & $\mathbf{56,519}$ & $78,837$ & $70,182$ & $75,949$ & $99,669$ & & $ 23.05 $ & $ \mathbf{10.08} $ & $ 45.14 $ & $ 24.82 $ & $ 5,674.61 $ \\
			Average (uni-modal) & $15$ & $18,646$ & $28,478$ & $\mathbf{12,624}$ & $23,300$ & $25,796$ & & $ 4.78 $ & $ \mathbf{2.09} $ & $ 2.28 $ & $ 2.94 $ & $ 4,310.88 $ \\
			Average (multi-modal) & $81$ & $\mathbf{46,234}$ & $81,183$ & $73,979$ & $75,398$ & $81,825$ & & $ 18.97 $ & $ \mathbf{10.71} $ & $ 46.39 $ & $ 25.48 $ & $ 3,814.77 $ \\
			Median results & $96$ & $1,520$ & $1,848$ & $960$ & $2,042$ & $\mathbf{775}$ & & $ 0.42 $ & $ 0.19 $ & $ \mathbf{0.17} $ & $ 0.23 $ & $ 0.39 $ \\
			\bottomrule
	\end{tabular}}
	\label{tab:total}
\end{sidewaystable}

In different subsets of test problems, again, on average, the \halrect{} partitioning scheme-based algorithms dominate the other schemes.
The dominance of the \halrect{} partitioning scheme can be seen especially on more complex, multi-modal, non-convex, and $n > 4$ test problems.
Solving multi-modal problems with \halrect{} partitioning scheme-based algorithms required approximately $33 \%$ and $37 \%$ evaluations of fever functions compared to the other two algorithms based on the best partition schemes (1-DBDP and N-DTC).
Among the different selection schemes, the highest level of dominance has been observed using the GL selection scheme.
\halrect\texttt{-GL} required approximately $62 \%$ and $65\%$ fever function evaluations compared to the other two best algorithms (1-DBDP-GL and N-DTC-GL).
On a subset of non-convex test cases, \halrect{} partitioning-based algorithms required approximately $38 \%$ and $55 \%$ fever function evaluations than the other two best algorithms (1-DBDP and N-DTC).
Once again, the \halrect\texttt{-GL} version has shown even more outstanding performance and outperformed the second-best algorithm 1-DBDP-GL by approximately $59 \%$ of fever function evaluations.

Apart from the convex test problems, the advantage of \halrect{} partitioning scheme-based algorithms is lesser on more straightforward test problems.
For the $n \leq 4$ optimization test instances, \halrect{} partitioning-based algorithms required approximately $4 \%$ and $24 \%$ fever function evaluations than the other two best 1-DTDV and 1-DTC partitioning-technique-based algorithms.
However, looking at individual algorithms, the most efficient \halrect\texttt{-GL} algorithm outperformed the second-best 1-DTDV-GL by requiring approximately $43 \%$ fewer objective function evaluations.
Similar trends persist for uni-modal test problems.

The median value is the only criterion for which \halrect{} partitioning-based algorithms were not dominant.
Based on the median values, 1-DTDV and 1-DTC algorithms appear to be the most effective and can solve at least half of the problems with the best performance.

Based on the number of function evaluation criteria among the selection schemes, the best overall performance was achieved using two-step-based Pareto selection (GL).
All partitioning strategies combined with the latter POH selection scheme solved the largest number of test problems and showed the best performance, especially on more complex ones.
The best combination, out of 15 tested, proved to be the \halrect\texttt{-GL} algorithm, the second-best 1-DTC-GL, and the third-best \halrect\texttt{-IA}.

Based on execution times, the fastest performing partitioning scheme is N-DTC.
On average, the N-DTC required approximately 41\% of fever execution times than the second fastest partition strategies (1-DBDP).
It is not surprising since the N-DTC partitioning scheme subdivides the hyper-rectangle through all the largest side lengths, resulting in more function evaluations but fewer expensive computations, like POH selection.
Overall, the proposed \halrect{} partitioning scheme ranks only fourth in speed.
Additional calculations hampered the performance of the algorithm.
However, due to their exceptional performance and a small number of failures, the \halrect\texttt{-GL} and \halrect\texttt{-IA} algorithms rank second and third in terms of running speed, behind only the N-DTC-GL algorithm.
Finally, the situation in favor of the \halrect{} algorithm will be even more promising when the values of the objective functions are more expensive.
In the case studied, the test functions are cheap.

Finally, operational characteristics in \Cref{fig:Nonlinear} show the behavior of all fifteen algorithms on all box-constrained test problems from \directlib.
When a given budget of function evaluations is low ($M_{\rm max} \leq 1,000$), all algorithms perform similarly regardless of the partitioning scheme.
All algorithms solved approximately $60 \%$ of the test problems within this relatively small budget.
However, when the maximum budget for function evaluations increased ($M_{\rm max} > 1,000$), the algorithms based on the \halrect{} partitioning strategy combined with IA and GL selection schemes showed the best performance.

\begin{figure}[ht]
	\resizebox{\textwidth}{!}{
		\begin{tikzpicture}
			\begin{groupplot}[
				group style={
					group size=3 by 1,
					x descriptions at=edge bottom,
					y descriptions at=edge left,
					horizontal sep=0pt,
				},
				height=0.65\textwidth,width=0.5\textwidth,
				xmode=log
				]
				\nextgroupplot[
				xmode=log,
				title  = {Selection scheme -- IO},
				ymin=0,ymax=1.025,
				ytick distance=0.1,
				xmin=10,xmax=1000000,
				xtick distance=10,
				xlabel = {Function evaluations},
				ylabel = {Proportion of problems solved},
				]
				\addplot[postaction={decoration={markings,mark=between positions 0 and 1 step 0.095 with {\node[circle,draw=sandstorm,fill=sandstorm,inner sep=1.5pt] {};}},decorate,},sandstorm,line width=0.75pt] table[x=T,y=DDO] {data/overall.txt};
				\addplot[postaction={decoration={markings,mark=between positions 0 and 1 step 0.1 with {\node[mark=square,draw=blue,fill=blue,inner sep=1.5pt] {};}},decorate,},blue,line width=0.75pt] table[x=T,y=DRO] {data/overall.txt};
				\addplot[postaction={decoration={markings,mark=between positions 0 and 1 step 0.105 with 	{\node[diamond,draw=onyx!50,fill=onyx!50,inner sep=1.5pt] {};}},decorate,},onyx!50,line width=0.75pt] table[x=T,y=BIO] {data/overall.txt};
				\addplot[postaction={decoration={markings,mark=between positions 0 and 1 step 0.11 with {\node[regular 	polygon,regular polygon sides=3,draw=red,fill=red,inner sep=1pt] {};}},decorate,},red,line width=0.75pt] table[x=T,y=ADO] {data/overall.txt};
				\addplot[postaction={decoration={markings,mark=between positions 0 and 1 step 0.09 with {\node[circle,draw=black,fill=black,inner sep=1.5pt,solid] {};}},decorate,},black,line width=0.75pt] table[x=T,y=HIO] {data/overall.txt};
				\node [draw,fill=white] at (rel axis cs: 0.7,0.21) {\shortstack[l]{
						{\scriptsize \textbf{Partitioning}} \\
						\ref{p10} {\tiny \halrect} \\
						\ref{p11} {\tiny N-DTC} \\
						\ref{p12} {\tiny 1-DTC} \\
						\ref{p13} {\tiny 1-DBDP} \\
						\ref{p14} {\tiny 1-DTDV}}};
				\nextgroupplot[
				xmode=log,
				title  = {Selection scheme -- IA},
				ymin=0,ymax=1.025,
				ytick distance=0.1,
				xmin=10,xmax=1000000,
				xtick distance=10,
				xlabel = {Function evaluations},
				]
				\addplot[mark=*,black,mark options={scale=1.5, fill=black},line width=0.75pt] coordinates {(0.1,0.1)} ;
				\label{p10}
				\addplot[mark=*,sandstorm,mark options={scale=1.5, fill=sandstorm},line width=0.75pt] coordinates {(0.1,0.1)} ;
				\label{p11}
				\addplot[mark=square*,blue,mark options={scale=1.5, fill=blue},line width=0.75pt] coordinates {(0.1,0.1)} ;
				\label{p12}
				\addplot[mark=diamond*,onyx!50,mark options={scale=1.5, fill=onyx!50},line width=0.75pt] coordinates {(0.1,0.1)} ;
				\label{p13}
				\addplot[mark=triangle*,red,mark options={scale=1.5, fill=red},line width=0.75pt] coordinates {(0.1,0.1)} ;
				\label{p14}
				\addplot[postaction={decoration={markings,mark=between positions 0 and 1 step 0.095 with {\node[circle,draw=sandstorm,fill=sandstorm,inner sep=1.5pt] {};}},decorate,},sandstorm,line width=0.75pt] table[x=T,y=DDA] {data/overall.txt};
				\addplot[postaction={decoration={markings,mark=between positions 0 and 1 step 0.1 with {\node[mark=square,draw=blue,fill=blue,inner sep=1.5pt] {};}},decorate,},blue,line width=0.75pt] table[x=T,y=DRA] {data/overall.txt};
				\addplot[postaction={decoration={markings,mark=between positions 0 and 1 step 0.105 with 	{\node[diamond,draw=onyx!50,fill=onyx!50,inner sep=1.5pt] {};}},decorate,},onyx!50,line width=0.75pt] table[x=T,y=BIA] {data/overall.txt};
				\addplot[postaction={decoration={markings,mark=between positions 0 and 1 step 0.11 with {\node[regular 	polygon,regular polygon sides=3,draw=red,fill=red,inner sep=1pt] {};}},decorate,},red,line width=0.75pt] table[x=T,y=ADA] {data/overall.txt};
				\addplot[postaction={decoration={markings,mark=between positions 0 and 1 step 0.09 with {\node[circle,draw=black,fill=black,inner sep=1.5pt,solid] {};}},decorate,},black,line width=0.75pt] table[x=T,y=HIA] {data/overall.txt};
				\node [draw,fill=white] at (rel axis cs: 0.7,0.21) {\shortstack[l]{
						{\scriptsize \textbf{Partitioning}} \\
						\ref{p10} {\tiny \halrect} \\
						\ref{p11} {\tiny N-DTC} \\
						\ref{p12} {\tiny 1-DTC} \\
						\ref{p13} {\tiny 1-DBDP} \\
						\ref{p14} {\tiny 1-DTDV}}};
				\nextgroupplot[
				legend pos=outer north east,
				xmode=log,
				legend style={font=\tiny},
				legend style={draw=none},
				legend style={row sep=0.5pt},
				title  = {Selection scheme -- GL},
				ymin=0,ymax=1.025,
				ytick distance=0.1,
				xmin=10,xmax=1000000,
				xtick distance=10,
				xlabel = {Function evaluations},
				]
				\addplot[postaction={decoration={markings,mark=between positions 0 and 1 step 0.095 with {\node[circle,draw=sandstorm,fill=sandstorm,inner sep=1.5pt] {};}},decorate,},sandstorm,line width=0.75pt] table[x=T,y=DDG] {data/overall.txt};
				\addplot[postaction={decoration={markings,mark=between positions 0 and 1 step 0.1 with {\node[mark=square,draw=blue,fill=blue,inner sep=1.5pt] {};}},decorate,},blue,line width=0.75pt] table[x=T,y=DRG] {data/overall.txt};
				\addplot[postaction={decoration={markings,mark=between positions 0 and 1 step 0.105 with 	{\node[diamond,draw=onyx!50,fill=onyx!50,inner sep=1.5pt] {};}},decorate,},onyx!50,line width=0.75pt] table[x=T,y=BIG] {data/overall.txt};
				\addplot[postaction={decoration={markings,mark=between positions 0 and 1 step 0.11 with {\node[regular 	polygon,regular polygon sides=3,draw=red,fill=red,inner sep=1pt] {};}},decorate,},red,line width=0.75pt] table[x=T,y=ADG] {data/overall.txt};
				\addplot[postaction={decoration={markings,mark=between positions 0 and 1 step 0.09 with {\node[circle,draw=black,fill=black,inner sep=1.5pt,solid] {};}},decorate,},black,line width=0.75pt] table[x=T,y=HGL] {data/overall.txt};
				\node [draw,fill=white] at (rel axis cs: 0.7,0.21) {\shortstack[l]{
						{\scriptsize \textbf{Partitioning}} \\
						\ref{p10} {\tiny \halrect} \\
						\ref{p11} {\tiny N-DTC} \\
						\ref{p12} {\tiny 1-DTC} \\
						\ref{p13} {\tiny 1-DBDP} \\
						\ref{p14} {\tiny 1-DTDV}}};
			\end{groupplot}
	\end{tikzpicture}}
	\caption{Operational characteristics of three new \halrect{} variations (based on \halrect{} partitioning scheme) vs. twelve \direct-type algorithms (introduced in \cite{Stripinis2021b}) on the whole set box-constrained test problems from \directlib.}
	\label{fig:Nonlinear}
\end{figure}
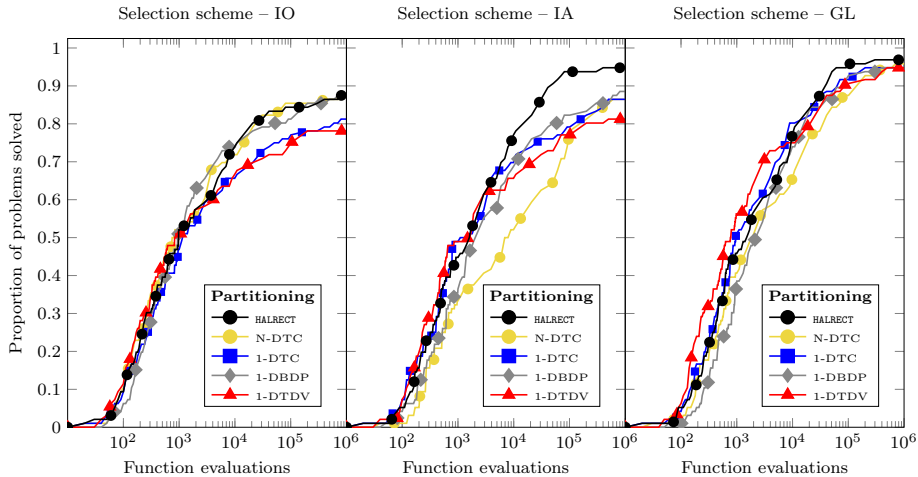

\subsection{Investigating the impact of the domain perturbation}\label{sec:perturb}

In investigating different partitioning techniques, one method may be lucky, because the partitioning approach in the initial steps naturally samples near the solution.
In such situations, the location of the solution may favor one partitioning strategy over another.
This section investigates the robustness of the partitioning approaches, especially the newly introduced \halrect{}, to slight perturbations of the domain.
This work extends our similar experiments described in \cite{Stripinis2021b}, where we identified test problems for which a particular partitioning scheme (regardless of the selection strategy) had a clear dominance, possibly due to the conveniently defined variable bounds.
Using the following rule, we perturbed the initial domain ($D = [\mathbf{a}, \mathbf{b}]$) for all box-constrained test problems from \directlib{}:
\begin{equation}\label{eq:pertrub}
	D^{\rm pert}_j = [\min{(a_j + \rho d_j, x_j^{\rm min})},  b_j + \rho d_j]_j, j = 1,...n,
\end{equation}
where $d_j = \mid b_j - a_j \mid $, and $\rho$ is a percentage of the shift.
The perturbed domain $D^{\rm pert}$ is obtained by shifting the original $D$ (given in \cref{tab:test}) by a $\rho$ percentage.
Since there is a risk that the solution may change when the domain is shifted, the calculation on the left-hand side of the bound checks that the shifted $(a_j + \rho d_j)$ coordinate is not greater than $x_j^{\rm min}$.
We used two different values ($\rho = 2.5 \%$ and $ \rho = 5 \%$) for the domain perturbation in the experimental study.

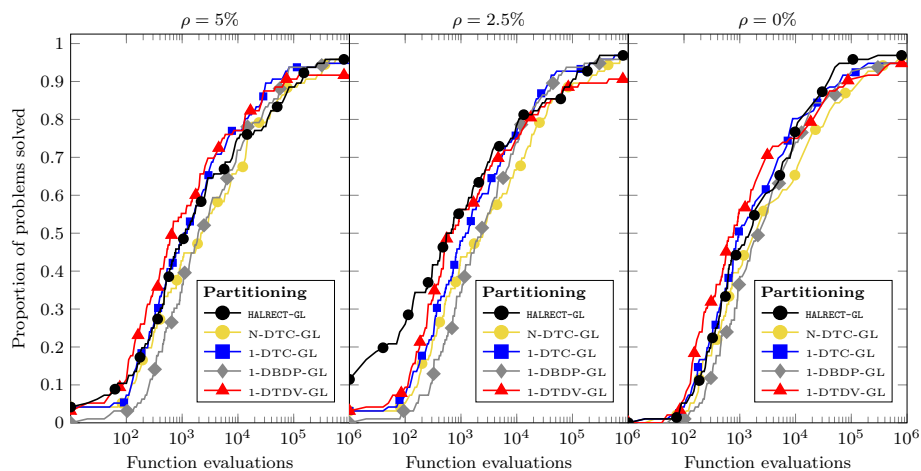
\begin{figure}[ht]
	\resizebox{\textwidth}{!}{
		\begin{tikzpicture}
			\begin{groupplot}[
				group style={
					group size=3 by 1,
					x descriptions at=edge bottom,
					y descriptions at=edge left,
					vertical sep=20pt,
					horizontal sep=0pt,
				},
				height=0.65\textwidth,width=0.5\textwidth,
				xmode=log
				]
				\nextgroupplot[
				legend pos=outer north east,
				xmode=log,
				legend style={font=\tiny},
				legend style={draw=none},
				legend style={row sep=0.5pt},
				title  = {$\rho = 5 \%$},
				ymin=0,ymax=1.025,
				ytick distance=0.1,
				xmin=10,xmax=1000000,
				xtick distance=10,
				title style={yshift=-0.75em},
				ylabel = {Proportion of problems solved},
				xlabel = {Function evaluations},
				]
				\addplot[postaction={decoration={markings,mark=between positions 0 and 1 step 0.095 with {\node[circle,draw=sandstorm,fill=sandstorm,inner sep=1.5pt] {};}},decorate,},sandstorm,line width=0.75pt] table[x=T,y=DVV] {data/Overallass1.txt};
				\addplot[postaction={decoration={markings,mark=between positions 0 and 1 step 0.1 with {\node[mark=square,draw=blue,fill=blue,inner sep=1.5pt] {};}},decorate,},blue,line width=0.75pt] table[x=T,y=GVV] {data/Overallass1.txt};
				\addplot[postaction={decoration={markings,mark=between positions 0 and 1 step 0.105 with 	{\node[diamond,draw=onyx!50,fill=onyx!50,inner sep=1.5pt] {};}},decorate,},onyx!50,line width=0.75pt] table[x=T,y=BVV] {data/Overallass1.txt};
				\addplot[postaction={decoration={markings,mark=between positions 0 and 1 step 0.11 with {\node[regular 	polygon,regular polygon sides=3,draw=red,fill=red,inner sep=1pt] {};}},decorate,},red,line width=0.75pt] table[x=T,y=AVV] {data/Overallass1.txt};
				\addplot[postaction={decoration={markings,mark=between positions 0 and 1 step 0.09 with {\node[circle,draw=black,fill=black,inner sep=1.5pt,solid] {};}},decorate,},black,line width=0.75pt] table[x=T,y=HVV] {data/Overallass1.txt};
				\node [draw,fill=white] at (rel axis cs: 0.7,0.21) {\shortstack[l]{
						{\scriptsize \textbf{Partitioning}} \\
						\ref{p10} {\tiny \halrect\texttt{-GL}} \\
						\ref{p11} {\tiny N-DTC-GL} \\
						\ref{p12} {\tiny 1-DTC-GL} \\
						\ref{p13} {\tiny 1-DBDP-GL} \\
						\ref{p14} {\tiny 1-DTDV-GL}}};
				\nextgroupplot[
				legend pos=outer north east,
				xmode=log,
				legend style={font=\tiny},
				legend style={draw=none},
				legend style={row sep=0.5pt},
				title  = {$\rho = 2.5 \%$},
				ymin=0,ymax=1.025,
				ytick distance=0.1,
				xmin=10,xmax=1000000,
				title style={yshift=-0.75em},
				xtick distance=10,
				xlabel = {Function evaluations},
				]
				\addplot[postaction={decoration={markings,mark=between positions 0 and 1 step 0.095 with {\node[circle,draw=sandstorm,fill=sandstorm,inner sep=1.5pt] {};}},decorate,},sandstorm,line width=0.75pt] table[x=T,y=DII] {data/Overallass1.txt};
				\addplot[postaction={decoration={markings,mark=between positions 0 and 1 step 0.1 with {\node[mark=square,draw=blue,fill=blue,inner sep=1.5pt] {};}},decorate,},blue,line width=0.75pt] table[x=T,y=GII] {data/Overallass1.txt};
				\addplot[postaction={decoration={markings,mark=between positions 0 and 1 step 0.105 with 	{\node[diamond,draw=onyx!50,fill=onyx!50,inner sep=1.5pt] {};}},decorate,},onyx!50,line width=0.75pt] table[x=T,y=BII] {data/Overallass1.txt};
				\addplot[postaction={decoration={markings,mark=between positions 0 and 1 step 0.11 with {\node[regular 	polygon,regular polygon sides=3,draw=red,fill=red,inner sep=1pt] {};}},decorate,},red,line width=0.75pt] table[x=T,y=AII] {data/Overallass1.txt};
				\addplot[postaction={decoration={markings,mark=between positions 0 and 1 step 0.09 with {\node[circle,draw=black,fill=black,inner sep=1.5pt,solid] {};}},decorate,},black,line width=0.75pt] table[x=T,y=HII] {data/Overallass1.txt};
				\node [draw,fill=white] at (rel axis cs: 0.7,0.21) {\shortstack[l]{
						{\scriptsize \textbf{Partitioning}} \\
						\ref{p10} {\tiny \halrect\texttt{-GL}} \\
						\ref{p11} {\tiny N-DTC-GL} \\
						\ref{p12} {\tiny 1-DTC-GL} \\
						\ref{p13} {\tiny 1-DBDP-GL} \\
						\ref{p14} {\tiny 1-DTDV-GL}}};
				\nextgroupplot[
				legend pos=outer north east,
				xmode=log,
				legend style={font=\tiny},
				legend style={draw=none},
				legend style={row sep=0.5pt},
				title  = {$\rho = 0 \%$},
				ymin=0,ymax=1.025,
				ytick distance=0.1,
				xmin=10,xmax=1000000,
				title style={yshift=-0.75em},
				xtick distance=10,
				xlabel = {Function evaluations},
				]
				\addplot[postaction={decoration={markings,mark=between positions 0 and 1 step 0.095 with {\node[circle,draw=sandstorm,fill=sandstorm,inner sep=1.5pt] {};}},decorate,},sandstorm,line width=0.75pt] table[x=T,y=DDG] {data/overall.txt};
				\addplot[postaction={decoration={markings,mark=between positions 0 and 1 step 0.1 with {\node[mark=square,draw=blue,fill=blue,inner sep=1.5pt] {};}},decorate,},blue,line width=0.75pt] table[x=T,y=DRG] {data/overall.txt};
				\addplot[postaction={decoration={markings,mark=between positions 0 and 1 step 0.105 with 	{\node[diamond,draw=onyx!50,fill=onyx!50,inner sep=1.5pt] {};}},decorate,},onyx!50,line width=0.75pt] table[x=T,y=BIG] {data/overall.txt};
				\addplot[postaction={decoration={markings,mark=between positions 0 and 1 step 0.11 with {\node[regular 	polygon,regular polygon sides=3,draw=red,fill=red,inner sep=1pt] {};}},decorate,},red,line width=0.75pt] table[x=T,y=ADG] {data/overall.txt};
				\addplot[postaction={decoration={markings,mark=between positions 0 and 1 step 0.09 with {\node[circle,draw=black,fill=black,inner sep=1.5pt,solid] {};}},decorate,},black,line width=0.75pt] table[x=T,y=HGL] {data/overall.txt};
				\node [draw,fill=white] at (rel axis cs: 0.7,0.21) {\shortstack[l]{
						{\scriptsize \textbf{Partitioning}} \\
						\ref{p10} {\tiny \halrect\texttt{-GL}} \\
						\ref{p11} {\tiny N-DTC-GL} \\
						\ref{p12} {\tiny 1-DTC-GL} \\
						\ref{p13} {\tiny 1-DBDP-GL} \\
						\ref{p14} {\tiny 1-DTDV-GL}}};
			\end{groupplot}
	\end{tikzpicture}}
	\caption{Operational characteristics of five \direct-type algorithms based on different partitioning schemes combined with GL selection on the whole set of box-constrained perturbed problems from \directlib.}
	\label{fig:Nonlinears}
\end{figure}

The experimental results obtained of five \direct-type algorithms based on different partitioning schemes combined with GL selection are illustrated in \Cref{fig:Nonlinears}.
The efficiency of the \halrect\texttt{-GL} algorithm, when a given budget of function evaluations is low ($M_{\rm max} \leq 200$), has increased.
However, when ($ 30,000 < M_{\rm max} < 200,000$), the performance of the \halrect\texttt{-GL} algorithm slightly worsened compared to the initial results.
However, the algorithm's efficiency remains the same with a large objective function evaluation budget ($M_{\rm max} > 200,000$).
Algorithms based on other partitioning schemes behave similarly for different values of $\rho$. 
The percentage of solved test problems remains similar for almost all algorithms.
A more noticeable difference is using a 1-DTDV-GL algorithm.
When $\rho = 2.5 \%$, the percentage of solved problems is reduced by $\sim 5\%$, and when $\rho = 5 \%$, it is reduced by $\sim 4 \%$.



\section{Conclusion}\label{sec:conclusions}
This paper introduces a new \direct-type algorithm (\halrect) for box-constrained global optimization problems.
A new deterministic approach combines halving (bisection) with a new multi-point sampling scheme in contrast to trisection and midpoint sampling used in most existing \direct-type algorithms.
Three selection schemes and four strategies are introduced to calculate the aggregated information of the objective function used in the selection of the candidate.
In this way, twelve variations of the \halrect{} algorithm are introduced and experimentally compared.
Three of the most promising versions were selected and compared versus twelve recent \direct-type algorithms.
The extensive experimental results revealed that the new algorithms based on \halrect{} partitioning schemes give results comparable and often superior to these 12 \direct-type algorithms.
Further investigation has shown that small perturbations in the domain $D$ of the test problems can help the \halrect{} algorithm to represent better and select POHs, which can significantly improve performance efficiency.

\section*{Code availability}
All implemented versions of the \halrect{} algorithm are available at the GitHub repository: \url{https://github.com/blockchain-group/DIRECTGO} and can be used under the MIT license. We welcome contributions and corrections to this work.

\section*{Data statement}
\texttt{DIRECTGOLib} - \direct{} \textbf{G}lobal \textbf{O}ptimization test problems \textbf{Lib}rary is designed as a continuously-growing open-source GitHub repository to which anyone can easily contribute.
The exact data underlying this article from \directlib{} can be accessed either on GitHub or at Zenodo:
\begin{itemize}
	\item GitHub: \url{https://github.com/blockchain-group/DIRECTGOLib/tree/v1.1},
	\item Zenodo: \url{https://doi.org/10.5281/zenodo.6491951},
\end{itemize}
and used under the MIT license.
We welcome contributions and corrections to this work.

\appendix
\section{\directlib{} library}
\label{apendixas}

A summary of all used box-constrained optimization problems from \directlib\cite{DIRECTGOLibv11,DIRECTGOLibv1.1} and their properties are given in \Cref{tab:test}.
\cite{Stripinis2021b}
Test problems with the ${\alpha}$ symbol indicate that the non-default domain $D$ was used for the test problem.
The modified domain $D$ was taken from the \cite{Stripinis2021b} study for all the ${\alpha}$ symbol-marked test problems.
Here, the main features are reported: problem number (\#), name of the problem, source, dimensionality ($ n $), optimization domain ($ D $), problem type, and the known minimum ($ f^* $).
Some of these test problems have several variants, e.g., \textit{Bohachevsky}, \textit{Shekel}, and some of them, like \textit{Alpine}, \textit{Csendes}, \textit{Griewank}, etc., can be tested for varying dimensionality.

\begin{table}
	\caption{Key characteristics of the \directlib\cite{DIRECTGOLibv11,DIRECTGOLibv1.1} test problems for box-constrained global optimization}
	\resizebox{\textwidth}{!}{
		\begin{tabular}[tb]{@{\extracolsep{\fill}}rlrrrrrr}
			\toprule
			\# & Name & Source & $n$  &  $ D $ & Type & No. of minima & $ f^* $ \\
			\midrule
			$1,2,3$  & \textit{Ackley}$^{\alpha}$ 		& \cite{Hedar2005,Derek2013} 		& $2,5,10$ 	& $[-18, 47]^n$ 								& non-convex & multi-modal & $0.0000$			\\
			$4,5,6$  & \textit{Alpine}$^{\alpha}$ 		& \cite{Gavana2021} 		   		& $2,5,10$ 	& $[\sqrt[i]{2}, 8 + \sqrt[i]{2}]^n$ 			& non-convex & multi-modal & $-2.8081^n$ 		\\
			$7$ 	   & \textit{Beale}      	 			& \cite{Hedar2005,Derek2013} 		& $2$  		& $[-4.5, 4.5]^n$   	   						& non-convex & multi-modal & $0.0000$   		\\
			$8$ 	   & \textit{Bohachevsky$1^{\alpha}$}  	& \cite{Hedar2005,Derek2013} 		& $2$  		& $[-55, 145]^n$   	   							& convex 	 & uni-modal   & $0.0000$   		\\
			$9$ 	   & \textit{Bohachevsky$2^{\alpha}$} 	& \cite{Hedar2005,Derek2013} 		& $2$  		& $[-55, 145]^n$   	   							& non-convex & multi-modal & $0.0000$   		\\
			$10$ 	   & \textit{Bohachevsky$3^{\alpha}$}  	& \cite{Hedar2005,Derek2013} 		& $2$  		& $[-55, 145]^n$   	   							& non-convex & multi-modal & $0.0000$   		\\
			$11$ 	   & \textit{Booth}          			& \cite{Hedar2005,Derek2013} 		& $2$  		& $[-10, 10]^n$  	 	 						& convex 	 & uni-modal   & $0.0000$			\\
			$12$ 	   & \textit{Branin}          			& \cite{Hedar2005,Dixon1978} 		& $2$  		& $[-5, 10] \times [10,15]$ 					& non-convex & multi-modal & $0.3978$			\\
			$13$ 	   & \textit{Bukin6}           			& \cite{Derek2013} 		  	 		& $2$  		& $[-15, 5] \times [-3,3]$  					& convex 	 & multi-modal & $0.0000$   		\\
			$14$ 	   & \textit{Colville}        			& \cite{Hedar2005,Derek2013} 		& $4$  		& $[-10, 10]^n$           						& non-convex & multi-modal & $0.0000$   		\\
			$15$ 	   & \textit{Cross\_in\_Tray}   		& \cite{Derek2013} 		  	 		& $2$    	& $[0, 10]^n$           						& non-convex & multi-modal & $-2.0626$   		\\
			$16$ 	   & \textit{Crosslegtable}   			& \cite{Gavana2021} 		  	 	& $2$    	& $[-10, 15]^n$           						& non-convex & multi-modal & $-1.000$   		\\
			$17,18,19$ & \textit{Csendes}$^{\alpha}$   	& \cite{Gavana2021} 		  	 	& $2,5,10$  & $[-10, 25]^n$          	& convex 	 & multi-modal & $0.0000$   		\\
			$20$ 	   & \textit{Damavandi}   				& \cite{Gavana2021} 		  	 	& $2$    	& $[0, 14]^n$           						& non-convex & multi-modal & $0.0000$   		\\
			$21,22,23$ & \textit{Deb$01$}$^{\alpha}$    	& \cite{Gavana2021} 				& $2,5,10$  & $[-0.55, 1.45]^n$         					& non-convex & multi-modal & $-1.0000$   		\\
			$24,25,26$ & \textit{Deb$02$}$^{\alpha}$   	& \cite{Gavana2021} 				& $2,5,10$  & $[0.225, 1.225]^n$ 							& non-convex & multi-modal & $-1.0000$   		\\
			$27,28,29$ & \textit{Dixon\_and\_Price}   	& \cite{Hedar2005,Derek2013} 		& $2,5,10$  & $[-10, 10]^n$         						& convex 	 & multi-modal & $0.0000$   		\\
			$30$ 	   & \textit{Drop\_wave}$^{\alpha}$   	& \cite{Derek2013} 		  	 		& $2$    	& $[-4, 6]^n$      								& non-convex & multi-modal & $-1.0000$   		\\
			$31$ 	   & \textit{Easom}$^{\alpha}$     		& \cite{Hedar2005,Derek2013} 		& $2$    	& $\left[\dfrac{-100}{i+1}, 100i \right]^n$ 	& non-convex & multi-modal & $-1.0000$			\\
			$32$ 	   & \textit{Eggholder}       			& \cite{Derek2013} 		  	 		& $2$   	& $[-512, 512]^n$         						& non-convex & multi-modal & $-959.6406$	 	\\
			$33$ 	   & \textit{Goldstein\_and\_Price}$^{\alpha}$ & \cite{Hedar2005,Dixon1978} & $2$   	& $[-1.1, 2.9]^n$          						& non-convex & multi-modal & $3.0000$ 			\\
			$34,35,36$ & \textit{Griewank}$^{\alpha}$     & \cite{Hedar2005,Derek2013} & $2,5,10$ & $\left[-\sqrt{600i}, \dfrac{600}{\sqrt{i}} \right]^n$ & non-convex & multi-modal & $0.0000$   		\\
			$37$ 	   & \textit{Hartman$3$}       			& \cite{Hedar2005,Derek2013} 		& $3$    	& $[0, 1]^n$              						& non-convex & multi-modal & $-3.8627$			\\
			$38$ 	   & \textit{Hartman$6$}        		& \cite{Hedar2005,Derek2013} 		& $6$    	& $[0, 1]^n$            						& non-convex & multi-modal & $-3.3223$			\\
			$39$ 	   & \textit{Holder\_Table}      		& \cite{Derek2013} 			 		& $2$    	& $[-10, 10]^n$           						& non-convex & multi-modal & $-19.2085$			\\
			$40$ 	   & \textit{Hump}              		& \cite{Hedar2005,Derek2013} 		& $2$    	& $[-5, 5]^n$           						& non-convex & multi-modal & $-1.0316$			\\
			$41$ 	   & \textit{Langermann}        		& \cite{Derek2013} 			 		& $2$    	& $[0, 10]^n$           						& non-convex & multi-modal & $-4.1558$			\\
			$42,43,44$ & \textit{Levy}              		& \cite{Hedar2005,Derek2013} 		& $2,5,10$  & $[-10, 10]^n$             						& non-convex & multi-modal & $0.0000$   		\\
			$45$ 	   & \textit{Matyas}$^{\alpha}$         & \cite{Hedar2005,Derek2013} 		& $2$    	& $[-5.5, 14.5]^n$  	 						& convex 	 & uni-modal   & $0.0000$ 			\\
			$46$ 	   & \textit{McCormick}         		& \cite{Derek2013} 	 		 		& $2$    	& $[-1.5, 4] \times [-3,4]$ 					& convex 	 & multi-modal & $-1.9132$ 			\\
			$47$ 	   & \textit{Michalewicz}       		& \cite{Hedar2005,Derek2013} 		& $2$    	& $[0, \pi]^n$            						& non-convex & multi-modal & $-1.8013$	 		\\
			$48$ 	   & \textit{Michalewicz}       		& \cite{Hedar2005,Derek2013} 		& $5$    	& $[0, \pi]^n$          						& non-convex & multi-modal & $-4.6876$			\\
			$49$     & \textit{Michalewicz}       		& \cite{Hedar2005,Derek2013} 		& $10$   	& $[0, \pi]^n$         	 						& non-convex & multi-modal & $-9.6601$			\\
			$50$ 	   & \textit{Perm$4$} 						& \cite{Hedar2005,Derek2013} 		& $4$    	& $[-i, i]^n$  									& non-convex & multi-modal & $0.0000$   		\\
			$51,52,53$ & \textit{Pinter}$^{\alpha}$     	& \cite{Gavana2021} 		  	 	& $2,5,10$  & $[-5.5, 14.5]^n$      						& non-convex & multi-modal & $0.0000$   		\\
			$54$ 	   & \textit{Powell}       				& \cite{Hedar2005,Derek2013} 		& $4$    	& $[-4, 5]^n$             						& convex 	 & multi-modal & $0.0000$   		\\
			$55$ 	   & \textit{Power\_Sum}$^{\alpha}$     & \cite{Hedar2005,Derek2013} 		& $4$    	& $[1, 4 + \sqrt[i]{2}]^n$            			& convex 	 & multi-modal   & $0.0000$   		\\
			$56,57,58$ & \textit{Qing}     	    		& \cite{Gavana2021} 		  	 	& $2,5,10$  & $[-500, 500]^n$      	 						& non-convex & multi-modal & $0.0000$   		\\
			$59,60,61$ & \textit{Rastrigin}$^{\alpha}$    & \cite{Hedar2005,Derek2013} 		& $2,5,10$  & $[-5\sqrt[i]{2}, 7+\sqrt[i]{2}]^n$       		& non-convex & multi-modal & $0.0000$   		\\
			$62,63,64$ & \textit{Rosenbrock}$^{\alpha}$   & \cite{Hedar2005,Dixon1978}  		& $2,5,10$  & $\left[-\dfrac{5}{\sqrt{i}}, 10\sqrt{i} \right]^n$ & non-convex & uni-modal & $0.0000$   		\\
			$65,66,67$ & \textit{Rotated\_H\_Ellip}$^{\alpha}$ & \cite{Derek2013} 	 		& $2,5,10$  & $[-35, 96]^n$   								& convex 	 & uni-modal   & $0.0000$   		\\
			$68,69,70$ & \textit{Schwefel}$^{\alpha}$ & \cite{Hedar2005,Derek2013} & $2,5,10$ & $\left[-500 + \dfrac{100}{\sqrt{i}}, 500 - \dfrac{40}{\sqrt{i}} \right]^n$ & non-convex & multi-modal & $0.0000$ \\
			$71$ 	   & \textit{Shekel$5$}  				& \cite{Hedar2005,Derek2013} 		& $4$    	& $[0, 10]^n$             						& non-convex & multi-modal & $-10.1531$			\\
			$72$ 	   & \textit{Shekel$7$}   				& \cite{Hedar2005,Derek2013}	 	& $4$    	& $[0, 10]^n$             						& non-convex & multi-modal & $-10.4029$			\\
			$73$ 	   & \textit{Shekel$10$}  				& \cite{Hedar2005,Derek2013} 		& $4$    	& $[0, 10]^n$             						& non-convex & multi-modal & $-10.5364$	 		\\
			$74$ 	   & \textit{Shubert}           		& \cite{Hedar2005,Derek2013} 		& $2$    	& $[-10, 10]^n$           						& non-convex & multi-modal & $-186.7309$		\\
			$75,76,77$ & \textit{Sphere}$^{\alpha}$       & \cite{Hedar2005,Derek2013} 		& $2,5,10$  & $[-2.75, 7.25]^n$      						& convex 	 & uni-modal   & $0.0000$   		\\
			$78,79,80$ & \textit{Styblinski\_Tang}$^{\alpha}$ & \cite{Clerc1999} 			 	& $2,5,10$  & $[-5, 5 + \sqrt[i]{3}]^n$       	    					& non-convex & multi-modal & $-39.1661n$		\\
			$81,82,83$ & \textit{Sum\_of\_Powers}$^{\alpha}$ & \cite{Derek2013}		 		& $2,5,10$  & $[-0.55, 1.45]^n$      						& convex 	 & uni-modal   & $0.0000$   		\\
			$84,85,86$ & \textit{Sum\_Square}$^{\alpha}$  & \cite{Clerc1999} 			 		& $2,5,10$  & $[-5.5, 14.5]^n$        						& convex 	 & uni-modal   & $0.0000$			\\
			$87$ 	   & \textit{Trefethen}           		& \cite{Gavana2021} 				& $2$    	& $[-2, 2]^n$           						& non-convex & multi-modal & $-3.3068$			\\
			$88,89,90$ & \textit{Trid}              		& \cite{Hedar2005,Derek2013} 		& $2,5,10$  & $[-100, 100]^n$           					& convex 	 & multi-modal & $\vartheta$		\\
			$91,92,93$ & \textit{Vincent}           		& \cite{Clerc1999} 					& $2,5,10$  & $[0.25, 10]^n$          						& non-convex & multi-modal & $-n$   			\\
			$94,95,96$ & \textit{Zakharov}$^{\alpha}$     & \cite{Hedar2005,Derek2013} 		& $2,5,10$  & $[-1.625, 13.375]^n$      					& convex 	 & multi-modal & $0.0000$    		\\
			\midrule
			\multicolumn{8}{l}{$\vartheta$ -- $-\frac{1}{6}n^3 - \frac{1}{2}n^2 +  \frac{2}{3}n$}\\
			\multicolumn{8}{l}{${\alpha}$ -- domain $D$ was taken from \cite{Stripinis2021b}} \\
			\multicolumn{8}{l}{${i = 1, ..., n}$} \\
			
			\bottomrule
	\end{tabular}}
	\label{tab:test}
\end{table}

\bibliographystyle{spmpsci} 
\bibliography{bibliography} 
\end{document}